\pdfoutput=1
\documentclass[11pt,letterpaper]{article}%
\usepackage[utf8]{inputenc}
\usepackage[rflt]{floatflt}
\usepackage{fullpage}
\usepackage[T1]{fontenc}
\usepackage{graphicx}
\usepackage{epsfig}
\usepackage{color}
\usepackage{enumerate}
\usepackage{verbatim}
\usepackage{amsthm,amssymb}
\usepackage{amsmath}
\usepackage{thmtools,thm-restate}
\usepackage{wrapfig}
\usepackage{xspace}
\usepackage{authblk}
\usepackage{amsfonts}
\usepackage{bbm}
\usepackage{algorithm, caption, subcaption}
\usepackage[noend]{algpseudocode}
\usepackage{ulem}
\usepackage{todonotes}

\definecolor{darkgreen}{rgb}{0,0.5,0}
\usepackage{hyperref}
\hypersetup{
    unicode=false,          %
    colorlinks=true,        %
    linkcolor=blue,          %
    citecolor=darkgreen,        %
    filecolor=magenta,      %
    urlcolor=cyan           %
}

\usepackage[disableredefinitions,roman]{complexity}
\usepackage[framemethod=tikz]{mdframed}
\global\mdfdefinestyle{myframe}{leftmargin=.75in,rightmargin=.75in,linecolor=black,linewidth=1.5pt,innertopmargin=10pt,innerbottommargin=10pt} 
\usepackage{mdframed}
\usepackage{framed}
\usepackage{nicefrac}

\date{}

\newtheorem{theorem}{Theorem}[section]
\newtheorem{lemma}[theorem]{Lemma}
\newtheorem{claim}[theorem]{Claim}
\newtheorem{subclaim}[theorem]{Subclaim}
\newtheorem{fact}[theorem]{Fact}

\newtheorem{remark}[theorem]{Remark}

\newtheorem{corollary}[theorem]{Corollary}

\theoremstyle{definition}
\newtheorem{definition}[theorem]{Definition}

\usepackage[capitalize, nameinlink]{cleveref}
\crefname{theorem}{Theorem}{Theorems}
\Crefname{lemma}{Lemma}{Lemmas}
\Crefname{alg}{Algorithm}{Algorithms}
\Crefname{claim}{Claim}{Claims}
\Crefname{subclaim}{Subclaim}{Subclaims}
\Crefname{infclaim}{Claim}{Claims}
\Crefname{observation}{Observation}{Observations}
\Crefname{invariant}{Invariant}{Invariants}
\Crefname{algorithm}{Algorithm}{Algorithms}
\Crefname{fact}{Fact}{Facts}

\def\polylog{\operatorname{polylog}}

\newcommand{\eps}{\varepsilon}

\newcommand{\cnt}{\rho}
\newcommand{\muh}{\hat{\mu}}

\newcommand{\Ex}{\mathbb{E}}

\renewcommand{\Pr}{\operatorname{{Pr}}}
\DeclareMathOperator*{\argmax}{arg\,max}
\DeclareMathOperator*{\argmin}{arg\,min}

\allowdisplaybreaks[1]

\newcommand{\tstar}{\textbf{T*}}
\newcommand{\tprime}{\textbf{T'}}
\renewcommand{\R}{\mathbb{R}}
\newcommand{\tmin}{\tau_{\textrm{min}}}

\newcommand{\dtv}{\mathsf{d_{TV}}}

\newcommand{\dhsq}{\mathsf{d_{h}^2}}
\newcommand{\dchi}{\mathsf{d_{\chi^2}}}

\def\tri{\operatorname{Tri}}
\def\mtri{\operatorname{Mod-Tri}}
\def\step{\operatorname{Step}}
\def\mstep{\operatorname{Mod-Step}}
\def\rstep{\operatorname{Rand-Step}}
\def\rmstep{\operatorname{Rand-Mod-Step}}
\def\unif{\operatorname{Unif}}
\def\bern{\operatorname{Bern}}

\def\cgam{C_{\gamma}}
\def\cdist{C_{\operatorname{dist}}}
\def\cround{C_{\operatorname{round}}}
\def\gsmall{\gamma_{\textrm{small}}}
\def\glarge{\gamma_{\textrm{large}}}

\def\est{\hat{\theta}}

\newcommand\numberthis{\addtocounter{equation}{1}\tag{\theequation}}

\newcommand*\diff{\mathop{}\!\mathrm{d}}

\newcommand{\rtag}[1]{\text{\quad $\left(\text{#1}\right)$}}

\title{Attainability of Two-Point Testing Rates for Finite-Sample\\ Location Estimation}
\author[]{Spencer Compton}
\author[]{Gregory Valiant}
\affil[]{Stanford University \authorcr
  \{\tt comptons, valiant\}@stanford.edu}

\begin{document}

\maketitle
\begin{abstract}
Le Cam's two-point testing method yields perhaps the simplest lower bound for estimating the mean of a distribution: roughly, if it is impossible to well-distinguish a distribution centered at $\mu$ from the same distribution centered at $\mu+\Delta$, then it is impossible to estimate the mean by better than $\Delta/2$. It is setting-dependent whether or not a nearly matching upper bound is attainable. We study the conditions under which the two-point testing lower bound can be attained for univariate mean estimation; both in the setting of \textit{location estimation} (where the distribution is known up to translation) and \textit{adaptive location estimation} (unknown distribution). Roughly, we will say an estimate nearly attains the two-point testing lower bound if it incurs error that is at most polylogarithmically larger than the \textit{Hellinger modulus of continuity} for $\tilde{\Omega}(n)$ samples.

Adaptive location estimation is particularly interesting, as some distributions admit much better guarantees than sub-Gaussian rates (e.g. $\unif(\mu-1,\mu+1)$ permits error $\Theta(\frac{1}{n})$, while the sub-Gaussian rate is $\Theta(\frac{1}{\sqrt{n}})$), yet it is not obvious whether these rates may be adaptively attained by one unified approach. Our main result designs an algorithm that nearly attains the two-point testing rate for mixtures of symmetric, log-concave distributions with a common mean. 
Moreover, this algorithm runs in near-linear time and is parameter-free. In contrast, we show the two-point testing rate is not nearly attainable even for symmetric, unimodal distributions.

We complement this with results for location estimation, showing the two-point testing rate is nearly attainable for unimodal distributions, but unattainable for symmetric distributions.

\end{abstract}

\section{Introduction}

Estimating the mean of a distribution $D$ from $n$ samples is a well-studied task, both in the setting of \textit{location estimation} (where $D$ is known up to translation) and \textit{adaptive location estimation} (where $D$ is unknown). While in some settings the typical estimators such as the sample mean/median are near-optimal (e.g. i.i.d. samples from a Gaussian), in many others there are approaches that may perform much better. A classical example is how for the uniform distribution, $\unif(\mu - 1, \mu + 1)$, the sample mean/median will produce an estimate $\muh$ with expected error $\Ex[|\mu - \muh|] = \Theta(\frac{1}{\sqrt{n}})$, while the sample midrange (taking the midpoint between the smallest sample and the largest sample) only incurs error $\Theta(\frac{1}{n})$. Such phenomena naturally raise questions regarding how well the mean of any particular distribution can be learned, as well as when there are separations between the non-adaptive and the adaptive settings.

Perhaps the simplest lower bound for this task is given by Le Cam's two-point testing method: if hypothesis testing between $D$ centered at $\mu$ and $D$ centered at $\mu + \Delta$ must fail with constant probability, then any estimator of the mean must incur error at least $\Delta/2$ with constant probability. It is setting-dependent whether or not a nearly matching upper bound is attainable. Our work aims to study the shape-constraints (e.g. symmetric, unimodal, log-concave) under which the two-point testing rate can be attained for the tasks of location estimation and adaptive location estimation. In contrast, distributions have mostly so far been treated on a more case-by-case basis.

\begin{figure}
\centering
\begin{subfigure}{.3\linewidth}
    \centering
    \includegraphics[width=\linewidth]{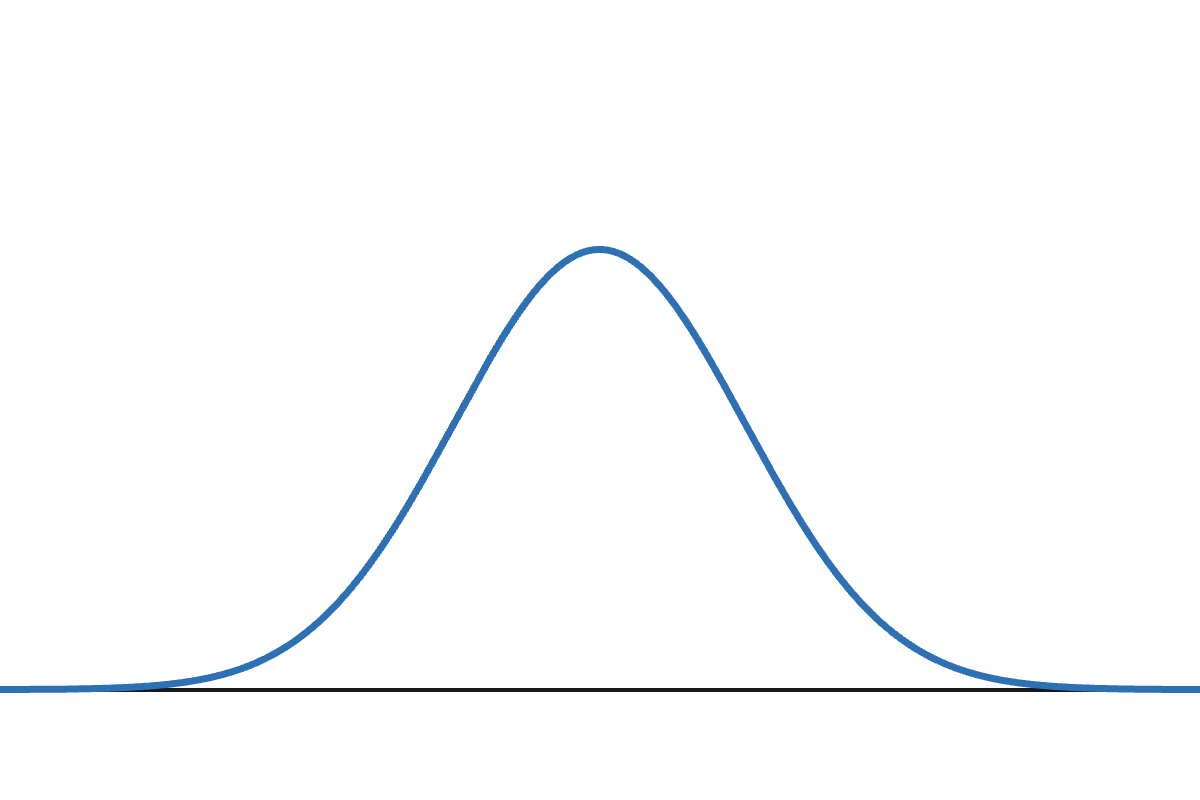}
    \caption{Gaussian}\label{fig:gaussian}
\end{subfigure}
    \hfill
\begin{subfigure}{.3\linewidth}
    \centering
    \includegraphics[width=\linewidth]{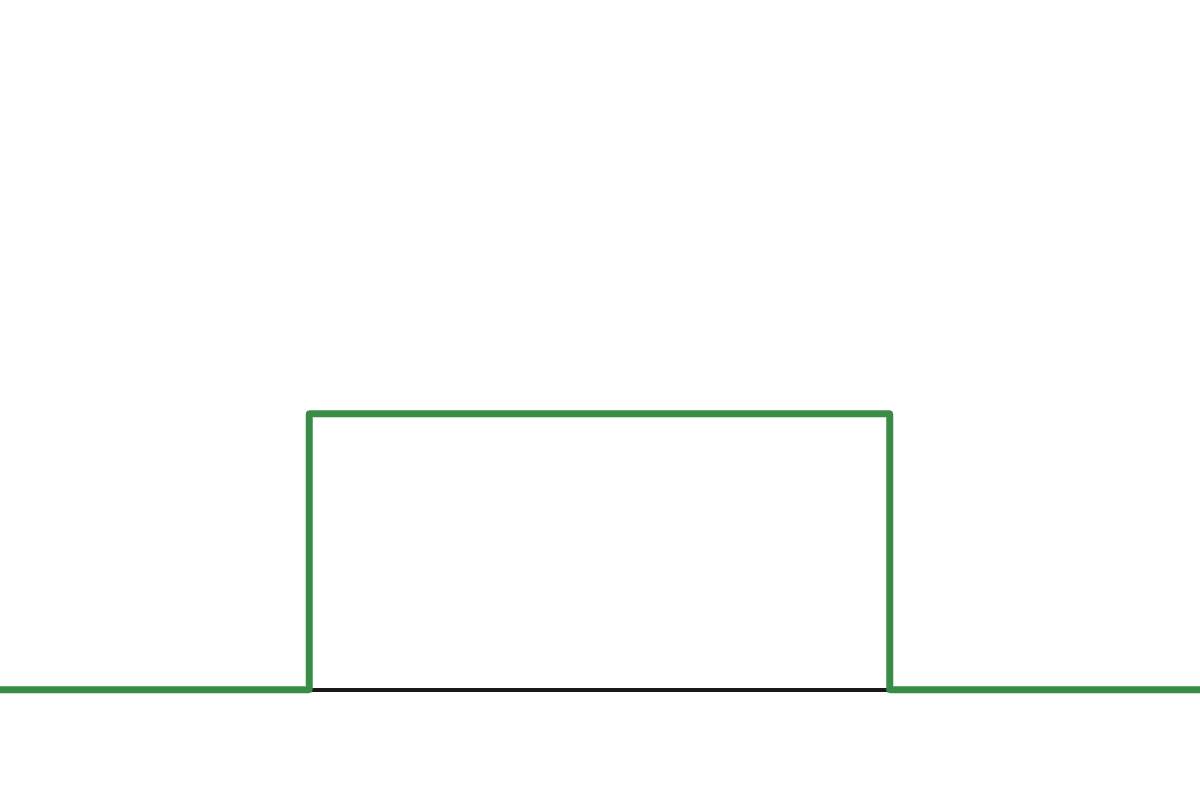}
    \caption{Uniform}\label{fig:unif}
\end{subfigure}
   \hfill
\begin{subfigure}{.3\linewidth}
    \centering
    \includegraphics[width=\linewidth]{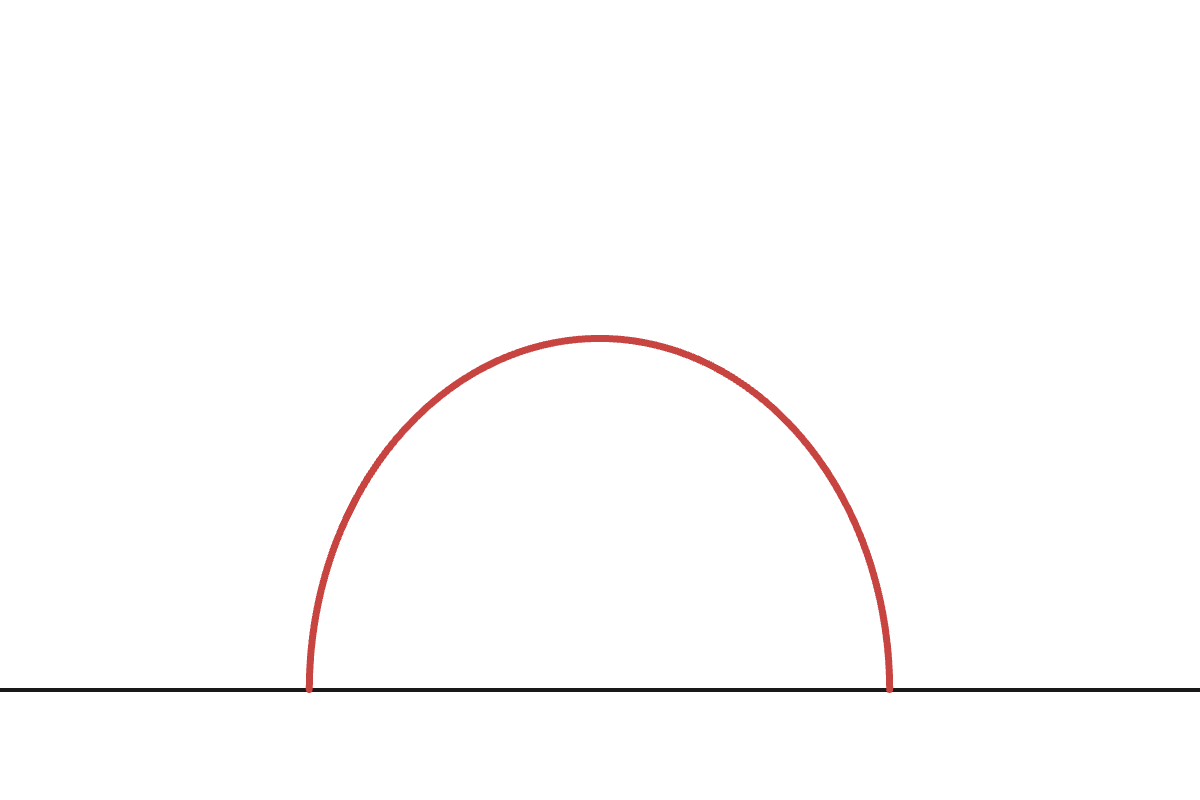}
    \caption{Semicircle: $p(x) \propto \sqrt{1-|x|^2}$}\label{fig:semi-circle}
\end{subfigure}

\begin{subfigure}{.3\linewidth}
    \centering
    \includegraphics[width=\linewidth]{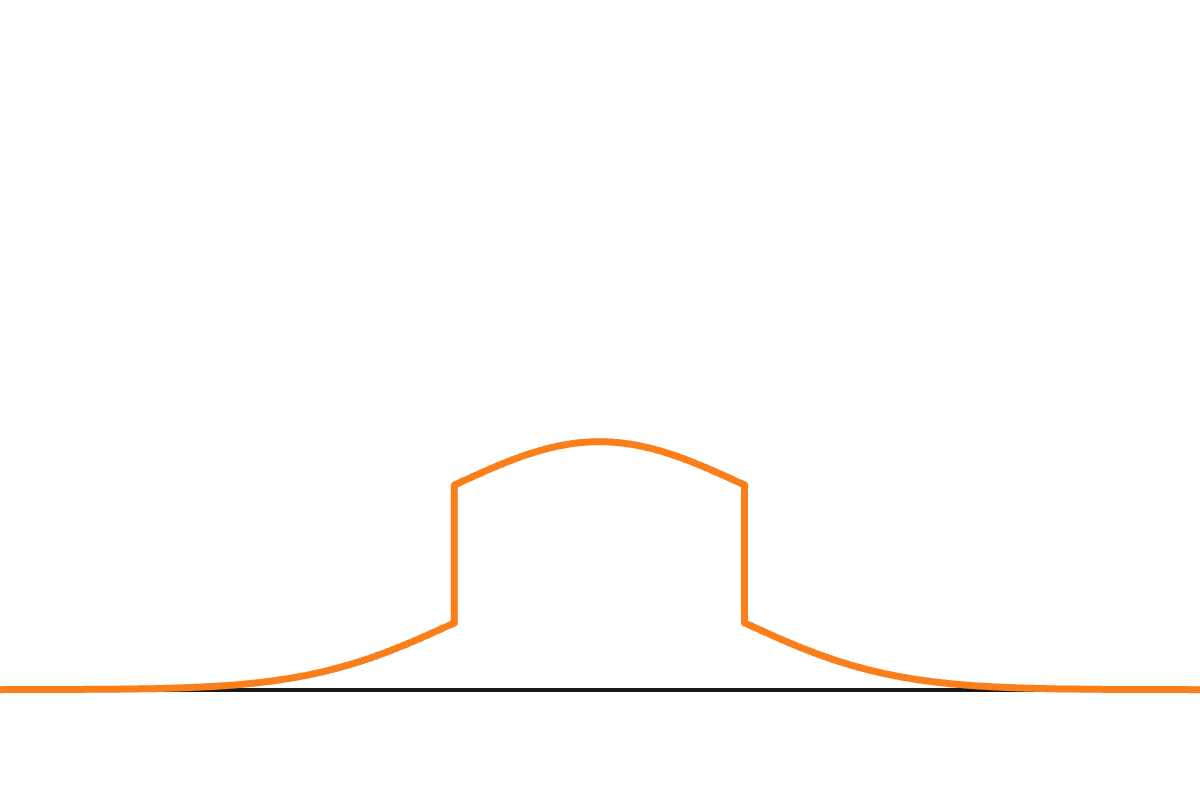}
    \caption{Mixture of a Gaussian and Uniform distribution}\label{fig:add}
\end{subfigure}
    \hfill
\begin{subfigure}{.3\linewidth}
    \centering
    \includegraphics[width=\linewidth]{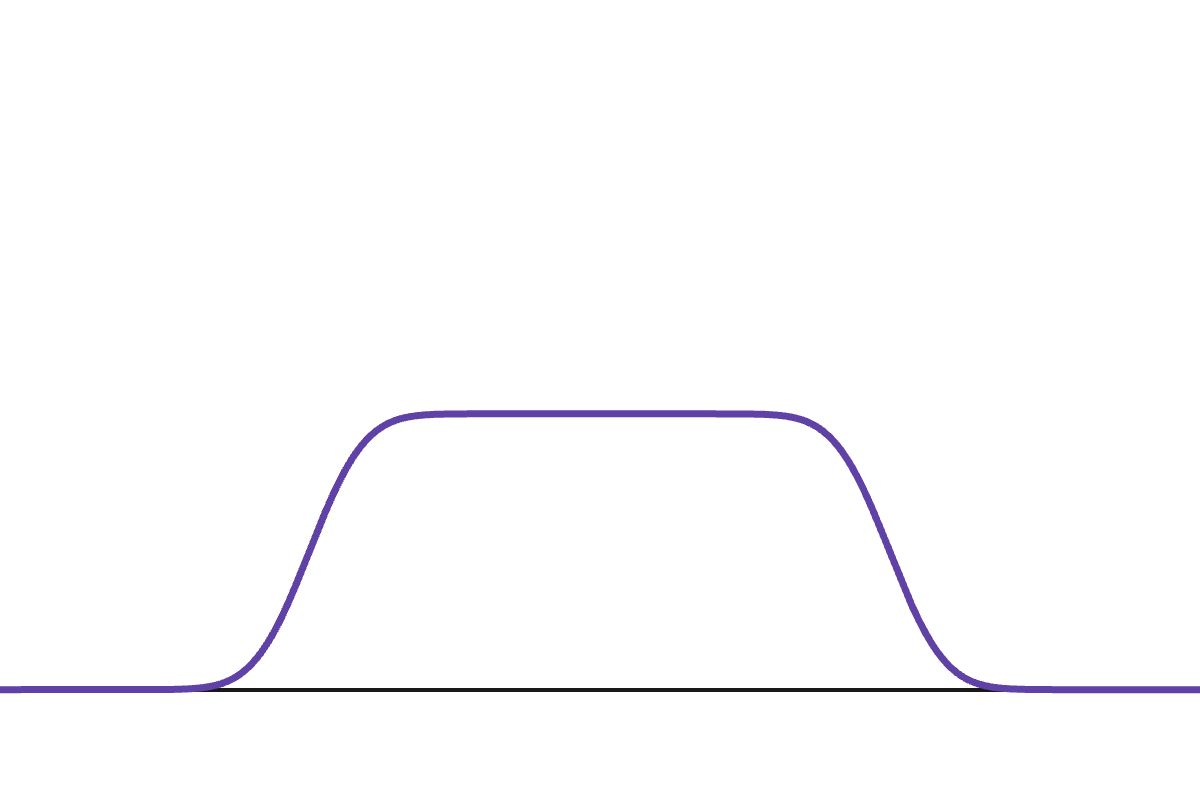}
    \caption{Convolution of a Gaussian and Uniform distribution}\label{fig:conv}
\end{subfigure}
   \hfill
\begin{subfigure}{.3\linewidth}
    \centering
    \includegraphics[width=\linewidth]{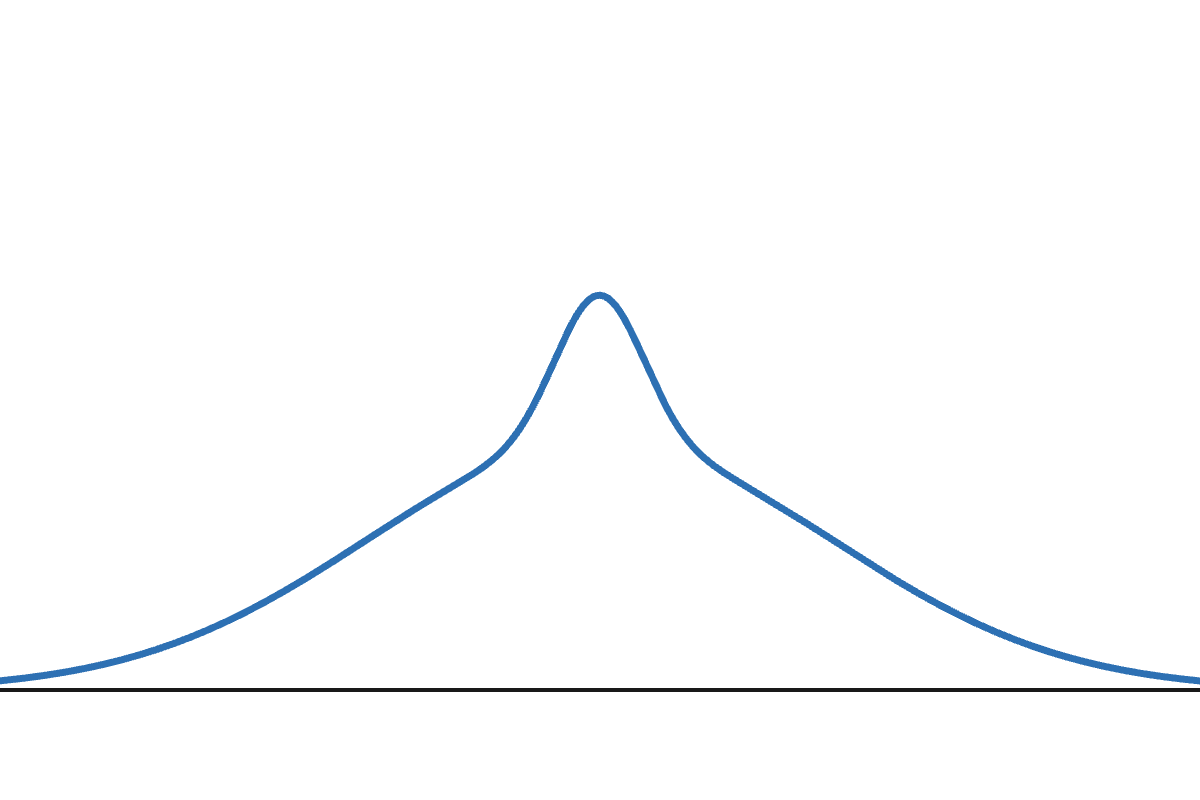}
    \caption{Mixture of two Gaussians with different variances}\label{fig:entangled}
\end{subfigure}
{\caption{Examples of symmetric log-concave densities and mixtures of log-concave densities}
\label{fig:images}}
\end{figure}

\textbf{Examples. } Let us showcase some instances that illustrate interesting behaviors for adaptive location estimation. 
\begin{itemize}
    \item For $n$ samples from a Gaussian $N(\mu,\sigma^2)$, the sample mean/median both incur optimal error of $|\mu - \muh| = \Theta(\frac{\sigma}{\sqrt{n}})$.
    \item For the uniform distribution $\unif(\mu-1\,\mu+1)$, the sample midrange (the midpoint between the smallest and the largest sample) incurs much better error of $\Theta(\frac{1}{n})$. This phenomenon occurs because there is information in the sharp discontinuity: the sample minimum and maximum concentrate within $\Theta(\frac{1}{n})$ of their expectation; the same phenomena enables $\tilde{O}(n^{-2/3})$ error for the semicircle distribution by the sample midrange.
    \item For a mixture $\frac{1}{2} N(\mu,1) + \frac{1}{2} \unif(\mu-1,\mu+1)$ (\cref{fig:add}), the sample midrange would no longer perform optimally, instead incurring error $\Theta({1}/{\sqrt{\log(n)}})$, yet the MLE would still attain $\Theta(\frac{1}{n})$ (as remarked in \cite{kao2024choosing}). This begs the question of when knowing the distribution up to translation (so one can, say, use the MLE) changes the rate dramatically. There are many more examples where rates much better than the sub-Gaussian $\Theta(\frac{\sigma}{\sqrt{n}})$ can be attained.
    \item \label{item:conv-disc} The convolution of the uniform distribution $\unif(\mu-1,\mu+1)$ and the Gaussian distribution $N(\mu,n^{-2\alpha})$ (for a constant $\alpha \in (0,1)$; \cref{fig:conv}) is merely a log-concave distribution, yet the earlier approaches are not sharp: the sample mean/median incurs error $\tilde{\Theta}(n^{- 1/2})$, the sample midrange incurs error $\tilde{\Theta}(n^{-\alpha})$, yet our later results would show the optimal error is $\tilde{\Theta}(n^{- 1/2 - \alpha/2 })$ by more carefully leveraging information from the tails. This sharper rate is not obviously attainable from the guarantees of known prior work.
    \item Mixtures of Gaussians with a common mean (even a two-component mixture $w_1 N(\mu, \sigma_1) + (1-w_1) N(\mu, \sigma_2)$ is non-trivial, \cref{fig:entangled}) demonstrate interesting behavior, studied as \textit{entangled mean estimation} or \textit{heteroskedastic mean estimation}, where works \cite{chierichetti2014learning,liang2020learning,yuan2020learning,pensia2022estimating,devroye2023mean,compton2024near} analyzed a collection of algorithms (median, shorth, modal, iterative trimming, and balance finding estimators) and resolved that the optimal rate entails a phase transition \cite{liang2020learning,compton2024near}.
\end{itemize}

The examples we presented were all solved by a collection of different estimators, and it is natural to wonder whether a unified approach can adaptively recover near-optimal rates for many distributions. Our main result will design a new algorithm that nearly attains the two-point testing lower bound for all these examples.

\begin{figure}
\centering
\begin{subfigure}{.32\linewidth}
    \centering
    \includegraphics[width=\linewidth]{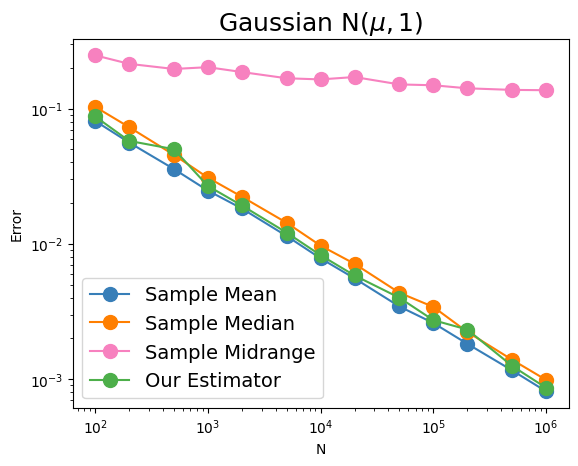}
    \caption{}\label{fig:plot-gaussian}
\end{subfigure}
    \hfill
\begin{subfigure}{.32\linewidth}
    \centering
    \includegraphics[width=\linewidth]{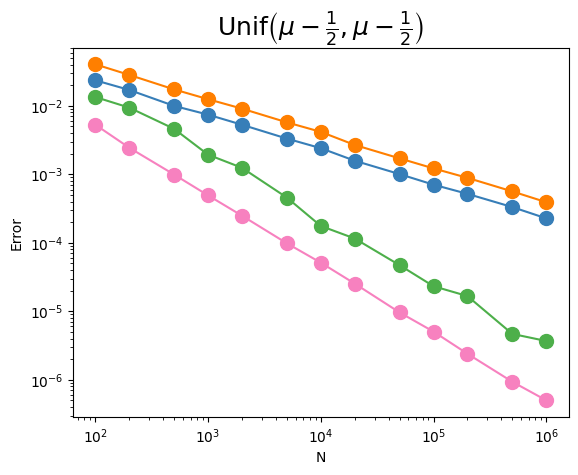}
    \caption{}\label{fig:plot-unif}
\end{subfigure}
   \hfill
\begin{subfigure}{.32\linewidth}
    \centering
    \includegraphics[width=\linewidth]{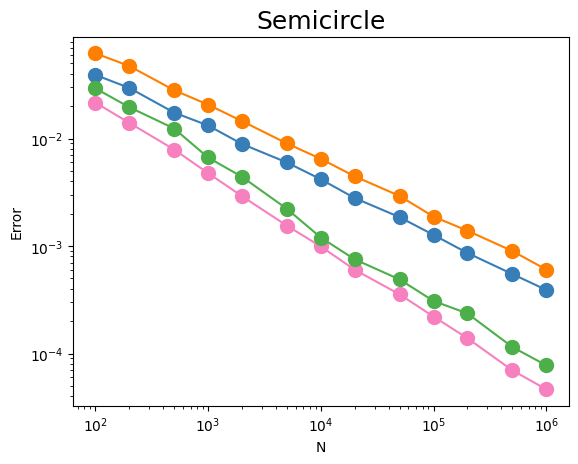}
    \caption{}\label{fig:plot-semi-circle}
\end{subfigure}

\vspace{-0.05cm}
\begin{subfigure}{.32\linewidth}
    \centering
    \includegraphics[width=\linewidth]{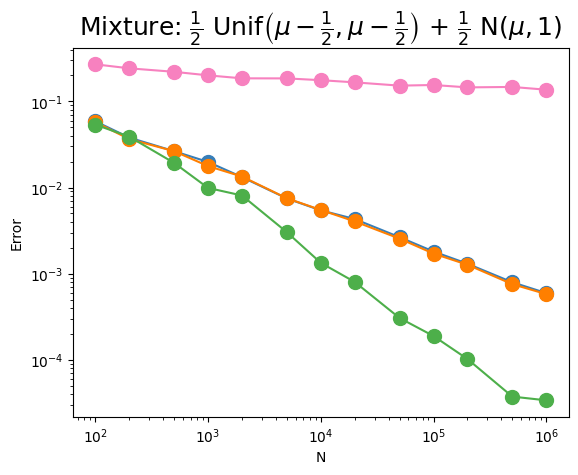}
    \caption{}\label{fig:plot-add}
\end{subfigure}
    \hfill
\begin{subfigure}{.32\linewidth}
    \centering
    \includegraphics[width=\linewidth]{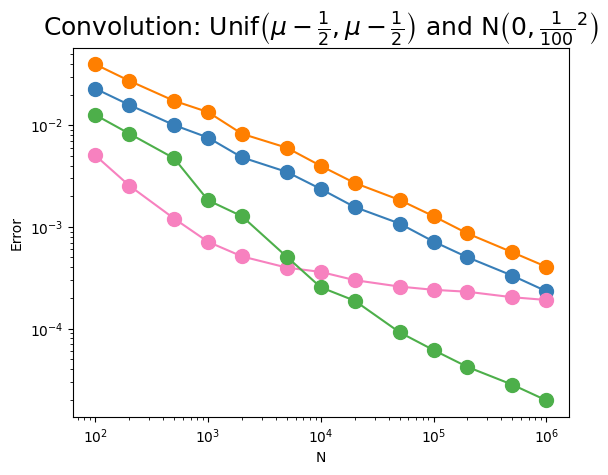}
    \caption{}\label{fig:plot-conv}
\end{subfigure}
   \hfill
\begin{subfigure}{.32\linewidth}
    \centering
    \includegraphics[width=\linewidth]{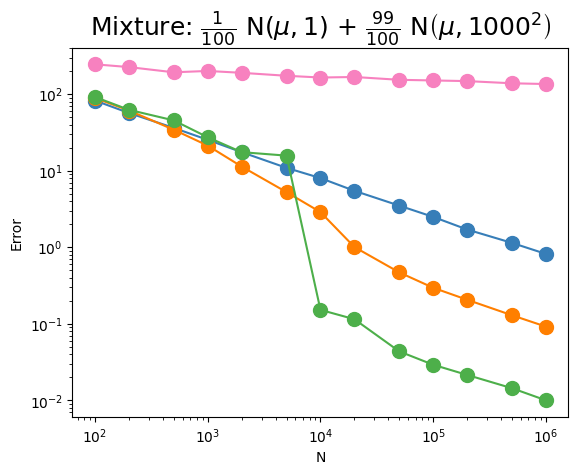}
    \caption{}\label{fig:plot-entangled}
\end{subfigure}
{\caption{Performance of our estimator (\cref{algo:fast}) for corresponding distributions in \cref{fig:images}.}
\label{fig:plots}}
\end{figure}

\textit{Simulations. } We examine performance of our estimator on these examples in \cref{fig:plots}, where each point is the average of $500$ tests. Running a short Python implementation\footnote{\url{https://github.com/SpencerCompton/mean-estimation}} of our estimator on $N=10^6$ samples took approximately $40$ seconds on a laptop. We interpret our estimator in \cref{fig:plot-gaussian,fig:plot-unif,fig:plot-semi-circle,fig:plot-add} as behaving similarly to the optimal rates: $\tilde{\Theta}(\frac{1}{\sqrt{n}})$ (\cref{fig:plot-gaussian}), $\tilde{\Theta}(\frac{1}{n})$ (\cref{fig:plot-unif}), $\tilde{\Theta}(\frac{1}{n^{2/3}})$ (\cref{fig:plot-semi-circle}), and $\tilde{\Theta}(\frac{1}{n})$ (\cref{fig:plot-add}). Lagging behind by a multiplicative factor, as in \cref{fig:plot-unif}, is not too surprising as our algorithm and analysis are loose up to polylogarithmic factors. In \cref{fig:plot-conv}, we observe how when $N$ is small relative to the standard deviation of the Gaussian convolution then behavior is similar to the uniform distribution, for larger $N$ there is information in the tail not leveraged by the other estimators, and for $N$ even larger than our simulation then we expect sample median/mean to improve beyond sample midrange and close the gap with our estimator (recall our earlier discussion of \cref{fig:conv}). Finally, in \cref{fig:plot-entangled}, we observe a sharp improvement in performance when $N$ is large enough that our estimator is able to detect the mixture component with smaller weight and standard deviation. Collectively, these simulations give some insight into how we adaptively attain sharper guarantees for many distributions with one estimator.

\textbf{Hellinger modulus of continuity. } We now provide background to introduce the \textit{Hellinger modulus of continuity} which will characterize the two-point testing rate. The \textit{Hellinger distance} is a distance metric on probability distributions:
\begin{definition}[Hellinger distance]
   If $P,Q$ are distributions over the same probability space $\Omega$ with densities $p$ and $q$, then the \textit{squared Hellinger distance} between $P$ and $Q$ is
   \[
        \dhsq(P,Q) = \frac{1}{2}\int_\Omega \left(\sqrt{p(x)} - \sqrt{q(x)}\right)^2
   \]
\end{definition}

Throughout this paper, we may also directly reference the Hellinger distance between probability densities. The Hellinger distance may be related to the total variation distance:
\begin{fact}[e.g. \cite{le2000asymptotics} page 44]\label{fact:tv-hell}
    \[
    \dhsq(P,Q) \le \dtv(P,Q) \le \sqrt{2 \dhsq(P,Q)}
    \]
\end{fact}

The Hellinger distance \textit{tensorizes}, which makes it ideal for studying the sample complexity of hypothesis testing.
\begin{fact}[Tensorization of Hellinger distance; e.g. \cite{le2000asymptotics} page 45]\label{fact:tensor-hell}
Suppose $P,Q$ are distributions over the same probability space $\Omega$, and let $P^{\otimes n}$ and $Q^{\otimes n}$ denote the distribution of $n$ i.i.d. samples from $P$ and $Q$ respectively.
Then
\[
\dhsq(P^{\otimes n},Q^{\otimes n}) = 1 - \left(1-\dhsq(P,Q)\right)^n.
\]
\end{fact}
In particular, as a corollary of \Cref{fact:tv-hell,fact:tensor-hell}, the Hellinger distance is ideal for measuring the sample complexity of hypothesis testing between two distributions.
If $P,Q$ are distributions over the same probability space, then
\begin{equation}
\dtv(P^{\otimes n},Q^{\otimes n}) \ge \left(1 - e^{-n\cdot \dhsq(P,Q)}\right),\label{eq:tvbd}
\end{equation}
so that once $n \sim \frac{1}{\dhsq(P,Q)}$, $n$ samples distinguish between $P$ and $Q$ with at least constant probability.
The second inequality in \Cref{fact:tv-hell} shows that if $n \ll \frac{1}{\dhsq(P,Q)}$, hypothesis testing between $P$ and $Q$ with fewer than $n$ samples is information-theoretically impossible except with vanishing probability.

Since the squared Hellinger distance $\dhsq(P,Q)$ informs the sample complexity of hypothesis testing between $P$ and $Q$, Donoho and Liu \cite{donoho1987geometrizing} introduced the \textit{Hellinger modulus of continuity} that yields often-sharp two-point testing lower bounds. The Hellinger modulus is defined for a functional $T$ and class $\mathbf{F}$ as
\begin{equation*}
    \omega(\eps) \triangleq \sup \{ |T(F_1) - T(F_0)|: \dhsq(F_1,F_0) \le \eps, F_i \in \mathbf{F} \}.
\end{equation*}
For estimating the mean of a distribution $D$, the Hellinger modulus can be instantiated as
\begin{equation*}
    \omega_D(\eps) \triangleq \sup \{ | \mu_1 - \mu_2 |: \dhsq(D_{\mu_1},D_{\mu_2}) \le \eps, \, \, \mu_1,\mu_2 \in \R \},
\end{equation*}
where $D_\mu$ denotes the distribution $D$ centered at $\mu$. Given our earlier background, we see that $\omega_D(\frac{1}{n})$ informs some two-point testing style lower bound, since $D_{\mu_1}$ and $D_{\mu_2}$ will only be distinguishable with constant probability. As immediately explored by Donoho and Liu \cite{donoho1987geometrizing,donoho1991geometrizing2,donoho1991geometrizing3}, it is often possible to nearly attain the Hellinger modulus in statistical estimation tasks. For example, they show in \cite{donoho1991geometrizing2} the Hellinger modulus rate is asymptomatically attainable if $\mathbf{F}$ is convex, $T$ is linear, and $\omega$ is H\"olderian; this style of result is recently furthered in \cite{juditsky2009nonparametric,polyanskiy2019dualizing}. In our setting $T$ is linear, but the main obstacle in employing techniques from such works is that our class $\mathbf{F}$ is not convex. Observe how convex combinations of translations of $D$ are not necessarily a translation of $D$. Shape-constraints do not form a convex set either; convex combinations of translations of symmetric distributions need not be symmetric. 

We will study the shape-constraints on $D$ under which it is possible to attain error $|\mu - \muh| \le \polylog(n) \cdot \omega_D(\frac{\polylog(n)}{n})$ (our formal statement of results will add dependence on a failure probability $\delta$). This roughly corresponds to error that is polylogarithmically larger than the two-point testing bound for $\frac{n}{\polylog(n)}$ samples. As motivation, when we can adaptively attain the two-point testing rate, this makes it simple to analytically approximate the optimal error (it lies between the two-point testing rates for $\frac{n}{\polylog(n)}$ and $n$ samples, meaning it suffices to compute this simple lower bound instead of more sophisticated lower bound methods), and we may state strong guarantees like \textit{``given $n$ samples, our adaptive estimator performs at least as good as a distribution-specific optimal estimator given $n/\polylog(n)$ samples.''}

\subsection{Preliminaries} \label{sec:preliminaries}
A probability density $p$ is a $k$-mixture if $p(x) = \sum_{i=1}^k w_i \cdot p_i(x)$, where $w_i \ge 0$, $\sum_{i=1}^k w_i = 1$, and each $p_i$ is a density. It is a $k$-mixture of log-concave distributions if each $p_i$ is log-concave. It is a mixture of centered/symmetric components if all mixture components are symmetric around a common point. We denote $p_\Delta$ to be the density $p$ shifted to recenter at $\Delta$, meaning $p_\Delta(x) \triangleq p(x-\Delta)$. We use asymptotic notation very strictly to mean the statements hold for some universal constants chosen independently of the arguments. For example, $a \le O(b) + 2$ means there exists a universal constant $C > 0$ where $a \le C \cdot b + 2$. We interchangeably use equivalent notation of the form $a \le O(1) \cdot b + 1$. The $\Omega(\cdot)$ notation is used similarly. Additionally, $a = \Theta(b) + 2$ means there exist universal constants $C_1,C_2 > 0$ where $C_1 \cdot b + 2 \le a \le C_2 \cdot b + 2$.

 \subsection{Our Contributions}\label{subsec:contributions}
We present positive and negative results on the attainability of the two-point testing rate, both in the settings of location estimation and adaptive location estimation. We begin with our most interesting finding: the positive result for adaptive location estimation. We follow with our three complementary results that elucidate the landscape of these tasks more broadly. 

\textbf{Attainability for adaptive location estimation.} For mixtures of $k$ symmetric log-concave distributions with the same center, we show that the two-point testing rate is nearly attainable:

\begin{restatable}{theorem}{fastthm}\label{thm:fastthm}
    Suppose $p$ is a mixture of $k$ centered/symmetric log-concave distributions. There exists some universal constant $\cdist \ge 1$, where if
    \begin{equation*}
        \Delta^* \triangleq \omega_p\left(\cdist \cdot \frac{k}{n} \cdot \log(2n/\delta) \cdot \log^2(2n)\right)
    \end{equation*}

    then with probability $1-\delta$ the output $\muh$ of \cref{algo:fast} will satisfy $|\mu - \muh| \le \Delta^*/2$. Moreover, \cref{algo:fast} always runs in $O(n \log(n) \log(\log(n)))$ time.
\end{restatable}

\textit{Brief intuition. } Here is an informal outline of an algorithm that guides our ideas:
\begin{enumerate}
    \item Consider a possible estimate $\muh$ of the true mean $\mu$. 
    \item Test if there is an interval that reveals the true distribution is not symmetric around $\muh$. Precisely, check if there exists an $0 \le a < b$ where the number of samples within $[\muh - b, \muh - a]$ is noticeably different from the number within $[\muh + a, \muh + b]$.
    \item For any $\muh$ that passes this test, hope it is a good estimate of $\mu$.
\end{enumerate}
Nothing is immediately clear about the performance of this algorithm. First, it is inefficient to consider all values of $\muh,a,b$, but we will delay this concern. Notably, it is not clear how good of an estimate $\muh$ must be if it passes these interval tests. For arbitrary symmetric distributions, a $\muh$ passing interval tests can indeed be a relatively poor estimate (e.g. consider distributions with many discontinuities, where intervals are not leveraging all the information, as we will later see in \cref{theorem:adaptive-lb}). Surprisingly, we show that for mixtures of log-concave distributions, $\muh$ is close (in terms of the Hellinger modulus) to $\mu$ with high probability.

We observe that performance of our informal algorithm boils down to the following key question: if $p$ and a translation of $p$ have large Hellinger distance, must there be an interval of the domain where their expected number of samples are noticeably different? This is not true for general $p$, but we will show it holds for $p$ satisfying our assumptions.

Trying to answer this question, we draw connections to \cite{bhatt2021information,pensia2023communication} who show how the Hellinger distance between any two distributions can be approximately preserved by a channel $T$ (here, channel just means a deterministic function of the observable) that outputs an indicator of a threshold of the likelihood ratio: i.e. the indicator of $p(x)/q(x) \ge \tau$ for a well-chosen threshold parameter $\tau \ge 0$. Roughly, if $P$ and $Q$ are easy to distinguish from $n$ samples, then $T(P)$ and $T(Q)$ are easy to distinguish from $\tilde{O}(n)$ samples. From their results, it becomes clear that our key question is essentially resolved if the appropriate likelihood threshold channel can be simulated by an indicator of an interval of the domain (we call this an interval statistic). Later, we show it is also enough to approximately simulate the channel.

In the simpler case of $k=1$, a simple calculation reveals that any likelihood threshold channel is exactly simulated by an interval statistic. This is not true for $k>1$, but with non-trivial analysis involving piecewise-approximations of the densities and likelihood ratios, we are able to show that it is still possible to approximate the channel sufficiently well with an interval statistic.

Eventually, we further refine our approach to permit a near-linear time algorithm that still aligns with the intuition of the informal algorithm we discussed. This gives an efficient algorithm (with no tuning parameters) that we evaluated in \cref{fig:plots} on our examples of \cref{fig:images}.

\textbf{Unattainability for adaptive location estimation.} We show that if the distribution is only promised to be unimodal and symmetric, then such a rate is unattainable:

\begin{restatable}{theorem}{adaptivelb}\label{theorem:adaptive-lb}
    There exist universal constants $0<C_1,C_2<1$ such that for any $n$ larger than a sufficiently large constant, and $\nu \ge 1$, then for every estimator $\est$ there is a unimodal and symmetric distribution where $\est$ incurs much larger error than the two-point testing rate with constant probability:
    \begin{equation*}
        \min_{\est} \max_{\substack{\textrm{unimodal/symmetric $D$} \\ \mu \in \R}} \Pr_{X \sim D(x-\mu)^{\otimes n}, \est}\left[|\est(X) - \mu |  \ge \nu \cdot \omega_{D}\left(\frac{C_1}{\nu \cdot n^{9/10} \sqrt{\log(n)}}\right) > 0 \right] \ge C_2
    \end{equation*}
    Note that the statement has randomness over $\est$ to account for non-deterministic estimators.
\end{restatable}

Notably, the exponent for $n$ is $9/10$ instead of 1. Observe that if we invoke this theorem with e.g. $\nu = n^{0.01}$, we rule out the possibility of a positive guarantee of the form $n^{0.01} \cdot \omega_{D}\left(\frac{C}{n^{0.91} \sqrt{\log(n)}}\right) > \polylog(n) \cdot \omega_{D}\left(\frac{C}{n^{0.91} \sqrt{\log(n)}}\right) \ge \polylog(n) \cdot \omega_{D} \left( \frac{\polylog(n)}{n^{0.92}} \right)$ for sufficiently large $n$, since the first $\omega_D(\cdot)$ term is positive, and $\omega_D(\cdot)$ is non-decreasing.

\textit{Brief intuition. } In the proof of our positive result \cref{thm:fastthm}, we crucially leveraged that thresholds of the likelihood ratio of log-concave mixtures and their translations could be well-approximated by interval statistics. For our hard instance, we hope to use a distribution where the likelihood ratio with its translation is large in regions that are very spaced apart, so interval statistics are less helpful because any large interval must contain large fractions of the domain that contain minimal information. Moreover, if we consider a family of such distributions with different spacings, then we expect it will be impossible to find where the likelihood ratio is large. We will show there is no estimator that attains two-point testing rates for all distributions in this family.

More concretely, we consider a \textit{step distribution}, which is a unimodal and symmetric distribution that resembles a collection of steps. Comparing this distribution with a slight translation in \cref{fig:step}, we see that the likelihood ratio is not equal to $1$ in regions that are spaced apart. We carefully study a family of step distributions with different step widths, and show this mixture family is indistinguishable from a triangle distribution (which has a worse two-point testing rate).

 \begin{figure}
\centering
\includegraphics[width=0.5\linewidth]{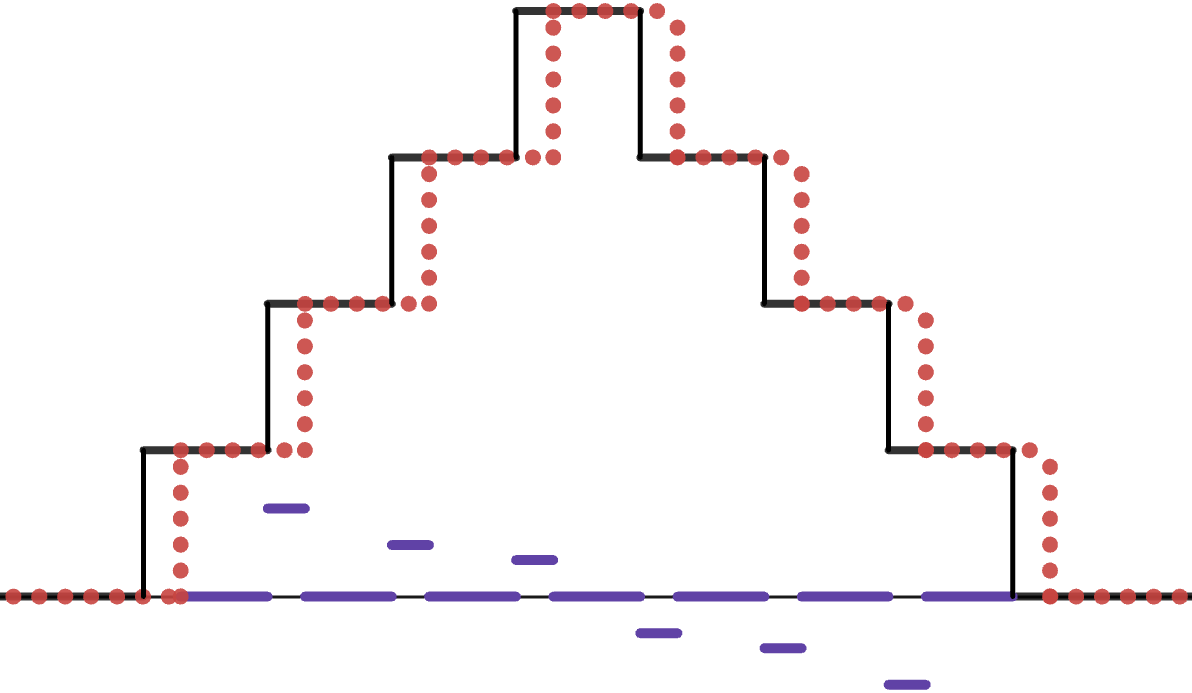}
{\caption{Step distribution (black, solid), a slight translation (red, dotted), and the logarithm of their likelihood ratio (purple, solid). Observe the likelihood ratio is not equal to $1$ in disjoint regions.}
\label{fig:step}}
\end{figure}

\textbf{Attainability for location estimation.} On the other hand, we show the two-point testing rate is attainable for location estimation even when the distribution is only promised to be unimodal:

\begin{restatable}{theorem}{localg}\label{thm:loc-alg}
Suppose $p$ is a unimodal probability density with mode $p(0)$, $\sqrt{n} \ge 6 \log(2/\delta)$, and $\delta \in (0,\frac{1}{2})$. There exists some universal constant $\cdist \ge 1$, where if 
\begin{equation*}
    \Delta^* \triangleq \omega_p\left(\cdist \cdot \frac{\log(n/\delta)}{n} \right)
\end{equation*}

then with probability $1-\delta$, the output $\muh$ of our algorithm will satisfy $|\mu - \muh| \le 4  \Delta^*$.
\end{restatable}
We remark that the condition of $\sqrt{n} \ge 6 \log(2/\delta)$ is semi-arbitrary, but our proof does need at least some bound on $\delta$ in relation to $n$.

\textit{Brief intuition. } While the work of \cite{gupta2024minimax} shows that a variant of the MLE attains a form of minimax optimality for this task, it is still not obvious how to directly analyze whether their algorithm attains the two-point testing rate for this task. Thus, we present and analyze a simple approach that attains this guarantee.

For our approach, we use the first $n/2$ samples as candidates for our estimate $\muh$. We prove that with high probability, one of these samples $X_i$ will satisfy that $\dhsq(p_{\mu},p_{X_i}) \le O(1) \cdot \frac{\log(1/\delta)}{n}$. From there, we are able to leverage a tournament procedure that is essentially the same as Le Cam-Birg\'e's pairwise comparison estimator (exposited in Section 32.2.2 \cite{polyanskiy2025information}; see also \cite{le2012asymptotic,van2002statistical,birge1983approximation}).

We remark that this approach should be fairly straightforward to extend to mixtures of a bounded number of unimodal distributions (not necessarily with the same center) if desired. For our purposes, we primarily desired to show this contrast with the corresponding negative result for unimodal distributions in adaptive location estimation.

\textbf{Unattainability for location estimation.} Finally, we show that if the distribution is only promised to be symmetric, then such a rate is unattainable:

\begin{restatable}{theorem}{loclb}\label{thm:loc-lb}
    For any positive integer $n$ and positive value $\nu$, there exists a distribution $D_{n,\nu}$ that is symmetric around $0$, and for every estimator $\est(X)$, there exists a centering $\mu$ where $\est$ incurs large error with constant probability:
    \begin{equation*}
        \min_{\est} \max_{\mu} \Pr_{X \sim D_{n,\nu}(x-\mu)^{\otimes n}, \est}\left[|\est(X) - \mu |  \ge \nu \cdot \omega_{D_{n,\nu}}\left(\frac{1}{10}\right) > 0 \right] \ge \frac{1}{4}
    \end{equation*}
    Note that the statement has randomness over $\est$ to account for non-deterministic estimators.
\end{restatable}

This indicates that location estimation does not get much easier from symmetry alone, as the lower bound is quite strong: by setting $\nu$ as desired, the error gets arbitrarily worse than $\omega_D(\tfrac{1}{10}) \ge \omega_D(\tfrac{1}{n})$. The constants in our theorem statement are semi-arbitrary, but adding more variables to our theorem does not seem more insightful in our primary goal of showing that the two-point testing rate is not nearly attainable under just an assumption of symmetry.

\textit{Brief intuition. } Our analysis considers a family of distributions and uses the probabilistic method to conclude that at least one distribution satisfies desired technical properties which enable a type of packing lower bound. Our family of distributions will essentially be uniform distributions $\unif(\mu-1,\mu+1)$ with a random half of regions of their support missing. The family is slightly modified to enforce symmetry constraints. From the details of our construction, these modified distributions should not actually be much easier to estimate than by using the sample midrange for error $\Theta(\frac{1}{n})$, but the two-point testing lower bound will deceptively look much more favorable.

\subsection{Related Work}
\textbf{Asymptotic setting. } Location estimation and adaptive location estimation have been more extensively studied in the asymptotic settings: where the distribution $D$ is fixed and then we analyze the performance of estimators as $n \rightarrow \infty$. For location estimation, it is known that the Fisher information rate is attainable: the MLE asymptotically approaches $N(\mu,\tfrac{1}{n\mathcal{I}})$, where $\mathcal{I}$ is the Fisher information of $D$ (e.g. see Chapter 7 of \cite{van2000asymptotic}). For adaptive location estimation, many works have studied estimation under the assumption that $D$ is symmetric (e.g. \cite{stein1956efficient,van1970efficiency,stone1975adaptive,sacks1975asymptotically,beran1978efficient,dalalyan2006penalized}). Stone \cite{stone1975adaptive} showed that the Fisher information rate is asymptotically attainable if $D$ is symmetric.  More recently, Laha \cite{laha2019location} showed that tuning parameters may be avoided for adaptive location estimation of symmetric distributions if $D$ is also log-concave. 

For distributions with infinite Fisher information (e.g. $\unif(\mu - 1, \mu + 1)$, non-smooth distributions), it is perhaps sharper to consider a result of Le Cam \cite{lecam1973convergence} who showed the Hellinger distance two-point testing rate is attainable given conditions related to the covering number of the family under the Hellinger metric.

\textbf{Finite-sample setting. } In this setting, we focus on how well the location may be estimated for a particular $D$ and $n$. The work of \cite{gupta2024minimax} showed that for location estimation, variants of the MLE attained minimax optimal guarantees for any $D$ and $n$, yet it does not necessarily reveal what the optimal rate is. The works of \cite{gupta2022finite} and \cite{gupta2023finite} study location estimation and adaptive location estimation, respectively, and show how estimators similar to \cite{stone1975adaptive} are able to attain the \textit{smoothed Fisher information rate}, which is the Fisher information of $D$ convolved with $N(0,r^2)$ (where $r$ is a smoothing parameter that depends on $n$, and they require $D$ is symmetric for adaptive location estimation). For some distributions, this is sufficient to attain guarantees with optimal constant factors. Unfortunately, for other distributions, the smoothing parameter $r$ may be sufficiently large such that too much information is lost. For example, their error guarantees for $\unif(\mu-1,\mu+1)$ are polynomially worse than $\Theta(\frac{1}{n})$.

The balance finding algorithm of \cite{compton2024near} for heteroskedastic mean estimation inspires our estimator. The algorithm looks for an estimate $\muh$ that exhibits a particular kind of balance, where for parameters $w$ and $\Delta$, the number of samples within $w$ to the left of $\muh$ and $w$ to the right of $\muh$ are approximately balanced, yet there is strong imbalance for $\muh \pm \Delta$. In this way, balance finding also leverages interval statistics to inform its estimator. While the balance finding algorithm attains desired guarantees for the distributions in \cref{fig:gaussian,fig:entangled}, it incurs polynomially-suboptimal errors for \cref{fig:unif,fig:semi-circle,fig:add,fig:conv}. Sweep-line techniques similarly enable near-linear time.

The work of \cite{kao2024choosing} focuses on adaptive location estimation with the goal of minimizing the $L_\gamma$ loss for $\gamma \ge 2$, where $\gamma$ is chosen data-dependently (the guarantees are a mix of asymptotic and finite-sample). Their approach is sufficient to enable sharp rates for distributions such as $\tilde{O}(\frac{1}{n})$ for $\unif(\mu-1,\mu+1)$ and $\tilde{O}(n^{-2/3})$ for the semicircle distribution. Their results also extend to the regression setting. In their discussion, they remark how this approach is unable to leverage discontinuities in the interior of the support, such as in \cref{fig:add}, which our results will encompass.

\textbf{Additional related work. } For examples such as \cref{fig:add}, much of the difficulty of adaptive location estimation boils down to determining where the discontinuity in the density occurs. In this sense, it is natural that techniques will be shared with the richly-studied task of density estimation. Focusing on log-concave distributions, it is recently known that the log-concave MLE learns the density within optimal Hellinger distance up to logarithmic factors (for any number of dimensions) \cite{han2016approximation,kim2016global,kur2019optimality}. Most relevant to our work are the techniques of \cite{chan2014efficient}, who (among other results) optimally learn mixtures of log-concave distributions in total variation distance up to logarithmic factors. Their techniques analyze estimates where the number of samples empirically within collections of intervals roughly match the expected number of samples for the estimate. Their analysis uses piecewise-polynomial approximations of log-concave distributions. Later, our work will design an algorithm that also verifies whether indicators of intervals match what is expected given shape-constraints, whose analysis also uses piecewise approximations of log-concave distributions (and a slightly finer notion of matching). This line of prior work is crucially leveraging the notion of $\mathcal{A}_k$ distance, roughly defined as the total variation distance witnessed by the union of $k$ disjoint intervals (also studied, for example, by \cite{devroye2001combinatorial,diakonikolas2014testing,diakonikolas2015optimal,diakonikolas2017near,diakonikolas2019testing,diakonikolas2023testing}; \cite{gerber2025density} is related). Our work will later focus instead on the \textit{Hellinger distance} witnessed by the union of $k$ disjoint intervals. 

An interesting recent line of work focuses on getting optimal constant-factor dependence on the sub-Gaussian rate (e.g. \cite{catoni2012challenging,lee2022optimal,lee2022optimal2,gupta2024beyond}). In contrast, our work focuses on shape-constrained distributions where we may perform polynomially better than the sub-Gaussian rate (but incur logarithmic-factors of lossiness in our analysis).

For some recent examples (among many) to showcase the influence of the modulus of continuity perspective: \cite{cai2015framework} introduces a local modulus of continuity as a benchmark for estimating convex functions, \cite{duchi2024right} uses the local modulus of continuity (instead with total variation distance) for locally private estimation, and \cite{foster2021statistical} presents an analog of the modulus of continuity for interactive learning.

\section{Adaptive Location Estimation for Log-Concave Mixtures}

Please recall the ``brief intuition'' for \cref{thm:fastthm} in \cref{subsec:contributions}, where we provided an informal outline of our algorithm and some key ideas for our proof method. In this section, we will provide an algorithm for estimating the mean of mixtures of centered/symmetric log-concave distributions, with a guarantee in terms of the Hellinger modulus of the distribution:

\fastthm*

We now roughly outline our proof structure. Our goal is to show that there exists a failing interval test if $\muh$ is poor enough such that $\dhsq(p_{\mu},p_{\muh})$ is large. Roughly, we will later show that this occurs if whenever $\dhsq(p_{\mu},p_{\muh})$ is large, there exists some interval that witnesses the distance: the expected number of samples inside this interval is noticeably different for $p_{\mu}$ and $p_{\muh}$. We focus on showing this witnessing property first, and then focus on the algorithmic aspects later.

First, in \cref{sec:approxchannel}, we discuss the results of \cite{bhatt2021information,pensia2023communication} that show how the Hellinger distance between any two distributions can be approximately preserved by a channel that outputs an indicator of a threshold of the likelihood ratio: i.e. the indicator of $p(x)/q(x) \ge \tau$ for a well-chosen threshold parameter $\tau \ge 0$. We then observe how a channel that approximates the optimal thresholding channel still approximately preserves the Hellinger distance between the two distributions. Second, in \cref{sec:logapprox}, we prove how any likelihood thresholding channel between a log-concave mixture and its translation can be approximated by an interval statistic. This proof relies on a careful approximation of the distribution and likelihood ratio by piecewise-constant functions. Finally, we have shown our desired witnessing property. In \cref{sec:alg}, we combine these tools to show how they imply that any sufficiently bad estimate $\muh$ will fail some interval test with high probability. We further refine the structure of these interval tests to permit a near-linear time algorithm that still aligns with the intuition of the informal algorithm we discussed.

\subsection{Near-Optimality of Approximate Likelihood Threshold Channels}\label{sec:approxchannel}
Consider the task of distinguishing between two distributions $p$ and $q$ from samples. It is classically known that the sample complexity of this task is $\Theta(\frac{1}{\dhsq(p,q)})$ by looking at the product of the likelihood ratio for all samples. Interestingly, \cite{bhatt2021information} and \cite{pensia2023communication} show that the sample complexity only increases logarithmically if we merely look at statistics of the indicator of a threshold on the likelihood ratio. We will focus on the form of the result given by \cite{pensia2023communication} for convenience, but the result of either paper would yield the tool that is crucial for our work. More concretely, consider the class of thresholds on the likelihood ratio:

\begin{definition}
    $\mathcal{T}^{\operatorname{thresh}} \triangleq \{\mathbbm{1}_{p(x)/q(x) \ge \tau}(x) : \tau \ge 0 \}$
\end{definition}

Then, \cite{pensia2023communication} show there exists a $\tstar \in \mathcal{T}^{\operatorname{thresh}}$ where $\dhsq(\tstar p, \tstar q) \approx \dhsq(p,q)$. We state a special-case of one of their results as follows:\footnote{In their work, they show results for when your threshold may output one of $D$ options, indicating whether $p(x)/q(x) \in [0,\tau_1), [\tau_1,\tau_2),\dots, \textrm{ or } [\tau_{D-1},\infty)$. It is sufficient for our work to focus on $D=2$. They also study other ``well-behaved'' f-divergences beyond Hellinger distances.}
\begin{theorem}[Corollary 3.4 of \cite{pensia2023communication}; preservation of Hellinger distance]\label{thm:og-thresh}
    For any $p,q \in \Delta_k$ (the k-dimensional simplex), there exists a $\tstar \in \mathcal{T}^{\operatorname{thresh}}$ such that the following holds:
    \begin{equation}
        1 \le \frac{\dhsq(p,q)}{\dhsq(\tstar p, \tstar q)} \le 1800 \min\{k, k' \},
    \end{equation}
    where $k' = \log(4/\dhsq(p,q))$.
\end{theorem}

We remark on some properties of this result. Note that properties 2-4 simultaneously hold for $p,q$ or after exchanging $p,q$:

\begin{remark} \label{remark:opt-props}\mbox{}
    \begin{enumerate}
        \item The proof of \cref{thm:og-thresh} also holds for continuous distributions $p,q$ if we replace the dependence on $\min\{k,k'\}$ with just $k'$.
        \item The proof also implies a stronger bound that $\frac{\dhsq(p,q)}{\dhsq(\tstar p, \tstar q)} \le \frac{\dhsq(p,q)}{\left(\sqrt{\Pr_{x \sim p}[\tstar(x)=1]} - \sqrt{\Pr_{x \sim q}[\tstar(x)=1]}\right)^2} \le 1800 \min\{k, k'\}$.
        \item $\tstar(x) = \mathbbm{1}[\frac{p(x)}{q(x)} \ge 1 + \tau^*]$ where $\sqrt{\frac{\dhsq(p,q)}{104 \log(4/\dhsq(p,q))}} \le \tau^* \le 1$.
        \item $\Pr_{x \sim p}[\tstar(x) = 1] \ge \frac{\dhsq(p,q)}{1800 \log(4/\dhsq(p,q))}$.
    \end{enumerate}
\end{remark}
\begin{proof}
    (1) holds immediately by replacing all notation in their original proof with the corresponding notation for continuous distributions.

    (2) holds immediately from their proof as well.
    
    (3) holds from the following observations about their proof (see their Section 3.2 for reference). In their ``Case 1'', observe that $\tau^*=1$. In their ``Case 2'', we use more details in their proof. In terms of their notation (their $\delta$ is our $\tau^*$), note that they choose a threshold of $1+\delta$ such that:
    \begin{align*}
        & \delta^2 \\
        & \ge \delta^2 \Pr[Y \ge \delta^2]\\
        & \ge \frac{\Ex[Y]}{13 \cdot (1 + \log(1/\Ex[Y]))}\intertext{Using their inequality that $\Ex[Y] \ge \dhsq(p,q)/4$:}
        & \ge \frac{\dhsq(p,q)}{52 \cdot (1 + \log(4/\dhsq(p,q)))}\\
        & \ge \frac{\dhsq(p,q)}{104 \log(4/\dhsq(p,q))}
    \end{align*}

    As their $\delta$ is our $\tau^*$, this implies $\tau^* \ge \sqrt{\frac{\dhsq(p,q)}{104 \log(4/\dhsq(p,q))}}$.

    (4) holds simply by:
    \begin{align*}
        & \Pr_{x \sim p}[\tstar(x) = 1] \\
        & \ge \dtv(\tstar p, \tstar q) \\
        & \ge \dhsq(\tstar p, \tstar q) \intertext{Using the result of \cref{thm:og-thresh}:}
        & \ge \frac{\dhsq(p,q)}{1800 \log(4/\dhsq(p,q))}
    \end{align*}
\end{proof}

For our work, we hope to leverage a channel $\tprime$ (not necessarily a proper thresholding function) that approximates $\tstar$, and conclude that $\tprime$ similarly preserves Hellinger distance like $\tstar$. 

\begin{theorem}[Modified Corollary 3.4 of \cite{pensia2023communication}; preservation of Hellinger distance for approximating thresholds]\label{thm:mod-thresh}
    For any continuous distributions $p,q$, let $\tstar \in \mathcal{T}^{\operatorname{thresh}}$ be the threshold yielded by \cref{thm:og-thresh}. Without loss of generality, suppose $\tstar$ thresholds by $1 + \tau^*$ for $\tau^* \ge 0$ (swap $p$ and $q$ otherwise). Then, consider a channel $\tprime$ an $(\alpha,\beta)$-approximation if it satisfies:
    \begin{enumerate}
        \item $\tprime(x) = 1$ only if $p(x)/q(x) \ge 1 + \alpha \cdot \tau^*$ for $0 < \alpha \le 1$.
        \item $\Pr_{x \sim p}[\tprime(x) = 1] \ge \beta \cdot \Pr_{x \sim p}[\tstar(x) = 1]$ for $0 < \beta \le 1$.
    \end{enumerate}
    For any such $(\alpha,\beta)$-approximation $\tprime$, the following holds:
    \begin{equation}
        1 \le \frac{\dhsq(p,q)}{\dhsq(\tprime p, \tprime q)}  \le \frac{\dhsq(p,q)}{\left(\sqrt{\Pr_{x \sim p}[\tprime(x)=1]} - \sqrt{\Pr_{x \sim q}[\tprime(x)=1]} \right)^2} \le \frac{3744 k'}{\alpha^2 \beta},
    \end{equation}
    where $k' = \log(4/\dhsq(p,q))$.
\end{theorem}
\begin{proof}
    The first part of the inequality follows from data-processing inequality, as remarked in \cite{pensia2023communication}. The second part of the inequality follows by definition of Hellinger distance. For the remaining portion, \emph{we merely state adjustments for the proof of \cite{pensia2023communication} to include the necessary terms with $\alpha,\beta$}. 

    \textit{``Case 1'' of \cite{pensia2023communication}.} Analogous to their notation (but for continuous distributions), let $A_{2,\infty}$ be the subset of the domain where $p(x)/q(x) \ge 2$. Then, let $p' \triangleq \Pr_{x \sim p}[x \in A_{2,\infty}]$. As they argue, then $\dhsq(p,q) \le 4 p'$. We now compute:

    \begin{align*}
        &\left(\sqrt{\Pr_{x \sim p}[\tprime(x)=1]}-\sqrt{\Pr_{x \sim q}[\tprime(x)=1]}\right)^2\\
        & \ge \left(\sqrt{\Pr_{x \sim p}[\tprime(x)=1]}-\sqrt{\frac{1}{1 + \alpha \cdot \delta} \Pr_{x \sim p}[\tprime(x)=1]}\right)^2 \intertext{Recall for this case, $\delta=1$:}
        & = \left(1 - \sqrt{\frac{1}{1 + \alpha}}\right)^2 \cdot \Pr_{x \sim p}[\tprime(x)=1] \\
        & \ge \left(1 - \sqrt{\frac{1}{1 + \alpha}}\right)^2 \cdot \beta \cdot \Pr_{x \sim p}[\tstar(x)=1] \intertext{Observe that $\Pr_{x \sim p}[\tstar(x) = 1] = p'$ and use $p' \ge \dhsq(p,q)/4$:}
        & \ge \left(1 - \sqrt{\frac{1}{1 + \alpha}}\right)^2 \cdot \beta \cdot \frac{\dhsq(p,q)}{4}\\
        & \ge \left(1 - \sqrt{1-\frac{\alpha}{2}}\right)^2 \cdot \beta \cdot \frac{\dhsq(p,q)}{4}\\
        & \ge \left(\frac{\alpha}{4}\right)^2 \cdot \beta \cdot \frac{\dhsq(p,q)}{4}\\
        & = \alpha^2 \cdot \beta \cdot \frac{\dhsq(p,q)}{64}\\
        & \implies \frac{\dhsq(p,q)}{\left(\sqrt{\Pr_{x \sim p}[\tprime(x)=1]}-\sqrt{\Pr_{x \sim q}[\tprime(x)=1]}\right)^2} \le \frac{64}{\alpha^2 \beta}
    \end{align*}

    \textit{``Case 2'' of \cite{pensia2023communication}.} Adjusting their notation for continuous distributions, let $A_{1,2}$ be the subset of the domain where $p(x)/q(x) \in (1,2)$. They consider a random variable $X$ in terms of $\delta(x) \triangleq \frac{p(x)-q(x)}{q(x)}$, where $\Pr[X > \delta] = \Pr_{x \sim q}[x \in A_{1,2}, \delta(x) > \delta]$ and $\Pr[X=0] = 1 - \Pr_{x \sim q}[x \in A_{1,2}]$. This random variable is insightful because, as they argue, $\dhsq(p,q) \le 4 \Ex[X^2]$. $\tstar$ chooses to threshold at $1+\delta$ where $\delta = \argmax_{\delta} \delta^2 \Pr[X \ge \delta^2]$.  We now lower bound $\dhsq(\tprime p, \tprime q)$ as they lower bounded $\dhsq(\tstar p, \tstar q)$:

    \begin{align*}
        & \left(\sqrt{\Pr_{x \sim p}[\tprime(x)=1]} - \sqrt{\Pr_{x \sim q}[\tprime(x)=1]}\right)^2\\
        & \ge \left(\sqrt{\Pr_{x \sim p}[\tprime(x)=1]} - \sqrt{\frac{1}{1 + \alpha \delta} \cdot \Pr_{x \sim p}[\tprime(x)=1]}\right)^2\\
        & \ge \left(\sqrt{\beta \cdot \Pr_{x \sim p}[\tstar(x)=1]} - \sqrt{\beta \cdot \frac{1}{1 + \alpha \delta} \cdot \Pr_{x \sim p}[\tstar(x)=1]}\right)^2\\
        & = \beta \cdot \Pr_{x \sim p}[\tstar(x)=1]  \left( 1 - \sqrt{1 - \frac{\alpha \delta}{1 + \alpha \delta}}\right)^2 \\
        & \ge \beta \cdot \Pr_{x \sim p}[\tstar(x)=1] \cdot \left(\frac{\alpha \delta}{6}\right)^2 \\
        & = \frac{\alpha^2 \beta}{36} \cdot \Pr_{x \sim p}[\tstar(x)=1] \delta^2 \rtag{using $\sqrt{1-\frac{x}{1+x}} \le 1 - \frac{x}{6}$ for $0 \le x \le 1$} \\
        & \ge \frac{\alpha^2 \beta}{36} \cdot \Pr_{x \sim q}[\tstar(x)=1] \delta^2\\
        & = \frac{\alpha^2 \beta}{36} \cdot \Pr[X^2 \ge \delta^2] \delta^2\intertext{Using that $\delta = \argmax_{\delta} \delta^2 \Pr[X^2 \ge \delta^2]$ and their Lemma 3.7 (reverse Markov inequality):}
        & \ge \frac{\alpha^2 \beta}{36} \cdot \frac{\Ex[X^2]}{13 \cdot (1 + \log(1/\Ex[X^2]))} \\
        & \ge \frac{\alpha^2 \beta}{144} \cdot \frac{\dhsq(p,q)}{13 \cdot (1 + \log(4/\dhsq(p,q)))} \rtag{using $\Ex[X^2] \ge \dhsq(p,q)/4$} \\
        & \ge \frac{\alpha^2 \beta}{3744} \cdot \frac{\dhsq(p,q)}{\log(4/\dhsq(p,q))}\\
        & \implies \frac{\dhsq(p,q)}{\left(\sqrt{\Pr_{x \sim p}[\tprime(x)=1]} - \sqrt{\Pr_{x \sim q}[\tprime(x)=1]}\right)^2} \le \frac{3744 k'}{\alpha^2 \beta} \quad\qedhere 
    \end{align*}

\end{proof}

\subsection{Approximating Likelihood Thresholds for Log-Concave Mixtures}\label{sec:logapprox}
\newcommand{\Ksup}{K^{\textrm{supp}}}

With \cref{thm:mod-thresh} in hand, we will prove \cref{lemma:intervals}, showing that for any $p$ satisfying our assumptions and a translation $p_{-\Delta}$, there is a channel $\tprime$ that is an indicator of intervals of the domain and $(\alpha,\beta)$-approximates $\tstar$. This will ultimately help us prove that there always exists an interval that witnesses the Hellinger distance between $p$ and its translation; we encourage readers to look ahead to the statement of \cref{cor:hk-to-int} to see what we are working towards.

Let us define the likelihood ratio $r(x) \triangleq \frac{p(x)}{p(x+\Delta)}$ and a related function $t(x) \triangleq r(x) - 1$. Recall the first condition of $(\alpha,\beta)$-approximation: when $\tstar$ thresholds by $1+\tau^*$ we require $\tprime(x)=1$ only if $p(x)/q(x) \ge 1 + \alpha \tau^*$. In the language of our new functions, this is conveniently written as $\tprime(x)=1$ only if $t(x) \ge \alpha \tau^*$. Accordingly, we now prove a technical result using the structure of $t$ under our assumptions, that will enable both conditions of $(\alpha,\beta)$-approximation:

\begin{lemma}\label{lemma:intervals}
    Suppose $p$ is a mixture of $k$ centered/symmetric log-concave distributions. Let $r(x)$ and $t(x)$ be defined with respect to $p_{-\Delta}$, so $r(x) \triangleq \frac{p(x)}{p(x+\Delta)}$ and $t(x) \triangleq r(x)-1$.
     Consider parameters $\tmin,\delta$ where $0 < \tmin \le \min(\frac{1}{k},\frac{1}{2})$ and $0<\delta \le \min(\frac{\tmin^2}{k},\frac{1}{2})$. 
    Then, there exist universal constants $C_1,C_2,C_3,C_4 > 0$ where for any $\tau \in [\tmin,1]$, there exists a collection of $r \le C_1 \cdot k \log(1/(\delta \tmin))$ disjoint intervals $I = I_1 \cup \dots \cup I_r$ where $t(x) \ge C_2 \cdot \tau$ for all $x \in I$, and $\Pr_{X \sim p}[x \in I] \ge C_3 \cdot \Pr_{X \sim p}[t(x) \ge \tau] - C_4 \cdot \delta k/\tmin^2$.

\end{lemma}

\begin{proof}
    Our main hope of accomplishing this will be to show that we can approximate $t$ sufficiently well (for most mass of $p$) by a piecewise-constant function with a small number of pieces.
    Then, selecting the pieces with large enough values relative to $\tau$, we will hopefully obtain a set of intervals satisfying our goal. 
    Recall $p_i$ are the mixture components of $p$, and analogously $t_i(x) \triangleq \frac{p_i(x)}{p_i(x+\Delta)} - 1$.
    We now introduce approximations for each $p_i$ and $t_i$.

    \textbf{Approximating $p$.} Without loss of generality, suppose the mixture is centered around $0$.

    \begin{lemma}[Piecewise-constant decomposition of log-concave densities; implicit in Lemma 27 of \cite{chan2014efficient}]\label{lemma:p-decomp}
        
        Let $q$ be a log-concave distribution over $\R$.
        For any $0<\delta \le \frac{1}{2}$, there exists a function $\tilde{q}$ which is a piecewise-constant function over $\R$ consisting of $O(\log(\frac{1}{\delta}))$ pieces.
        The function $\tilde{q}$ approximates $q$ in the sense that $\tilde{q}(x) \le q(x)$ for all $x \in \R$, $\tilde{q}(x) \ge \frac{1}{2} \cdot q(x)$ whenever $\tilde{q}(x) > 0$, and $\Pr_{x \sim q}[\tilde{q}(x)>0] \ge 1 -\delta$ where $\tilde{q}(x)=0$ only in the first and last piece of $\tilde{q}$ (a prefix and suffix of $\R$, respectively).

    \end{lemma}

    \begin{proof}
        This is implicitly shown in Lemma 27 of \cite{chan2014efficient} (stage (a) of their proof). Note how their proof uses one parameter, $\eps$, that determines both the multiplicative error ($\frac{1}{2}$ in our case) and the poorly-approximated mass in the tail ($\delta$ in our case), but that it yields this lemma statement when decoupling these parameters. We now provide brief intuition of the proof idea. Without loss of generality, suppose $q$ has its mode at $0$ and let us focus only on approximating the right half of the domain $[0,\infty]$. For all non-negative $i$, consider the $i$-th region to be the subset of the domain where $x$ is non-negative and $q(x) \in (\frac{q(0)}{2^{i+1}},\frac{q(0)}{2^i}]$. Observe that each region forms an interval of the domain: let the $i$-th region be $[a_i,b_i)$, and let $\ell_i \triangleq b_i-a_i$ be the length of the interval for the $i$-th region. 
        
        First, we remark that $\ell_i$ is non-increasing. For sake of contradiction, if this were not true, then $\frac{q(a_i + \ell_i)}{q(a_i)} < \frac{q(a_{i+1}+\ell_i)}{q(a_{i+1})}$, but this would violate log-concavity. Then, we remark that the probability from the $0$-th region is at least $\frac{q(0) \cdot \ell_0}{2}$, while the total probability from all regions with $i\ge j$ is at most $\frac{2 q(0)\cdot \ell_0}{2^j}$. Hence, for $j = O(\log(1/\delta))$, at most $\delta$ fraction of mass comes from regions after the $j$-th region, and the previous regions may all be approximated by powers of $2$ from $q(0)$ to $\frac{q(0)}{2^j}$.
    \end{proof}
    
    We will approximate each $p_i$ with $\tilde{p}_i$ using parameter $\delta$: resulting in $O(\log(1/\delta))$ pieces.

    Let us say that $\tilde{p}_i$ is \emph{supported} at all values of $x$ where $\tilde{p}_i(x)$ is nonzero, and \emph{unsupported} at all values of $x$ corresponding to the two (first and last) pieces that are $0$. 
    This notion aligns with where $\tilde{p}_i$ would be supported were it to be rescaled to define a probability density.

  resulting in $O(\log(n))$ pieces.

    More generally, we define our approximation $\tilde{p}$ for the entirety of $p$ as $\tilde{p}(x) \triangleq \sum_{i \in [k]} w_i \tilde{p}_i(x)$.
    Notice that $\tilde{p}(x)$ is a piecewise-constant function of $O(k \log(\frac{1}{\delta}))$ pieces: as $x$ increases from $x = -\infty$ towards $\infty$, the value of $\tilde{p}(x)$ only changes when one of $\tilde{p}_i(x)$ changes.

 We call $\tilde{p}(x)$ \emph{valid} if all unsupported mixture components are negligible compared to $\tilde{p}(x)$:
\begin{definition}\label{def:valid}
    $\tilde{p}(x)$ is \emph{$\kappa$-invalid} at value $x \in \mathbb{R}$ if and only if there exists an $i \in [k]$ where $\tilde{p}_i(x)$ is unsupported and $w_i \cdot p_i(x) \ge \kappa \cdot \tilde{p}(x)$. Otherwise $\tilde{p}(x)$ is \emph{$\kappa$-valid}.
\end{definition}
For ease of reading, sometimes we just state valid/invalid where $\kappa$ is implied.
    
    \begin{claim}\label{rem:tilde-tight}
        If $\tilde{p}(x)$ is $\kappa$-valid, for $\kappa \le \frac{1}{k}$, then $p(x)/4 \le \tilde{p}(x) \le p(x)$.
    \end{claim}
    \begin{proof}
    
        The latter half $\tilde{p}(x) \le p(x)$ holds even if $\tilde{p}(x)$ is invalid, by definition. 
        
        For the first half of our claim, we will analyze terms involving $p_i$ differently depending on whether or not $\tilde{p}_i(x)$ is supported at a value of $x$. 
        For convenience, let $K^{\textrm{supp}}(x) \subseteq [k]$ denote the mixtures where $\tilde{p}_i(x)$ is supported, and $K^{\textrm{unsupp}}(x) \subseteq [k]$ denote the complement. 
        Then, we bound:
        \begin{align*}
            & p(x) - \tilde{p}(x) = \sum_{i} w_i \left( p_i(x) - \tilde{p}_i(x) \right) \\
            & = \left( \sum_{i \in K^{\textrm{supp}}(x)} w_i \left( p_i(x) - \tilde{p}_i(x) \right) \right) + \left( \sum_{i \in K^{\textrm{unsupp}}(x)} w_i \left( p_i(x) - \tilde{p}_i(x) \right) \right) \intertext{Using that each supported $\tilde{p}_i(x) \in [p_i(x)/2,p_i(x)]$:}
            & \le \left( \sum_{i \in K^{\textrm{supp}}(x)} \frac{w_i p_i(x)}{2} \right) + \left( \sum_{i \in K^{\textrm{unsupp}}(x)} w_i p_i(x) \right) = \frac{p(x)}{2} + \sum_{i \in K^{\textrm{unsupp}}(x)} \frac{w_i p_i(x)}{2} \\
            & \le \frac{p(x)}{2} + \sum_{i \in K^{\textrm{unsupp}}(x)} \frac{\kappa \cdot \tilde{p}(x)}{2} \le \frac{p(x)}{2} + \frac{k \cdot \kappa \cdot \tilde{p}(x)}{2} \rtag{using that $\tilde{p}$ is valid}\\
            & \le \frac{p(x)}{2} + \frac{\tilde{p}(x)}{2} \implies \tilde{p}(x) \ge \frac{p(x)}{2 \cdot (1 + \frac{1}{2})} \ge \frac{p(x)}{4} \rtag{using $\kappa \le \frac{1}{k}$} \quad\qedhere
        \end{align*}
    \end{proof}
    We will show that $\tilde{p}(x)$ is valid for most of the mass of $p$, and that these valid regions correspond to a small number of disjoint intervals:
    \begin{claim}\label{claim:bound-inval}
        If $\kappa \le \frac{1}{k}$, then $\Pr_{X \sim p}[\tilde{p}(x) \textrm{ is invalid}] \le O(\frac{\delta}{\kappa})$
    \end{claim}
    \begin{proof}
        Let $S \subset \mathbb{R}$ be the values of $x$ where $\tilde{p}(x)$ is invalid.

        By definition, the total mass where $\tilde{p}(x)$ is invalid can be written as:
        
        \begin{equation}
            \int_{x \in S} \left(\left( \sum_{i \in K^{\textrm{supp}}(x)} w_i \cdot p_i(x) \right) + \left( \sum_{i \in K^{\textrm{unsupp}}(x)} w_i \cdot p_i(x) \right)\right) \diff x \label{step:total-p-invalid}
        \end{equation}
    The latter summation of \cref{step:total-p-invalid} is upper bounded by:
    \begin{align}
        & \int_{x \in S} \left( \sum_{i \in K^{\textrm{unsupp}}(x)} w_i \cdot p_i(x) \right) dx \nonumber \le \sum_{i=1}^{k} \int_{-\infty}^\infty \mathbbm{1}[\tilde{p}_i(x) \textrm{is unsupported}] \cdot  w_i \cdot p_i(x) dx \nonumber \\
        & \le \sum_{i=1}^k w_i \delta = \delta \label{step:delta-value}\rtag{using the guarantees for $\tilde{p}_i$ from \cref{lemma:p-decomp}}
    \end{align}

    Now, we bound the first summation of \cref{step:total-p-invalid}:
    \begin{align}
        & \int_{x \in S} \left( \sum_{i \in K^{\textrm{supp}}(x)} w_i \cdot p_i(x) \right) dx \nonumber \\
        & \le \int_{x \in S} 2 \tilde{p}(x) dx \nonumber \rtag{using $p_i(x)/2 \le \tilde{p}_i(x) \le p_i(x)$ when $i$ is supported} \intertext{Since $\tilde{p}(x)$ is invalid, there must be an $i \in K^{\textrm{unsupp}}(x)$ where $w_i \cdot p_i(x) \ge \kappa \cdot  \tilde{p}(x)$:}
        & \le \int_{x \in S} \frac{2}{\kappa} \left( \sum_{i \in K^{\textrm{unsupp}}(x)} w_i \cdot p_i(x) \right) dx \nonumber \\
        & \le \frac{2 \delta}{\kappa} = O\left(\frac{\delta}{\kappa}\right) \label{step:delta-value2}\rtag{using the previous bound on this summation in \cref{step:delta-value}}
    \end{align}

    Combining \cref{step:delta-value,step:delta-value2} yields $\Pr_{X \sim p}[\tilde{p}(x) \textrm{is invalid} ] \le O(\frac{\delta}{\kappa})$.
\end{proof}
 Moreover, the regions where $\tilde{p}$ is valid is the union of a small number of intervals:
 \begin{claim}\label{claim:valid-contig}
     The subset of $\mathbb{R}$ where $\tilde{p}(x)$ is $\kappa$-valid, is the union of at most $O(k \log(\frac{1}{\delta}))$ disjoint intervals.
 \end{claim}
 \begin{proof}
     For convenience, we use $\mathcal{D}_{\tilde{p}}$ to denote the set of intervals that correspond to the domain of each piece of $\tilde{p}$.
     Recall that $|\mathcal{D}_{\tilde{p}}| \le O(k \log(\frac{1}{\delta}))$.
     Also, recall our definition of invalidation that $\tilde{p}(x)$ is only $\kappa$-invalid if there is a $j \in [k]$ where $\tilde{p}_j(x)$ is unsupported and $w_j \cdot p_j(x) \ge \kappa \cdot \tilde{p}(x)$.

     For a naive analysis, observe that we are examining the domain after removing all regions of the domain where $\tilde{p}(x)$ is invalid.
     Generally, if we were to remove some number $Z$ of intervals from the domain, then the resulting subset of the domain is at most $Z+1$ intervals.
     This enables a simple analysis: for every pair of interval $\mathcal{I} \in \mathcal{D}_{\tilde{p}}$ and index $j \in [k]$, the distribution $p_j$ can only invalidate one interval among $\mathcal{I}$ (because $p_j$ is unimodal and $\tilde{p}$ is constant within $\mathcal{I}$).
     Thus, the subset of $\mathbb{R}$ where $\tilde{p}(x)$ is valid corresponds to at most $|\mathcal{D}_{\tilde{p}}| \cdot k + 1 \le O(k^2 \log(\frac{1}{\delta}))$ intervals.
     
     We will improve upon this by a factor of $k$ with a more careful argument.
     Let us study how a distribution $p_j$ may invalidate part of an interval $\mathcal{I} \in \mathcal{D}_{\tilde{p}}$.
     If the maximum value of $p_j$ is attained before the start of $\mathcal{I}$,\footnote{This claim is proven in general for log-concave $k$-mixtures, where the proof would be slightly simplified if we decided to leverage the centering.}
     then by unimodality of $p_j$, $j$ can only make a prefix of $\mathcal{I}$ invalid.
     Similarly, if the maximum value of $p_j$ is attained after $\mathcal{I}$, then $j$ can only make a suffix of $\mathcal{I}$ invalid.
     Meaning, if we ignore invalidations that occur from $p_j$ having a maxima inside $\mathcal{I}$, then $\tilde{p}$ is valid for everything in $\mathcal{I}$ that is not contained in the largest invalidating prefix or the largest invalidating suffix.
     Thus, when ignoring invalidation that occurs from such $p_j$, the subset of $\mathbb{R}$ where $\tilde{p}(x)$ is valid corresponds to a number of disjoint intervals that is at most $|\mathcal{D}_{\tilde{p}}|$.
     Finally, if we now consider for each $j$ the piece of $\tilde{p}$ that contains the maxima of $p_j$, and invalidate the one interval that $p_j$ invalidates (or possibly no interval), the number of non-deleted intervals of the domain increases by at most $1$.
 In total, the region where $\tilde{p}$ is valid is the union of $\le |\mathcal{D}_{\tilde{p}}| + k \le O(k \log(\frac{1}{\delta}))$ disjoint intervals.
 \end{proof}

\textbf{Approximating $t$.} Our last component will introduce our approximation for $t$, defined with respect to an approximation of each $t_i$:
  \begin{lemma}[$\tilde{t}_i$ decomposition]\label{lemma:t-decomp}
        For any log-concave distribution $q$ and values $0 < t_\textrm{low} \le t_{\textrm{high}} < \infty$, there exists a function $\tilde{t}$ over $\mathbb{R}$ that is piecewise-constant over $O(\log(\frac{2 \cdot t_{\textrm{high}}}{t_{\textrm{low}}}))$ pieces.
        The function $\tilde{t}$ approximates $t(x) \triangleq \frac{q(x)}{q(x+\Delta)} - 1$ in the sense that $\tilde{t}$ is within a factor of $2$ of $t(x)$ when $t(x)\in(t_\textrm{low},t_\textrm{high})$, $\tilde{t}(x)=0$ when $t(x)< t_\textrm{low}$, and $\tilde{t}(x)=t_\textrm{high}$ when $t(x)\ge t_\textrm{high}$.
    \end{lemma}
    \begin{proof}
        It is sufficient to show $t(x)$ is monotone by showing $\log(r(x))$ is monotone, as then $t(x) \triangleq 2^{\log(r(x))} - 1$ is monotone. 
        Recall that any log-concave distribution $q(x)$ can be written as $e^{-V(x)}$ where $V$ is a convex function.
        Then, $\log(r(x)) \triangleq V(x+\Delta)-V(x)$ which is monotone by convexity of $V$. 
As $t(x)$ is monotone, we can obtain this decomposition by setting $\tilde{t}(x)$ accordingly when it is smaller than $t_{\textrm{low}}$ or larger than $t_{\textrm{high}}$, and to the $O(\log(\frac{2 \cdot t_{\textrm{high}}}{t_{\textrm{low}}}))$ powers of $2$ in between.
    \end{proof}
    We set each $\tilde{t}_i$ using \cref{lemma:t-decomp} with $t_{\textrm{low}}=\tmin^2$, $t_{\textrm{high}}=1$, and $q = p_i$: resulting in $O(\log(1/\tmin))$ pieces. 
    We combine all $\tilde{t}_i(x)$ to produce $\tilde{t}(x)$, our approximation for $t(x)$, and show that it is a good approximation and piecewise-constant for a small number of pieces:
\begin{definition}[$\tilde{t}$ approximation]
    $\tilde{t}(x) \triangleq \sum_{i \in [k]} \tilde{t}_i(x) \cdot \frac{w_i \tilde{p}_i(x)}{\tilde{p}(x)}$
\end{definition}
\begin{remark} \label{rem:t-bound}
    $\tilde{t}(x) \le 1$
\end{remark}
\begin{proof}
Each $\tilde{t}_i(x) \le 1$ from \cref{lemma:t-decomp} with $t_{\textrm{high}}=1$.
So,
\begin{equation*}
    \tilde{t}(x) \triangleq \sum_{i \in [k]} \tilde{t}_i(x) \cdot \frac{w_i \tilde{p}_i(x)}{\tilde{p}(x)} = \sum_{i \in K^{\textrm{supp}}(x)} \tilde{t}_i(x) \cdot \frac{w_i \tilde{p}_i(x)}{\tilde{p}(x)}  \le \sum_{i \in K^{\textrm{supp}}(x)} \frac{w_i \tilde{p}_i(x)}{\tilde{p}(x)}  = 1.\quad\qedhere
\end{equation*}
\end{proof}
We show $\tilde{t}$ is constant and is a good approximation for $t$ in intervals %
where all $x$ are non-negative, $\tilde{p}$ is valid, all $\tilde{p}_i$ are constant, and all $\tilde{t}_i$ are constant:\footnote{
We note that before this, nothing has required that $p$ is a mixture of centered/symmetric components, only that its components are log-concave.
Now we will leverage how the components are centered.}
\begin{claim}\label{claim:good-t-tilde}
    For any interval where all $x \ge 0$, all $\tilde{p}_i(x)$ are constant, $\tilde{p}(x)$ is $\kappa$-valid for $\kappa \le \frac{\tmin^2}{k}$, and all $\tilde{t}_i$ are constant, then $\tilde{t}(x) = \Theta(1) \cdot \min\left(t(x), 1\right) - O\left(\tmin^2\right)$.

\end{claim}
\begin{proof}
    We begin by noting simple equivalent forms of $t(x)$:
    \begin{align}
        t(x) & \triangleq \frac{p(x)}{p(x + \Delta)} - 1 = \sum_{i=1}^k \frac{w_i \cdot p_i(x)}{p(x+\Delta)} - 1 \nonumber \\
        & = \sum_{i=1}^k \frac{w_i \cdot p_i(x) - w_i \cdot p_i(x+\Delta)}{p(x+\Delta)} \label{step:t-form1}\\
        & = \sum_{i=1}^k \frac{\left( \frac{p_i(x)}{p_i(x+\Delta)} - 1\right) \cdot  w_i \cdot p_i(x+\Delta)}{p(x+\Delta)}\nonumber  \\
        & = \sum_{i=1}^k \frac{t_i(x) \cdot  w_i \cdot p_i(x+\Delta)}{p(x+\Delta)} \label{step:t-form2}
    \end{align}
    We will mostly use forms \cref{step:t-form1} and \cref{step:t-form2}, noting also that equality holds for each summand, so we may define the summation with some summands in one form and some in the other form.

 Throughout this proof, we will utilize how when $x\ge 0$ all summands are non-negative due to the mixture being centered at $0$.
 For example, $\tilde{t}(x)$ would approximate $t(x)$ if we could show each summand in $\tilde{t}(x)$ multiplicatively approximates the corresponding summand in $t(x)$, but this would not hold if the summands could be positive and negative, as is the case if $p$ is not a mixture of centered components.

    With all the pieces in place, we are ready to show that $\tilde{t}(x)$ is a good approximation of $t(x)$. 
    We will proceed by analyzing two cases.
    First, when $p(x+\Delta) \ge \Omega(1) \cdot p(x)$, then we can well-approximate each summand in $t(x)$.
    Otherwise, when $p(x+\Delta) \ll p(x)$, then $t(x) \ge 1$, and we will show our summation will also be $\Omega(1)$, which sufficiently well-approximates $t$.

    \textbf{Case 1: $p(x+\Delta) \ge \frac{1}{16} p(x)$.}

We will drop from the summation $t(x)$ the indices corresponding to unsupported components of the mixture, and components for which $t_i$ is small; we claim that this does not affect the value of $t(x)$ significantly:

    \begin{remark}\label{rem:und-negl}
        $\sum_{i \in K^{\textup{unsupp}}(x)} \frac{w_i p_i(x) - w_i p_i(x+\Delta)}{p(x+\Delta)} \le 16\tmin^2$
        \footnote{$\sum_{i \in K^{\textrm{unsupp}}(x)} \frac{w_i p_i(x) - w_i p_i(x+\Delta)}{p(x+\Delta)} \le \sum_{i \in K^{\textrm{unsupp}}(x)} \frac{w_i p_i(x)}{p(x+\Delta)} \le 16 \cdot \sum_{i \in K^{\textrm{unsupp}}(x)} \frac{w_i p_i(x)}{p(x)} \le  16 \cdot \sum_{i \in K^{\textrm{unsupp}}(x)} \frac{w_i p_i(x)}{\tilde{p}(x)} \le  16 \cdot \sum_{i \in K^{\textrm{unsupp}}(x)} \kappa \le  16 \tmin^2$ where the second step used $p$ is unimodal, the penultimate step used $\tilde{p}$ is $\kappa$-valid, and the last step used $\kappa \le \frac{\tmin^2}{k}$.}
    \end{remark}\

    \begin{remark}\label{rem:small-ti-negl}
        $\sum_{i \textrm{ s.t. } t_i(x) \le \tmin^2} \frac{t_i(x) \cdot w_i \cdot p_i(x+\Delta)}{p(x+\Delta)} \le \tmin^2$
        \footnote{$\sum_{i \textrm{ s.t. } t_i(x) \le \tmin^2} \frac{t_i(x) \cdot w_i \cdot p_i(x+\Delta)}{p(x+\Delta)} \le \tmin^2 \cdot \sum_{i \textrm{ s.t. } t_i(x) \le \tmin^2} \frac{w_i \cdot p_i(x+\Delta)}{p(x+\Delta)} \le \tmin^2$}
    \end{remark}

Hence, 
\begin{align*}
t(x) &= \sum_{\substack{i \in K^{\textrm{supp}}(x)\\t_i(x) > \tmin^2}} w_i\cdot \frac{p_i(x) - p_i(x+\Delta)}{p(x+\Delta)} + O\left(\tmin^2\right).
\end{align*}
The denominator in the sum, $p(x+\Delta) = \Theta(1) \cdot \tilde{p}(x)$, first by the assumption that $p(x+\Delta) \ge \frac{1}{16} p(x)$ (the upper bound $p(x+\Delta) \le p(x)$ is immediate since we have assumed $x \ge 0$), and by the fact that $x$ is valid so $\tilde{p}(x) = \Theta(1) \cdot p(x)$ by \cref{rem:tilde-tight}.
We argue that for each term $i$ in the above summation, 

\begin{subclaim}\label{sub:diff-to-prod}
    $p_i(x) - p_i(x+\Delta) = \Theta(1)\cdot \tilde{t}_i(x)\cdot \tilde{p}_i(x)$
\end{subclaim}

\begin{proof}
    \textbf{Case (i): $t_i(x) \in (\tmin^2,1]$.}
    First, from \cref{lemma:t-decomp}, for terms where $t_i(x) \in (\tmin^2,1]$, $\tilde{t}_i(x)$ is a multiplicative constant-factor approximation of $t_i(x)$.
Hence by \cref{step:t-form2} we can write 
\[
p_i(x)-p_i(x+\Delta)
= t_i(x) p_i(x+\Delta)
= \Theta(1)\cdot \tilde{t}_i(x)  \cdot p_i(x+\Delta).
\]
Now, $t_i(x) \le 1$, implying that $p_i(x+\Delta) \ge \frac{1}{2} p_i(x)$.
Since $x \ge 0$ we always have $p_i(x+\Delta) \le p_i(x)$.
Furthermore, since $i$ is supported, $p_i(x) = \Theta(1)\cdot \tilde{p}_i(x)$.
Hence $p_i(x) - p_i(x+\Delta) = \Theta(1) \cdot \tilde{t}_i(x) \cdot \tilde{p}_i(x)$.

\textbf{Case (ii): $t_i(x) > 1$.}
Next, for the remaining terms where $1 < t_i(x) = \frac{p_i(x)}{p_i(x+\Delta)}-1$, we have by re-arranging that $p_i(x) > 2p_{i}(x+\Delta)$ and therefore $p_i(x) - p_i(x+\Delta) = \Theta(1) \cdot p_i(x)$.
Further, since $i \in \Ksup(x)$, then $\tilde{p}_i(x) = \Theta(1) \cdot p_i(x)$.
Therefore, using that $\tilde{t}_i(x)=1$ when $t_i(x)>1$:
\begin{align*}
p_i(x)-p_i(x+\Delta)
&= \Theta(1) \cdot \tilde{p}_i(x) = 
\Theta(1) \cdot \tilde{t}_i(x) \cdot \tilde{p}_i(x) \quad\qedhere
\end{align*} 
\end{proof}
Putting this together, \cref{sub:diff-to-prod} results in:
\begin{align*}
t(x) &= \sum_{\substack{i \in \Ksup(x)\\ t_i(x) > \tmin^2}} w_i \cdot
\Theta(1)\cdot \frac{\tilde{t}_i(x) \tilde{p}_i(x)}{p(x+\Delta)} + O\left(\tmin^2\right) \intertext{Using our assumption $p(x+\Delta) \ge \frac{1}{16}p(x)$ and \cref{rem:tilde-tight} from validity of $\tilde{p}(x)$:}
&= \sum_{\substack{i \in \Ksup(x)\\ t_i(x) > \tmin^2}} w_i \cdot
\Theta(1)\cdot \frac{\tilde{t}_i(x) \tilde{p}_i(x)}{\tilde{p}(x)} + O\left(\tmin^2\right)  \\
&= \Theta(1)\cdot \tilde{t}(x) + O\left(\tmin^2\right).\rtag{using $\tilde{t}_i(x)=0$ when $t_i(x) < \tmin^2$ or $i \notin \Ksup(x)$}
\end{align*}

    \textbf{Case 2: $p(x+\Delta) < \frac{1}{16} p(x)$.} Observe that $\tilde{t}(x) \le 1$ as in \cref{rem:t-bound}, and that if $p(x+\Delta) < \frac{1}{2}p(x)$ then $t(x)\ge 1$.
 Thus, to show $\tilde{t}(x) = \Theta(1) \cdot \min(t(x),1) - O(\tmin^2)$ it is sufficient to show $\tilde{t}(x) = \Omega(1)$ in this case. Our main intuition is that for $p(x+\Delta)$ to be much smaller than $p(x)$, then most of the mass must correspond to large $t_i(x)$ and accordingly our weighted sum of $\tilde{t}_i(x)$ will also be large. We now analyze the value of $\tilde{t}(x)$:

    \begin{align*}
        \tilde{t}(x) & \triangleq \sum_{i \in K^{\textrm{supp}}(x)} \tilde{t}_i(x) \cdot \frac{w_i \tilde{p}_i(x)}{\tilde{p}(x)} \\
        \intertext{Let us focus on the contribution from summands with large $t_i(x)$ as we believe it must be significant when $p(x+\Delta)$ is small:}
        & \ge \sum_{i \in K^{\textrm{supp}}(x)} \mathbbm{1}_{t_i(x) \ge 1} \cdot \tilde{t}_i(x) \cdot \frac{w_i \tilde{p}_i(x)}{\tilde{p}(x)}  = \sum_{i \in K^{\textrm{supp}}(x)} \mathbbm{1}_{t_i(x) \ge 1} \cdot \frac{w_i \tilde{p}_i(x)}{\tilde{p}(x)} \\
        \intertext{Additionally, because $\tilde{p}(x)$ is valid and all $\tilde{p}_i(x)$ are supported, we can convert from our approximations of $p$ and $p_i$ to the actual terms:}
        & \ge \Omega(1) \cdot \frac{1}{p(x)} \cdot  \sum_{i \in K^{\textrm{supp}}(x)} \mathbbm{1}_{t_i(x) \ge 1} \cdot w_i p_i(x)\\
        \intertext{At this point, we just need to lower bound the total mass from supported $p_i$ having $t_i(x)\ge 1$.
        Note that we can lower bound the total mass from all supported $p_i$ as $\sum_{i \in K^{\textrm{supp}}(x)} w_i p_i(x) \ge \tilde{p}(x) \ge p(x)/4$ by \cref{rem:tilde-tight}. Then, if at least $p(x)/8$ mass came from supported $p_i$ with $t_i(x)\le 1$, it would hold that $p(x+\Delta) \ge \frac{1}{16} p(x)$: violating our casework. Accordingly, we know $\sum_{i \in K^{\textrm{supp}}(x)} \mathbbm{1}_{t_i(x) \ge 1} \cdot w_i p_i(x) \ge p(x)/8$. Using this, we finish by:}
        & \ge \frac{1}{2 p(x)} \cdot \frac{p(x)}{8} \ge \Omega(1) \quad\qedhere
    \end{align*}

\end{proof}
\textbf{Concluding the desired set of intervals. } Finally, our proof of \cref{lemma:intervals} concludes by considering all intervals satisfying the conditions of \cref{claim:good-t-tilde}: $x \ge 0$, all $\tilde{p}_i(x)$ are constant, $\tilde{p}(x)$ is $\kappa$-valid, and all $\tilde{t}_i(x)$ are constant. Recall that we seek to find a collection of $r$ disjoint intervals $I = I_1 \cup \dots \cup I_r$ where: (i) $r = O(k \log(n))$, (ii) $\Pr_{X \sim p}[x \in I] \ge \Omega(1) \cdot \Pr_{X \sim p}[t(x) \ge \tau] - O(\delta k /\tmin^2)$, and (iii) $t(x) \ge \Omega(\tau)$ for all $x \in I$. We will choose $I_1 \cup \dots \cup I_r$ to be the subset of the intervals from \cref{claim:good-t-tilde} where $\tilde{t}(x) \ge C_1 \cdot \tau$ for a particular $C_1 > 0$.

We have yet to choose the parameter $\kappa$. We set $\kappa = \frac{\tmin^2}{k}$ as it is the largest value that lets us use \cref{claim:good-t-tilde}.

By \cref{claim:valid-contig} we know all $\kappa$-valid mass consists of $O(k \log(1/\delta))$ disjoint intervals.
As all $\tilde{p}_i$ and $\tilde{t}_i$ only change at most $O(k \log(1/(\delta \tmin))$ times in total, the number of disjoint intervals we are considering is thus $O(k \log(1/(\delta \tmin)))$.
Since we choose a subset of these intervals, $r = O(k \log(1/(\delta \tmin)))$: satisfying (i).

Let us observe how restricting to $x \ge 0$ does not limit us much.
For any negative value $x_- < 0$ where $t(x_-)>\tau$, note how there is a mapping to $x_+ \triangleq -x_-$ which is positive and $t(x_-) \le t(x_+)$ because $p$ is symmetric and unimodal, meaning $t(x_-)=\frac{p(x_-)}{p(x_- + \Delta)}-1 = \frac{p(x_+)}{p(x_- + \Delta)}-1 \le \frac{p(x_+)}{p(x_+ + \Delta)}-1 = t(x_+)$.
Thus, $\Pr_{X \sim p}[t(x) \ge \tau \cap x\ge 0] \ge \frac{1}{2} \cdot \Pr_{X \sim p}[t(x) \ge \tau]$. For any $x$ satisfying $t(x)\ge \tau$, $x\ge0$, and $\tilde{p}(x)$ is valid, then \cref{claim:good-t-tilde} will imply $\tilde{t}(x) \ge \Omega(\tau) - O(\tmin^2)$. Without loss of generality, suppose $1/\tmin$ is at least a sufficiently large constant, then we could conclude $\tilde{t}(x) \ge \Omega(\tau)$ under our conditions. If $1/\tmin$ is not this large, we can simply consider the guarantees of this lemma for a small enough $\tmin$ (that is still a constant bounded away from $0$), and see that it implies the lemma for large $\tmin$. So, since $\tilde{t}(x) \ge \Omega(\tau)$, if we set $C_1$ sufficiently small then $x$ will be in our collection $I$. We may then conclude
\begin{align*}
    & \Pr_{X \sim p}[x \in I]\\
    & \ge \Pr_{X \sim p}[t(x) \ge \tau \cap x \ge 0 \cap \tilde{p}(x) \textrm{ is valid}] \\
    & \ge \Pr_{X \sim p}[t(x) \ge \tau \cap x \ge 0] - \Pr_{X \sim p}[\tilde{p}(x) \textrm{ is invalid}] \\
    & \ge \frac{1}{2} \Pr_{X \sim p}[t(x) \ge \tau] - O\left(\frac{\delta}{\kappa}\right) \rtag{\cref{claim:bound-inval}}\\
    & = \frac{1}{2} \Pr_{X \sim p}[t(x) \ge \tau] - O(\delta k/\tmin^2),
\end{align*}

satisfying (ii).

Moreover, by \cref{claim:good-t-tilde} we know $\tilde{t}(x) = \Theta(1) \cdot \min(t(x),1)+O(\tmin^2)$, implying $t(x) \ge \Omega(1) \cdot (\tilde{t}(x) - O(\tmin^2))$. As before, without loss of generality we may suppose $1/\tmin$ is at least a sufficiently large constant, so the $O(\tmin^2)$ term is negligible compared to the $\tilde{t}(x) \ge C_1 \tau \ge C_1 \tmin$ term. So, there will be a $C_2 > 0$ where any such value of $x$ in one of these ranges where $\tilde{t}(x) \ge C_1 \cdot \tau$, must then satisfy $t(x) \ge C_2 \cdot \tau$, hence implying our final condition (iii) that $t(x) \ge \Omega(\tau)$ for all $x \in I$. This completes the proof of \cref{lemma:intervals}.
\end{proof}

We may now combine \cref{thm:mod-thresh} and \cref{lemma:intervals}. \cref{thm:mod-thresh} shows that an approximate likelihood threshold channel approximately preserves Hellinger distance, and \cref{lemma:intervals} yields that an interval statistic can approximate a likelihood threshold channel:

\begin{corollary}\label{cor:hk-to-int}
    Suppose $p$ is a mixture of $k$ centered/symmetric log-concave distributions. For any $\mu$ and $\Delta \ge 0$, there exists an interval that approximately preserves the Hellinger distance between $p_{\mu}$ and $p_{\mu-\Delta}$. In particular, there is an interval $I^* \triangleq [\mu+a,\mu+b]$, for $0 \le a < b$, where

    \begin{align*}
        & \left( \sqrt{\Pr_{x \sim p_{\mu}}[x \in [\mu+a,\mu+b]]} -  \sqrt{\Pr_{x \sim p_{\mu-\Delta}}[x \in [\mu+a,\mu+b]]} \right)^2 \\
        & \ge \Omega(1) \cdot \frac{\dhsq(p_{\mu},p_{\mu-\Delta})}{k \log(4k/\dhsq(p_{\mu},p_{\mu-\Delta})) \cdot \log(4/\dhsq(p_{\mu},p_{\mu-\Delta}))}.
    \end{align*}
\end{corollary}
\begin{proof}
    Consider the optimal thresholding channel $\tstar$ from \cref{thm:og-thresh} with thresholding parameter $\tau^*$ and properties discussed in \cref{remark:opt-props}.
    We hope to approximate this channel with $\tprime(x) \triangleq \mathbbm{1}_{x \in [\mu+a,\mu+b]}(x)$ in the $(\alpha,\beta)$ sense that \cref{thm:mod-thresh} implies would approximately preserve Hellinger distance.

    To achieve $(\alpha,\beta)$-approximation, we must satisfy: (1) $\tprime(x)=1$ only if $p_\mu(x)/p_{\mu-\Delta}(x) \ge 1 + \alpha \tau^*$ for $0 < \alpha \le 1$, and (2) $\Pr_{x \sim p}[\tprime(x)=1] \ge \beta \cdot \Pr_{x\sim p}[\tstar(x)=1]$ for $0 < \beta \le 1$.

    If we invoke \cref{lemma:intervals} with $\tmin \le \tau^*$ and use $\tau=\tau^*$, then all intervals will satisfy $t(x) \ge \Omega(\tau^*)$. Recall by \cref{remark:opt-props} (3) that $\tau^* \ge \sqrt{\frac{\dhsq(p_{\mu},p_{\mu-\Delta})}{104 \log(4/\dhsq(p_{\mu},p_{\mu-\Delta}))}}$. So, we may set $\tmin = \min \left(\sqrt{\frac{\dhsq(p_{\mu},p_{\mu-\Delta})}{104 \log(4/\dhsq(p_{\mu},p_{\mu-\Delta}))}}, \frac{1}{k} \right)$, and thus we approximate with $\alpha = \Omega(1)$. 

    Also, recall by \cref{remark:opt-props} (4) that $\Pr_{x \sim p}[\tstar(x)=1] \ge \frac{\dhsq(p_{\mu},p_{\mu-\Delta})}{1800 \log(4/\dhsq(p_{\mu},p_{\mu-\Delta}))}$. Accordingly, if we invoke \cref{lemma:intervals} with $\delta= C \cdot \frac{\dhsq(p_{\mu},p_{\mu-\Delta}) \cdot \tmin^2}{1800 \log(4/\dhsq(p_{\mu},p_{\mu-\Delta})) \cdot k}$ for sufficiently small $C$, then $\Pr_{x \sim p}[x \in I] \ge \Omega(1) \cdot \Pr_{x \sim p}[\tstar(x)=1]$. Hence, choosing $I^*$ to be the interval with the most probability mass among those yielded by \cref{lemma:intervals}:
    \begin{align*}
        & \Pr_{x \sim p}[x \in I^*] \ge \frac{1}{r} \cdot \Pr_{x \sim p}[x \in I]  \ge \Omega(1) \cdot \frac{1}{r} \cdot \Pr_{x \sim p}[\tstar(x)=1]\\
        & \ge \Omega(1) \cdot \frac{1}{r} \cdot \Pr_{x \sim p}[\tstar(x)=1] \ge \Omega\left(\frac{1}{k \log(1/(\tmin \cdot \delta))}\right) \cdot \Pr_{x \sim p}[\tstar(x)=1]\intertext{Using $\delta = C \cdot \frac{\dhsq(p_{\mu},p_{\mu-\Delta}) \cdot \tmin^2}{1800 \log(4/\dhsq(p_{\mu},p_{\mu-\Delta})) \cdot k}$:}
        & \ge \Omega\left(\frac{1}{k \cdot (1 + \log(1/\dhsq(p_{\mu},p_{\mu-\Delta})) + \log(k) + \log(1/\tmin))}\right) \cdot \Pr_{x \sim p}[\tstar(x)=1] \intertext{Using $\tmin = \min\left(\sqrt{\frac{\dhsq(p_{\mu},p_{\mu-\Delta})}{104 \log(4/\dhsq(p_{\mu},p_{\mu-\Delta}))}}, \frac{1}{k}\right)$:}
        & \ge \Omega\left(\frac{1}{k \cdot (1 + \log(1/\dhsq(p_{\mu},p_{\mu-\Delta}))+\log(k))}\right) \cdot \Pr_{x \sim p}[\tstar(x)=1]\\
    \end{align*}
    Thus, we approximate with $\beta = \Omega\left(\frac{1}{k \cdot \log(4k/\dhsq(p_{\mu},p_{\mu-\Delta}))}\right)$. Using \cref{thm:mod-thresh}, we conclude:

    \begin{align*}
        & \left( \sqrt{\Pr_{x \sim p_{\mu}}[x \in [\mu+a,\mu+b]]} -  \sqrt{\Pr_{x \sim p_{\mu-\Delta}}[x \in [\mu+a,\mu+b]]} \right)^2 \\
        & \ge \frac{\alpha^2 \beta \cdot \dhsq(p_{\mu},p_{\mu-\Delta})}{3744 \log(4/\dhsq(p_{\mu},p_{\mu-\Delta}))} \\
        & \ge \Omega\left(\frac{1}{k \cdot \log(4k/\dhsq(p_{\mu},p_{\mu-\Delta}))}\right) \cdot \frac{\dhsq(p_{\mu},p_{\mu-\Delta})}{3744 \log(4/\dhsq(p_{\mu},p_{\mu-\Delta}))}\\
        & \ge \Omega(1) \cdot \frac{\dhsq(p_{\mu},p_{\mu-\Delta})}{k \log(4k/\dhsq(p_{\mu},p_{\mu-\Delta})) \cdot \log(4/\dhsq(p_{\mu},p_{\mu-\Delta}))}.\quad\qedhere
    \end{align*}
    
\end{proof}

\subsection{Obtaining an algorithm for mean estimation}\label{sec:alg}
Our goal is to conclude that for any potential estimate $\muh$ where $|\mu - \muh|$ is sufficiently large, we can detect this in the form of an interval statistic, where the number of samples within $[\muh-b,\muh-a]$ is noticeably different from the number of samples within $[\muh+a,\muh+b]$ for $0 \le a < b$: hence witnessing that the distribution is not symmetric around $\muh$. Then, any $\muh$ that does not have such a distinguishing interval statistic would be a sufficiently good estimate of $\mu$. Our algorithm will then search for a $\muh$ without such a distinguishing statistic. We formalize this with \cref{algo:ident}.

\begin{algorithm}[t]
    \caption{Identifiability Algorithm} \label{algo:ident}
    \hspace*{\algorithmicindent} 
    \begin{flushleft}
      {\bf Input:} testing parameter $\gamma$, and $\rho(l,r)$ is the number of samples within $[l,r]$ from $n$ samples  \\
      {\bf Output:} estimate $\muh$\\
      {\bf Description:} This (inefficient) algorithm will output any $\muh$ that passes all possible tests.
    \end{flushleft}
    \begin{algorithmic}[1]
    
    \Procedure{Test}{$\muh,a,b,\gamma$}:
    \State $L \gets \cnt(\muh - b,  \muh - a)$ \Comment{Count samples within $[\muh - b, \muh - a]$.}
    \State $R \gets \cnt(\muh+a,\muh + b)$  \Comment{Count samples within $[\muh+a,\muh + b]$.}
    
    \If{$\left|\sqrt{L} - \sqrt{R} \right| > \gamma$} 
    
            \Return FAIL 
    \Else
    
        \Return PASS 
    \EndIf

    \EndProcedure 

    \Procedure{Estimate}{$\gamma$}
    
        \Return any $\muh$ that passes $\operatorname{Test(\muh,a,b,\gamma)}$ for all values of $0 \le a < b$
    \EndProcedure
    \end{algorithmic}
\end{algorithm}

Leveraging \cref{cor:hk-to-int} lets us almost immediately show that poor $\muh$ will have a test that captures almost all Hellinger distance:

\begin{corollary}\label{cor:test-exists}
    Suppose $p$ is a mixture of $k$ centered/symmetric log-concave distributions.  For any scalar $\muh \in \mathbb{R}$, let us define $\Delta \triangleq |\mu - \muh|$. Then, there is a test around $\muh$ that preserves the Hellinger distance: there are values $0 \le a < b$ where

    \begin{align*}
        & \left| \sqrt{\Pr_{x \sim p_{\mu}}[x \in [\muh-b,\muh-a]]} -  \sqrt{\Pr_{x \sim p_{\mu}}[x \in [\muh+a,\muh+b]]} \right| \\
        & \ge \Omega(1) \cdot \sqrt{\frac{\dhsq(p_{\mu},p_{\mu-2\Delta})}{k \log(4k/\dhsq(p_{\mu},p_{\mu-2\Delta})) \cdot \log(4/\dhsq(p_{\mu},p_{\mu-2\Delta}))}}.
    \end{align*}
\end{corollary}
\begin{proof}
    Without loss of generality, consider $\muh < \mu$. Let $a_{2 \Delta},b_{2 \Delta}$ be the values of $a,b$ yielded by \cref{cor:hk-to-int} when used on distributions $p_\mu,p_{\mu-2\Delta}$. For our test, we will choose values $a^*,b^*$ where $a^* \triangleq \Delta + a_{2\Delta}$ and $b^* \triangleq \Delta + b_{2\Delta}$. Then, our corollary immediately holds from realizing $\Pr_{x \sim p_{\mu}}[x \in [\muh+a^*,\muh+b^*]] = \Pr_{x \sim p_{\mu}}[x \in [\mu+a_{2\Delta},\mu+b_{2\Delta}]]$ and $\Pr_{x \sim p_{\mu}}[x \in [\muh-b^*,\muh-a^*]] = \Pr_{x \sim p_{\mu-2\Delta}}[x \in [\mu+a_{2\Delta},\mu+b_{2\Delta}]]$.
\end{proof}

What remains is to show is that if we choose $\gamma$ correctly, then with high probability, $\mu$ will pass all tests with the empirical samples, and all bad $\muh$ will fail some test with the empirical samples:

\begin{theorem}\label{thm:ident-log}
    Suppose $p$ is a mixture of $k$ centered/symmetric log-concave distributions. There exists some universal constants $\cgam,\cdist \ge 1$, 
    where if
    \begin{equation*}
        \Delta^* \triangleq \omega_p \left(\cdist \cdot \frac{k}{n} \cdot \log(2n/\delta) \cdot \log^2(2n) \right)
    \end{equation*}
    then with probability $1-\delta$ the output $\muh$ of \cref{algo:ident} with $\gamma = \cgam \cdot \sqrt{\log(2n/\delta)}$ will satisfy $|\mu - \muh| \le \Delta^*/2$.
\end{theorem}
\begin{proof}
Our proof will begin by stating required uniform convergence guarantees via \cref{lem:vclemma,claim:test-conc}, then in \cref{claim:good-test-ok} we show all tests centered at $\mu$ will pass, and in \cref{claim:test-happens} we show sufficiently poor $\muh$ will fail tests.

First, we will leverage normalized uniform convergence guarantees that are tighter for $f$ with small $\Ex[f]$. This is a standard tool, and we will use the particular form of Lemma 1 of \cite{dasgupta2007general} for convenience (which itself references \cite{vapnik2015uniform,bousquet2003introduction}). The following directly holds from Lemma 1 of \cite{dasgupta2007general} and the Sauer-Shelah lemma (e.g. see Lemma 1 on page 184 of \cite{bousquet2003introduction}). (This is not the only way to prove this style of uniform convergence; for example, you could also use refinements of the DKW inequality like in \cite{bartl2023variance,blanchard2024tight,reeve2024short}, or sample compression schemes.)

\begin{lemma}[Normalized uniform convergence; implied by Lemma 1 of \cite{dasgupta2007general}]
\label{lem:vclemma}
Let $X_1, \ldots, X_n$ be i.i.d. random variables taking their values in $\mathcal X$. Assume that the class $\mathcal{F}$ of $\{0, 1\}$-valued functions has the {\sc VC} dimension $d$. Then there is a numerical constant $C > 0$ such that for any $\delta \in (0, 1)$, with probability at least $1 - \delta$, for all $f \in \mathcal{F}$,
\begin{equation}
\label{eq:firstuniformbound}
\left|\sum_{i = 1}^n(f(X_i) - \Ex[ f(X_i)])\right| \le C \cdot \left(\sqrt{\left(\sum_{i = 1}^n\Ex[f(X_i)]\right)\left(d\log(n) + \log\left(\frac{2}{\delta}\right)\right)}+ d\log(n) + \log\left(\frac{2}{\delta}\right)\right)
\end{equation}
\end{lemma}
    Let $\rho(l,r)$ denote the random variable corresponding to the number of samples within $[l,r]$ from $n$ samples. We show how for all indicators of intervals, $|\sqrt{\rho(l,r)}-\sqrt{\Ex[\rho(l,r)]}|$ is small:
    \begin{claim}\label{claim:test-conc}
        With probability $1-\delta$, for all intervals $[l,r]$ it holds that:
        \begin{equation*}
            \left| \sqrt{\rho(l,r)} - \sqrt{\Ex[\rho(l,r)]} \right| \le O\left(\sqrt{\log(2n/\delta)}\right)
        \end{equation*}
    \end{claim}
    \begin{proof}
    We will bound this in two ways. Consider $|\sqrt{x}-\sqrt{y}|$ for non-negative $x,y$. Roughly, if $x\approx y$, then the quantity of interest is almost bounded by $\frac{|x-y|}{\sqrt{y}}$. More concretely, if (i) $x \ge y$ then $|\sqrt{x}-\sqrt{y}| = \int_{0}^{x-y} \frac{1}{\sqrt{t + y}} \diff t \le \frac{x-y}{\sqrt{y}}$. Otherwise, if (ii) $\tfrac{y}{2} \le x < y$, then $|\sqrt{x}-\sqrt{y}| = \int_{0}^{y-x} \frac{1}{\sqrt{y - t}} \diff t \le \frac{y-x}{\sqrt{y/2}}$. In our remaining case, (iii) $x < \tfrac{y}{2}$, then $|\sqrt{x}-\sqrt{y}| \le \sqrt{y} \le \frac{2 \cdot (y-x)}{\sqrt{y}}$. In all cases, $|\sqrt{x}-\sqrt{y}| \le \frac{2 |y-x|}{\sqrt{y}}$ resulting in the first argument of the next step. Additionally, by concavity, $|\sqrt{x}-\sqrt{y}| \le \sqrt{|x-y|}$ which may be much better when $y$ is small, giving us the second argument of the next step:
        \begin{align*}
            & \left| \sqrt{\rho(l,r)} - \sqrt{\Ex[\rho(l,r)]} \right| \le \min\left(\frac{2 \cdot |\rho(l,r) - \Ex[\rho(l,r)]|}{\sqrt{\Ex[\rho(l,r)]}}, \sqrt{|\rho(l,r)-\Ex[\rho(l,r)]|}\right)\intertext{We use that the uniform convergence guarantee \cref{eq:firstuniformbound} of \cref{lem:vclemma} holds with probability $1-\delta$, noting the VC dimension of interval indicators is $d=2$. Then, for all $[l,r]$,  $|\Ex[\rho(l,r)] - \rho(l,r)| \le O(1) \cdot \left(\sqrt{\Ex[\rho(l,r)] \cdot \log(2n/\delta)}+ \log(2n/\delta) \right)$, so:}
            & \le \min\Bigg(\frac{O(1) \cdot \left(\sqrt{\Ex[\rho(l,r)] \cdot \log(2n/\delta)}+ \log(2n/\delta) \right)}{\sqrt{\Ex[\rho(l,r)]}}, \\
            & \sqrt{O(1) \cdot \left(\sqrt{\Ex[\rho(l,r)] \cdot \log(2n/\delta)}+ \log(2n/\delta) \right)}\Bigg)\intertext{Consider using the first argument of the minimum when $\Ex[\rho(l,r)] \ge \log(2n/\delta)$ and the second argument when $\Ex[\rho(l,r)] < \log(2n/\delta)$, then we conclude:}
            & \le O\left(\sqrt{\log(2n/\delta)}\right) \quad\qedhere
        \end{align*}
        \end{proof}
        This type of uniform convergence guarantee will be sufficient to show that all tests which need to pass will pass, and every poor $\muh$ will have a test that fails. First, we show that with the correct $\mu$ all tests will pass:
        \begin{claim}\label{claim:good-test-ok}
            Under the test convergence event of \cref{claim:test-conc}, there exists some constant $\cgam \ge 1$ where \cref{algo:ident} will pass all tests centered at $\mu$ if $\gamma \ge \cgam \cdot \sqrt{\log(2n/\delta)}$.
        \end{claim}
        \begin{proof}
            For any test centered at $\mu$, our claim follows by:
            \begin{align*}
                & \left| \sqrt{\rho(\mu-b,\mu-a)} - \sqrt{\rho(\mu+a,\mu+b)} \right| \\
                & \le \bigg| \left|\sqrt{\rho(\mu-b,\mu-a)} - \sqrt{\Ex[\rho(\mu-b,\mu-a)]}\right| + \left|\sqrt{\rho(\mu+a,\mu+b)} - \sqrt{\Ex[\rho(\mu+a,\mu+b)]}\right|\\
                & + |\sqrt{\Ex[\rho(\mu-b,\mu-a)]} - \sqrt{\Ex[\rho(\mu+a,\mu+b)]}|\bigg|\\
                & = \left|\sqrt{\rho(\mu-b,\mu-a)} - \sqrt{\Ex[\rho(\mu-b,\mu-a)]}\right| + \left|\sqrt{\rho(\mu+a,\mu+b)} - \sqrt{\Ex[\rho(\mu+a,\mu+b)]}\right| \\ 
                & \le O\left(\sqrt{\log(2n/\delta)}\right) \qquad(\text{\cref{claim:test-conc}})
            \end{align*}
            
        \end{proof}
        Let us set $\gamma = \cgam \cdot \sqrt{\log(2n/\delta)}$ for the value of $\cgam$ yielded by \cref{claim:good-test-ok}. Then, for any poor $\muh$ there will be a test that fails:
        \begin{claim}\label{claim:test-happens}
            Under the test convergence event of \cref{claim:test-conc}, there exists some universal constant $\cdist \ge 1$ (as a function of $\cgam$), where \cref{algo:ident} will fail some test centered at $\muh$, for every:
            \begin{equation*}
                |\mu - \muh| > \Delta^*/2
            \end{equation*}
        \end{claim}
        \begin{proof}
        In the proof of this claim, we will mostly leverage our lower bound on $\dhsq(p_\mu,p_{\muh})$ from the conditions of this theorem, and the existence of a test that preserves this Hellinger distance via \cref{cor:test-exists}. To start, for any $0 \le a < b$ it holds:
        \begin{align*}
            & \left| \sqrt{\rho(\muh-b,\muh-a)} - \sqrt{\rho(\muh+a,\muh+b)} \right|\\
            & \ge \bigg| \left|\sqrt{\Ex[\rho(\muh-b,\muh-a)]} - \sqrt{\Ex[\rho(\muh+a,\muh+b)]}\right| - \left|\sqrt{\rho(\muh-b,\muh-a)} - \sqrt{\Ex[\rho(\muh-b,\muh-a)]}\right| \\
            & - \left|\sqrt{\rho(\muh+a,\muh+b)} - \sqrt{\Ex[\rho(\muh+a,\muh+b)]}\right|\bigg|\\
            & \ge \left|\sqrt{\Ex[\rho(\muh-b,\muh-a)]} - \sqrt{\Ex[\rho(\muh+a,\muh+b)]}\right| - O\left(\sqrt{\log(2n/\delta)}\right) \rtag{\cref{claim:test-conc}} \intertext{Let $\Delta \triangleq |\mu - \muh|$. Then, if we set $a$ and $b$ to the corresponding values from \cref{cor:test-exists}:}
            & \ge \Omega(1) \cdot \sqrt{\frac{n \cdot \dhsq(p_{\mu},p_{\mu-2\Delta})}{k \log(4k/\dhsq(p_{\mu},p_{\mu-2\Delta})) \cdot \log(4/\dhsq(p_{\mu},p_{\mu-2\Delta}))}} - O\left(\sqrt{\log(2n/\delta)}\right) \\
            & \ge \Omega(1) \cdot \sqrt{\frac{n \cdot \dhsq(p_{\mu},p_{\mu-2\Delta})}{k \log^2(4k/\dhsq(p_{\mu},p_{\mu-2\Delta}))}} - O\left(\sqrt{\log(2n/\delta)}\right) \intertext{Since this is non-decreasing in $\dhsq(p_{\mu},p_{\mu-2\Delta})$, we use our lower bound on $\dhsq(p_{\mu},p_{\mu-2\Delta})$ from $2\Delta \ge \Delta^*$ and the assumed lower bound from this theorem for $\dhsq(p_\mu,p_{\mu-\Delta})$ when $|\Delta|\ge \Delta^*$. Note that the value of this assumption was chosen so that the first term of the previous step will be sufficiently larger than the latter term. Hence:}
            & \ge \Omega(1) \cdot \sqrt{\frac{n \cdot \left( \cdist \cdot \frac{k}{n} \cdot \log(2n/\delta) \cdot \log^2(2n) \right)}{k \log^2\left(\tfrac{4k}{\cdist \cdot \frac{k}{n} \cdot \log(2n/\delta) \cdot \log^2(2n)}\right)}} - O\left(\sqrt{\log(2n/\delta)}\right)\\
            & = \Omega(1) \cdot \sqrt{\frac{ \cdist \cdot \log(2n/\delta) \cdot \log^2(2n)}{\log^2\left(\tfrac{4n}{\cdist  \cdot \log(2n/\delta) \cdot \log^2(2n)}\right)}} - O\left(\sqrt{\log(2n/\delta)}\right)\\
            & \ge \Omega(1) \cdot \sqrt{\frac{ \cdist \cdot \log(2n/\delta) \cdot \log^2(2n)}{O(1) \cdot \max\left( \log(1/\cdist), \log(2n)\right)^2}} - O\left(\sqrt{\log(2n/\delta)}\right) \intertext{If we choose a $\cdist \ge 1$, then:}
            & \ge \Omega(1) \cdot \sqrt{\frac{ \cdist \cdot \log(2n/\delta)}{O(1)}} - O\left(\sqrt{\log(2n/\delta)}\right) \ge \left(\Omega(1) \cdot \sqrt{\cdist} - O(1)\right) \cdot \sqrt{\log(2n/\delta)}\intertext{If we choose $\cdist$ to be sufficiently large in terms of $\cgam$, we obtain the desired:}
            & > \cgam \cdot \sqrt{\log(2n/\delta)} = \gamma
        \end{align*}
            Meaning, the corresponding test centered at $\muh$ will fail.
        \end{proof}
        Hence, this completes the proof of Theorem 2.20.
    \end{proof}

    Unfortunately, this algorithm is both (i) inefficient, and (ii) needs to know a confidence parameter $\delta$ to compute $\gamma$, which may be undesirable. We note that (i) can be partially remedied as \cref{algo:ident} can be simulated naively in $O(n^4)$ time by observing that tests are only determined by the set of samples inside the two intervals $[\muh-b,\muh-a]$ and $[\muh+a,\muh+b]$, so we may naively iterate over all sets in $O(n^4)$ time. We do not discuss this in-depth because we soon introduce a more nuanced algorithm that runs in near-linear time. For the parameter dependence raised in (ii), we note that this could be resolved by choosing the $\muh$ that passes all tests with the smallest value of $\gamma$. We state this corollary next for completeness. Our near-linear time algorithm will also leverage a similar idea to avoid any parameter dependence.

\begin{corollary}\label{cor:paramfree-ident-log}
    Consider a modified version of \cref{algo:ident} with $\gamma \ge 0$ set to be the smallest value such that at least one $\muh$ passes all tests. We now attain a similar guarantee to \cref{thm:ident-log} without needing to choose $\gamma$. Suppose $p$ is a mixture of $k$ centered/symmetric log-concave distributions. There exists some universal constant $\cdist \ge 1$, where if
    \begin{equation*}
        \Delta^* \triangleq \omega_p \left( \cdist \cdot \frac{k}{n} \cdot \log(2n/\delta) \cdot \log^2(2n)\right)
    \end{equation*}

    then with probability $1-\delta$ the output $\muh$ of the modified \cref{algo:ident} will satisfy $|\mu - \muh| \le \Delta^*/2$.  
\end{corollary}
\begin{proof}
    Note by \cref{thm:ident-log} if $\gamma = \cgam \log(2/\delta)$ then at least one $\muh$ will pass all tests, and all $\muh$ that pass the test satisfy the desired condition on $|\mu-\muh|$. Since at least one $\muh$ will pass all tests, then the modified algorithm will choose a value of $\gamma$ where $\gamma \le \cgam \log(2/\delta)$. Moreover, the set of $\muh$ that pass the tests with this $\gamma$ will be a subset of the $\muh$ that pass with the larger value, so they will also satisfy the condition on $|\mu - \muh|$.
\end{proof}

\subsubsection{Designing a Near-Linear Time Algorithm}\label{sec:fast}
Our analysis of the inefficient \cref{algo:ident} only leveraged the existence of significant tests for poor $\muh$, such as those shown in \cref{cor:test-exists}. For a faster algorithm, we will show the existence of tests with structure that makes the tests easier to find. First, we define one such structure for a test:

\begin{definition}[$\ell$-heavy test]
    An $\ell$-heavy test is a test where of the two intervals being compared, the interval with more samples contains exactly $\ell$ samples. Moreover, the endpoints of the larger interval are exactly the first and last of these $\ell$ samples (inclusive).
\end{definition}

We will show that it is sufficient to consider only $\ell$-heavy tests where $\ell$ is a power of $2$.
Second, we hope to efficiently find all $\ell$-heavy tests for a fixed $\gamma$ and $\ell$. We will observe that if a distribution is symmetric/unimodal and a possible estimate $\muh$ fails some test because one interval has significantly more samples than another, then we may conclude that $\mu$ is strictly on the side of the larger interval. Hence, it is sufficient to find the leftmost $\muh$ that fails an $\ell$-heavy test because the interval on its left is too populated, and similarly the rightmost $\muh$ that fails an $\ell$-heavy test because the interval on its right is too populated. We are able to compute this for a fixed $\ell$ and $\gamma$ in $O(n)$ time with a sweep-line algorithm. Third, we prove that it is sufficient to consider only $O(\log(n))$ values of $\gamma$, and binary search in $O(\log(\log(n)))$ iterations for the smallest such $\gamma$ having a $\muh$ that doesn't fail any discovered test. In total, we will obtain an $O(n \log(n) \log(\log(n)))$ time algorithm by considering only $O(\log(n))$ values of $\ell$, employing an $O(n)$ time sweep-line subroutine, and doing $O(\log(\log((n))))$ iterations of binary search over $\gamma$. We present the sweep-line subroutine in \cref{algo:sweep}, and the entire estimation procedure in \cref{algo:fast}.

We now prove our guarantees for \cref{algo:fast}, which are of the same flavor as \cref{thm:ident-log,cor:paramfree-ident-log} but running in $O(n \log(n) \log(\log(n)))$ time:

\fastthm*
\begin{proof}
    Most of our proof will be able to reuse claims from the proof of \cref{thm:ident-log}. Let us focus on the uniform convergence event of \cref{claim:test-conc} that holds with $1-\delta$ probability. Using a claim similar to \cref{claim:good-test-ok}, we will show that no test will incorrectly fail for large enough $\gamma$. For example, if the left interval has significantly more samples than the right interval, then $\mu < \muh$.
    \begin{claim}\label{claim:mod-no-wrong-fail}
        Under the test convergence event of \cref{claim:test-conc}, there exists some constant $\cgam \ge 1$ where all failing tests will have correct conclusions if $\gamma \ge \cgam \cdot \sqrt{\log(2n/\delta)}$.
    \end{claim}
    \begin{proof}
        We can analyze how different the empirical test value is from the quantity with the expectations:
        \begin{align*}
            & \left| \left( \sqrt{\Ex[\cnt(\muh - b, \muh - a)]} - \sqrt{\Ex[\cnt(\muh+a,\muh+b)]} \right) - \left( \sqrt{\cnt(\muh - b, \muh - a)} - \sqrt{\cnt(\muh+a,\muh+b)} \right) \right|\\
            & = \left| \left( \sqrt{\Ex[\cnt(\muh - b, \muh - a)]} - \sqrt{\cnt(\muh - b, \muh - a)}  \right) + \left( \sqrt{\cnt(\muh+a,\muh+b)} - \sqrt{\Ex[\cnt(\muh+a,\muh+b)]} \right) \right|\\
            & \le \left| \sqrt{\Ex[\cnt(\muh - b, \muh - a)]} - \sqrt{\cnt(\muh - b, \muh - a)}  \right| + \left| \sqrt{\cnt(\muh+a,\muh+b)} - \sqrt{\Ex[\cnt(\muh+a,\muh+b)]} \right| \\
            & \le O(1) \cdot \sqrt{\log(2n/\delta)} \rtag{using \cref{claim:test-conc}}
        \end{align*}
        Hence, for sufficiently large $\cgam$, if $\left( \sqrt{\cnt(\muh - b, \muh - a)} - \sqrt{\cnt(\muh+a,\muh+b)} \right) > \cgam \cdot \sqrt{\log(2n/\delta)}$, then we may conclude $\Ex[\cnt(\muh - b, \muh - a)] > \Ex[\cnt(\muh+a,\muh+b)]$, meaning $\mu < \muh$ since our distribution is symmetric and unimodal. The same can be said for if $\left( \sqrt{\cnt(\muh+a,\muh+b)} - \sqrt{\cnt(\muh - b, \muh - a)}\right) > \cgam \cdot \sqrt{\log(2n/\delta)}$, then we may conclude $\Ex[\cnt(\muh+a,\muh+b)] > \Ex[\cnt(\muh - b, \muh - a)]$, meaning $\mu > \muh$ since our distribution is symmetric and unimodal.
    \end{proof}
    This has shown that none of our test's conclusions will be incorrect with sufficiently large $\cgam$. Our next goal is to show that \cref{algo:fast} will consider a value of $\gamma$ that is close to considering the desired $\cgam \cdot \sqrt{\log(2n/\delta)}$:

    \begin{claim}\label{claim:good-gamma}
        For any value $\gamma_0>0$, \cref{algo:fast} has a value $\gamma \in \operatorname{List_\gamma}$ whose tests all evaluate the same as they would for some $\gamma' \in [\gamma_0,2\gamma_0)$.
    \end{claim}
    \begin{proof}
        Let $\gsmall \triangleq \tfrac{1}{\sqrt{n}}$ and $\glarge \triangleq \sqrt{n+1}$. Recall that $\operatorname{List_\gamma}$ contains the values: $\gsmall$, $\glarge$, and $2^i \cdot \gsmall$ for all integer values of $i\ge 1$ where $2^i \cdot \gsmall < \glarge$. For any $v \in [\gsmall,\glarge]$ the claim immediately holds. For $\gamma_0 \le \gsmall$, a test will fail if and only if the intervals have an unequal number of samples, so our claim will hold because $\operatorname{List_\gamma}$ contains $ \gsmall$. Finally, for $\gamma_0 \ge \glarge$, no test will fail because $|\sqrt{\cnt(\muh-b,\muh-a)}-\sqrt{\cnt(\muh+a,\muh+b)}| \le \sqrt{n} < \glarge$, so our claim will hold because $\operatorname{List_\gamma}$ contains $\glarge$.
    \end{proof}

    Thus, using \cref{claim:good-gamma}, let us consider the value $\gamma^* \in \operatorname{List_\gamma}$ that evaluates tests identically to $\cround \cdot \cgam \cdot \sqrt{\log(2n/\delta)}$, for $\cround \in [1,2)$. Note that all conclusions with $\gamma^*$ will be correct by \cref{claim:mod-no-wrong-fail}. We now show that there will be a failing $2^i$-heavy test for some value of $i$, for every $\muh$ with sufficiently large $|\mu - \muh|$:

    \begin{lemma}\label{claim:structure-test-happens}
        There exists some universal constant $\cdist \ge 1$ (as a function of $\cgam$), where under the test convergence event of \cref{claim:test-conc}, then some $2^i$-heavy test centered at $\muh$ will fail using $\gamma^*$, for every:
        \begin{equation*}
            |\mu - \muh| > \Delta^*/2
        \end{equation*}
    \end{lemma}
    \begin{proof}
        Looking into the previous proof of \cref{thm:ident-log}, by \cref{cor:test-exists} we knew there was a test centered at $\muh$ with $0 \le a < b$ that preserved Hellinger distance, and by \cref{claim:test-happens} we concluded that the test empirically fails under the uniform convergence event (the proof also implicitly shows that when the test fails, the interval with larger expectation will correctly have more samples empirically, so our analogous conclusion is valid). This proof still holds under our current theorem assumptions and using $\gamma^*$ (we are just not finished because the test is not necessarily a $2^i$-heavy test). 
        
        Let us consider the same test defined by $a,b$. Without loss of generality, consider $\muh < \mu$, so the right interval $[\muh + a, \muh + b]$ has more samples in expectation than the left interval $[\muh - b, \muh -a]$. Since the test empirically fails under the uniform convergence event, then certainly the right interval will have at least one sample. We note that any interval with a positive number of samples can be decomposed into two (possibly overlapping) intervals that each contain $2^i$ samples:

        \begin{claim}\label{claim:decomp}
            Consider an interval $[l,r]$ with $N \ge 1$ distinct samples inside the interval. There exist values $m_0 \le m_1$ where $[l,m_1]$ and $[m_0,r]$ both contain exactly $2^{\lfloor \log(N) \rfloor}$ samples.
        \end{claim}
        \begin{proof}
            This follows immediately from considering $[l,m_1]$ to be the longest interval containing exactly $2^{\lfloor \log(N) \rfloor}$ samples and starting at $l$, and considering $[m_0,r]$ to be the longest interval containing exactly $2^{\lfloor \log(N) \rfloor}$ samples and ending at $r$.
        \end{proof}

        We use the decomposition of \cref{claim:decomp} to consider two tests,\footnote{As an aside, we acknowledge the edge case where multiple samples have exactly the same value, so we cannot split into two tests via \cref{claim:decomp}. Observe that this occurs with probability $0$ unless $p$ contains an atom, which may only occur at its mode $\mu$. Since we have chosen $a$ such that $\muh + a \ge \mu$, this may only occur when our $a = \mu - \muh$. If less than half of the samples in $[\muh + a,\muh + b]$ occur at $\muh+a$, then \cref{claim:decomp} will successfully decompose into two intervals with $2^i$ samples. Otherwise, the following arguments will succeed with test $[\muh+a,\muh+a]$ that has at least $2^i$ samples compared to $[\muh-a,\muh-a]$ that has at most $1$ sample.}
        $[a,m_1]$ and $[m_0,b]$, where both contain $2^i$ samples and we are hoping one test will nearly be a good $2^i$-heavy test. Moving forward, we will show that the decomposition does yield a good test:

        \begin{claim}\label{claim:hel-split}
            Consider a subset $S$ of the domain, and the subsets $S_0,S_1 \subseteq S$ where $S_0 \cup S_1 = S$ and $\frac{p(x)}{q(x)} \ge 1$ for all $x \in S$. Then:
            \begin{equation*}
                \max_{i \in \{0,1\}} \left( \sqrt{\Pr_{x \sim p}[x \in S_i]} - \sqrt{\Pr_{x \sim q}[x \in S_i]} \right)^2 \ge \frac{1}{4} \cdot \left( \sqrt{\Pr_{x \sim p}[x \in S]} - \sqrt{\Pr_{x \sim q}[x \in S]}  \right)^2
            \end{equation*}
        \end{claim}
        \begin{proof}
            \begin{align*}
                & \max_{i \in \{0,1\}} \left( \sqrt{\Pr_{x \sim p}[x \in S_i]} - \sqrt{\Pr_{x \sim q}[x \in S_i]} \right)^2 \intertext{Let $i^* \in \{0,1\}$ be a value such that $|\Pr_{x \sim p}[x \in S_{i^*}] - \Pr_{x \sim q}[x \in S_{i^*}]| \ge \frac{1}{2} \cdot |\Pr_{x \sim p}[x \in S] - \Pr_{x \sim q}[x \in S]|$. Such an $i^*$ must exist by $p(x)/q(x)\ge 1$ for all $x \in S$:}
                & \ge \left( \sqrt{\Pr_{x \sim p}[x \in S_{i^*}]} - \sqrt{\Pr_{x \sim q}[x \in S_{i^*}]} \right)^2\\
                & \ge \left( \sqrt{\Pr_{x \sim q}[x \in S_{i^*}] + \frac{1}{2} \cdot |\Pr_{x\sim p}[x \in S]-\Pr_{x\sim q}[x \in S]|} - \sqrt{\Pr_{x \sim q}[x \in S_{i^*}]} \right)^2\\
                & \ge \left( \sqrt{\Pr_{x \sim q}[x \in S] + \frac{1}{2} \cdot |\Pr_{x\sim p}[x \in S]-\Pr_{x\sim q}[x \in S]|} - \sqrt{\Pr_{x \sim q}[x \in S]} \right)^2\\
                & \ge \left( \frac{1}{2} \cdot \left(\sqrt{\Pr_{x \sim q}[x \in S] +  |\Pr_{x\sim p}[x \in S]-\Pr_{x\sim q}[x \in S]|} - \sqrt{\Pr_{x \sim q}[x \in S]} \right) \right)^2\\
                & = \frac{1}{4} \cdot  \left(\sqrt{\Pr_{x \sim p}[x \in S]} - \sqrt{\Pr_{x \sim q}[x \in S]} \right)^2 \tag*{\qedhere} 
            \end{align*}
        \end{proof}

        Applying \cref{claim:hel-split} directly to \cref{cor:test-exists}, we get that one of $[a,m_1]$ and $[m_0,b]$ satisfy the guarantees of $a,b$ from \cref{cor:test-exists} up to a factor of $\frac{1}{4}$, and moreover this contains exactly $2^i$ samples. Using precisely the same proof as \cref{claim:test-happens} will yield our desired guarantee (note how the bound in terms of $\cgam$ in the original proof can be replaced by $\cround \cdot \cgam \le 2 \cgam$, which only changes constant factors). All that remains is that the test is not quite $2^i$-heavy, because although it contains exactly $2^i$ samples, its endpoints are not necessarily samples. This is easily remedied by contracting the interval to still contain $2^i$ samples, but have its starting endpoint be the leftmost sample inside and the ending endpoint be the rightmost sample inside. The test will still fail, because the heavier interval will not lose samples, and the lighter interval will not gain samples.
    \end{proof}

    This gives us a clear roadmap for finishing our proof. When we use $\gamma^*$, we know that all conclusions will be valid, and all sufficiently bad $\muh$ will be ruled out by failed $2^i$-heavy tests. When considering $\gamma > \gamma^*$, only a subset of the tests will fail, so certainly the binary search will end with a $\gamma \le \gamma^*$. Moreover, the values of $\muh$ that pass with $\gamma$ will only be a subset of the values that pass with $\gamma^*$, so we immediately have the desired bound on $|\muh - \mu|$.

    All the remains is to show that \cref{algo:sweep} correctly recovers the set of $\muh$ that pass $\ell$-heavy tests for a fixed $\ell$ and $\gamma$. Recall that it is sufficient to search for the rightmost $\muh$ that fails such a test because the right interval has much more samples, and the leftmost $\muh$ that fails such a test because the left interval has much more samples. Without loss of generality, we focus on the former:

    \begin{lemma}
        $\operatorname{BiggestLowerBound}([X_1,\dots,X_n],\gamma,\ell)$ computes the rightmost $\muh$ where an $\ell$-heavy test (with the heavier side being on the right) centered at $\muh$ fails with parameter $\gamma$.
    \end{lemma}
    \begin{proof}
        Recall that such an $\ell$-heavy test will have the right interval containing exactly $\ell$ samples, and its endpoints will be samples. So, the right interval will be $[X_i,X_{i+\ell-1}]$ for some $i \in \{1,\dots, n-\ell+1\}$. 
        
        Now, consider some left interval for the test. Recall that a test will fail if $\sqrt{R}-\sqrt{L}>\gamma$, where $R$ is the number of samples in the right interval and $L$ is the number of samples in the left interval. Since $R=\ell$, we conclude that a test will fail if and only if $L \le \lceil(\sqrt{\ell}-\gamma)^2 \rceil - 1$, denoted by $\operatorname{LeftCountCap}$ in \cref{line:leftcap}. There is also some structure for the best left interval: if the left interval could be moved to the right without including an additional sample, this would strictly improve $\muh$. So, the left interval must be $[X_j-(X_{i+\ell-1} - X_i),X_j)$ for some $j \le i$. Equivalently, for $l \le r$, there exists an $\ell$-heavy test with the right interval starting at $X_r$ (inclusive) and the left interval ending at $X_l$ (non-inclusive) if and only if the longest interval ending at $X_l$ (non-inclusive) containing at most $\operatorname{LeftCountCap}$ samples is at least as long as $X_{r+\ell-1} - X_r$. This will be the property our sweep-line crucially relies on. We informally refer to such a valid pairing as matching an $X_l$-left interval with an $X_r$-right interval. 

        We note two simple properties of the best matching:

        \begin{claim}\label{claim:right-prop}
            A $X_r$-right interval will not be in the best matching if there is an $r' > r$ where $X_{r' + \ell -1} - X_{r'} \le X_{r + \ell - 1} - X_r$.
        \end{claim}
        \begin{proof}
            Any valid matching including the $X_r$-right interval would also be valid with the $X_{r'}$-right interval, and the latter would have a larger $\muh$.
        \end{proof}
        
        The array $\operatorname{NonDominatedRightOption}$ tracks whether each $X_r$ has such a dominating $X_{r'}$, and $\operatorname{NonDominatedRightOption}[r]$ is true only if there is no such $r'$. 
        
        \begin{claim}\label{claim:left-prop}
            For a fixed $X_r$-right interval, the best matching will never include an $X_l$-left interval if there exists an $l < l' \le r$ where the $X_{l'}$-left interval is not shorter than the $X_l$-left interval.
        \end{claim}
        \begin{proof}
            Any valid matching with the $X_l$-left interval and the $X_r$-right interval would also be valid with the $X_{l'}$-left interval and the $X_r$ right interval, and the latter would have a larger $\muh$.
        \end{proof}

        We are now ready to explain the remaining aspects of the algorithm. Starting at \cref{line:rightloop}, we iterate over possible $X_i$-right intervals in increasing order of $i$. Before trying to match the $X_i$-right interval, we adjust our options for left intervals to match with. In $\operatorname{LeftStack}$, we are maintaining a stack of left intervals that are not dominated with respect to the property of \cref{claim:left-prop} (left intervals higher in the stack will correspond to $X_l$-left intervals with larger $l$ and shorter lengths). In \cref{line:removeleftprop}, we remove left intervals from $\operatorname{LeftStack}$ to maintain this property of the stack. By \cref{claim:right-prop}, it is permitted to only consider actually matching the $X_i$-right interval if $\operatorname{NonDominatedRightOption[i]}$ is true. In \cref{line:removeleftbig}, we note that if the top of $\operatorname{LeftStack}$ is too short to be matched with the $X_i$-right interval, then it will also be too short to be matched with any remaining non-dominated right intervals, so we may remove it from $\operatorname{LeftStack}$. Finally, in \cref{line:match}, we consider matching the $X_i$-right interval with the top left interval in $\operatorname{LeftStack}$ (if there is one). This left interval from the top of the stack is long enough to match with the $X_i$-right interval, and it is the rightmost such $X_l$-left interval that is sufficiently long. 

        By choosing the best $\mu_{\textrm{lower-bound}}$ of all matchings considered in \cref{line:match}, we find the largest $\muh$ failing a test of the desired structure.
    \end{proof}

    Thus, \cref{algo:fast} attains our desired guarantee.
\end{proof}

\begin{algorithm}[h]
    \caption{Fast Mean Estimation Algorithm} \label{algo:fast}
    \hspace*{\algorithmicindent} 
    \begin{flushleft}
      {\bf Input:} samples $X_1,\dots,X_n$ \\
      {\bf Output:} estimate $\muh$\\
      {\bf Description:} This $O(n \log(n) \log(\log(n)))$ time algorithm will output an estimate $\muh$ that passes tests based on a search over parameters $\gamma$ and $\ell$.
    \end{flushleft}
    \begin{algorithmic}[1]

    \Procedure{FixedGammaCheck}{$X_1,\dots,X_n$, $\gamma$} \Comment{Takes $O(n \log(n))$ time.}
        \State $\mu_{\textrm{lower-bound}} \gets -\infty$
        \State $\mu_{\textrm{upper-bound}} \gets \infty$
        \For{$\ell \in \{1, 2, \dots, 2^i, \dots, 2^{\lfloor\log(n)\rfloor}\}$} \Comment{Consider $O(\log(n))$ values of $\ell$.}
            \State $\mu_{\textrm{lower-bound}} \gets \max\left(\mu_{\textrm{lower-bound}},\operatorname{BiggestLowerBound}(X,\gamma,\ell) \right)$ \Comment{Takes $O(n)$ time.}
            \State $\mu_{\textrm{upper-bound}} \gets \min\left(\mu_{\textrm{upper-bound}},\operatorname{SmallestUpperBound}(X,\gamma,\ell) \right)$ \Comment{We did not explicitly define this function, but it is the same as $\operatorname{BiggestLowerBound}$ after reversing.}
        \EndFor
        \If{$\mu_{\textrm{lower-bound}} \le \mu_{\textrm{upper-bound}}$}
            \Return any $\muh$ inside $(\mu_{\textrm{lower-bound}},\mu_{\textrm{upper-bound}})$
        \Else 
            
            \Return FAIL \Comment{Using this $\gamma$, there was no $\muh$ that passed all tests.}
        \EndIf
    \EndProcedure
    
    \Procedure{Estimate}{$X_1,\dots,X_n$}
        \State $X_1,\dots,X_n \gets \operatorname{Sort}(X_1,\dots,X_n)$ \Comment{Sort in non-decreasing order in $O(n \log(n))$ time.}
        \State $\gsmall \gets \tfrac{1}{\sqrt{n}}$
        \State $\glarge \gets \sqrt{n+1}$
        \State $\operatorname{List_\gamma} \gets [\gsmall, 2 \cdot \gsmall, \dots, 2^i \cdot \gsmall, \dots, \glarge]$ 
        \Comment{List starting with $\gsmall$, and then containing $2^i \cdot \gsmall$ for $i\ge 1$ as long as $2^i \cdot \gsmall < \glarge$.}

        \State Binary search for the smallest $\gamma^* \in \operatorname{List_\gamma}$ where $\operatorname{FixedGammaCheck}(X,\gamma^*)$ returns a $\muh$ instead of failing. \Comment{$\operatorname{List_\gamma}$ contains $O(\log(n))$ values, so the binary search will try $O(\log(\log(n)))$ values of $\gamma$, with each check taking $O(n \log(n))$ time.}

        \Return $\operatorname{FixedGammaCheck}(X,\gamma^*)$
    \EndProcedure
    \end{algorithmic}
\end{algorithm}

\begin{algorithm}[!hp]
    \caption{Lower Bound Sweep-Line Algorithm} \label{algo:sweep}
    \hspace*{\algorithmicindent} 
    \begin{flushleft}
      {\bf Input:} sorted samples $X_1 \le \dots \le X_n$, thresholing parameter $\gamma$, heaviness parameter $\ell$ \\
      {\bf Output:} lower bound on $\mu$\\
      {\bf Description:} This $O(n)$ time algorithm will output the largest lower bound concluded by testing with parameter $\gamma$ with a right interval that contains exactly $\ell$ samples.
    \end{flushleft}
    \begin{algorithmic}[1]

    \Procedure{BiggestLowerBound}{$[X_1,\dots,X_n]$, $\gamma$, $\ell$}
        \State $\mu_{\textrm{lower-bound}} \gets -\infty$
        \If{$\sqrt{\ell} \le \gamma$}
            \Return $-\infty$ \Comment{No test could possibly fail.}
        \EndIf
        \State $\operatorname{LeftCountCap} \gets \lceil(\sqrt{\ell}-\gamma)^2 \rceil - 1$ \label{line:leftcap} \Comment{A test will fail if $\sqrt{R}-\sqrt{L} > \gamma$. Since $R=\ell$, we solve for the largest integer where if $L$ is at most this integer, then the test will fail.}
        \State $\operatorname{NonDominatedRightOption} \gets [\textrm{False}, \dots, \textrm{False}]$ \Comment{For $i \in \{1,\dots,n-\ell+1\}$, $\operatorname{NonDominatedRightOption}[i]$ is true if there is no interval containing $\ell$ samples that starts further to the right and is not longer.}
        \State $\operatorname{RightLength} \gets []$ \Comment{$\operatorname{RightLength}[i]$ will be the length of the interval starting at $X_i$ (inclusive) that contains $\ell$ samples, it will only be defined for $i \in \{1,\dots, n-\ell+1\}$.}
        \State $\operatorname{ShortestConsidered} \gets +\infty$ \Comment{We will consider $i$ in decreasing order and note the shortest interval containing $\ell$ samples we have yet seen.}
        \For{$i \in \{n-\ell+1,\dots,1\}$}
            \State $\operatorname{RightLength}[i] \gets X_{i+\ell-1} - X_i$
            \If{$\operatorname{RightLength}[i] < \operatorname{ShortestConsidered}$}
                \State $\operatorname{ShortestConsidered} \gets \operatorname{RightLength}[i]$
                \State $\operatorname{NonDominatedRightOption}[i] \gets \textrm{True}$
            \EndIf
        \EndFor
        \State $\operatorname{LeftLength} \gets []$ \Comment{$\operatorname{LeftLength}[i]$ will be the length of the longest interval ending at $X_i$ (non-inclusive) that contains at most $\operatorname{LeftCountCap}$ samples.}
        \State $\operatorname{LeftStack} \gets []$ \Comment{A stack of potential left intervals to match with right intervals. Items higher in the stack will have larger $i$ and shorter length (because otherwise, if it had larger $i$ and not shorter length, we would always prefer this interval and could remove the other).}
        \For{$i \in \{1,\dots,n-\ell+1\}$} \label{line:rightloop}
            \If{$i \le \operatorname{LeftCountCap}+1$}
                \State $\operatorname{LeftLength}[i] \gets \infty$
            \Else
                \State $\operatorname{LeftLength}[i] \gets X_i - X_{i-\operatorname{LeftCountCap}-1}$
            \EndIf
            \While{$\operatorname{LeftLength}[\operatorname{LeftStack.top()}] \le \operatorname{LeftLength}[i]$} \label{line:removeleftprop}
                \State $\operatorname{LeftStack.pop()}$
            \EndWhile
            \State $\operatorname{LeftStack.push(i)}$
            \If{$\operatorname{NonDominatedRightOption}[i]$}
                \While{$\operatorname{LeftLength}[\operatorname{LeftStack.top()}] \le \operatorname{RightLength}[i]$}\label{line:removeleftbig}
                    \State $\operatorname{LeftStack.pop()}$ \Comment{We cannot match this left interval with the right interval starting at $i$, nor any later $j>i$, so we may remove it.}
                \EndWhile
                \If{$\operatorname{LeftStack}$ is nonempty}\label{line:match}
                    \State $\mu_{\textrm{lower-bound}} \gets \max\left(\mu_{\textrm{lower-bound}}, \left(X_{\operatorname{LeftStack.top()}} + X_i\right)/2 \right)$
                \EndIf
            \EndIf
        \EndFor
        \Return $\mu_{\textrm{lower-bound}}$
    \EndProcedure
    \end{algorithmic}
\end{algorithm}

\section{Lower Bound for Adaptive Location Estimation of Symmetric, Unimodal Distributions}

We now aim to prove that it is not possible to adaptively attain the two-point testing rates if the distribution is only promised to be symmetric and unimodal. In our positive result, we focused on how indicators of intervals witness distance between log-concave mixtures and their translations. Looking inside this proof more, we leveraged how one could roughly threshold the likelihood ratio by looking at an interval of the domain.

In designing our hard instance, we seek to design a distribution where the likelihood ratio with its translation is large in regions that are very spaced apart. Moreover, if we consider a family of such distributions with different spacings, then we hope to show that it is impossible to attain the two-point testing rate. For a more visual depiction, consider the step distribution in \cref{fig:step}, which is a unimodal and symmetric distribution that resembles a collection of steps. Comparing this distribution with a slight translation in \cref{fig:step}, we see that the likelihood is strictly greater than $1$ in regions that are spaced apart. Our lower bound will consist of a family of distributions where the step width is random. Then, we will not know where to look for the spikes in the likelihood ratio. In fact, our proof will proceed by showing that a family of random step distributions is indistinguishable from a triangle with the same center. We then show that triangles have a much larger two-point testing lower bound than any step distribution in our family, concluding our proof.

\adaptivelb*
\begin{proof}
Let us define some relevant distributions in terms of a sample size $n \ge 1$, and parameter $0 < \eps < 1$ where $\frac{1}{2 \eps}$ is an integer.

\begin{definition}[Triangle Distribution]
     \begin{align*}
     \tri(x) = \begin{cases} 
          1-|x| & |x| \le 1 \\
          0 & |x| > 1 
       \end{cases}
       \end{align*}
\end{definition}

Before we define step distributions, let us define a helper function $s_w(x)$ which defines a function with three steps: 

\begin{definition} The function $s_w(x)$ has $0 \le w < \eps/2$ and is supported on $[0,\eps)$ such that:
    \begin{align*}
     s_w(x) = \begin{cases} 
          0 & 0 \le x<\eps/2-w\\
          \eps/2  & \eps/2-w \le x \le \eps/2+w \\
          \eps & \eps/2+w \le x < \eps
       \end{cases}
    \end{align*}
\end{definition}

Although not important yet, $s_w(x)$ was designed such that if we sample $w \sim \unif(0,\eps/2)$, then its marginal is identical to the line $f(x)=x$ on $[0,\eps)$. We now define the step distribution:

\begin{definition}[Step Distribution] Let $v$ be a vector of length $\frac{1}{2\eps}$, where each $v_i \in [0,\eps/2]$. The parameter $v_i$ informs the length of the $i$-th step:
     \begin{align*}
     \step_v(x) = \begin{cases} 
          1 - (i+1)\eps + s_{v_i}((i+1)\eps - |x|) & |x| \in [i\eps,(i+1)\eps) \textrm{ for } i \in \{0,\dots,\frac{1}{2\eps}-1\}\\
          1-|x| & \frac{1}{2} \le |x| < 1 \\
          0 & |x| \ge 1
       \end{cases}
    \end{align*}
\end{definition}

Now, consider the mixture where we sample $\frac{1}{2\eps}$ i.i.d. variables $v_i \sim \unif(0,\eps/2)$ and then receive $n$ samples from $\step_v(x)$:

\begin{definition}[Mixture of Step Distributions]
\begin{align*}
    \rstep_v(x_1,\dots,x_n) = \Ex_{v_1,\dots,v_n \sim \unif(0,\eps/2)}[\Pi_{i=1}^n[\step_v(x_i)]]
\end{align*}
\end{definition}

With these definitions, we can concretely outline our agenda. Let $\Delta_{\textrm{step}}$ be the largest two-point testing lower bound for any valid step distribution, and $\Delta_{\textrm{tri}}$ be a value such that $\dtv(\tri(x)^{\otimes n},\tri(x-\Delta_{\textrm{tri}})^{\otimes n})$ is small. We will observe that $\Delta_{\textrm{step}} \ll \Delta_{\textrm{tri}}$. Then, if we show $\dtv(\rstep(x_1,\dots,x_n),\tri(x)^{\otimes n})$ is small, this would imply that the quantity $\dtv(\rstep(x_1,\dots,x_n),\rstep(x_1-\Delta_{\textrm{tri}},\dots,x_n-\Delta_{\textrm{tri}}))$ is small, and thus that for any algorithm there is at least one step distribution where it incurs error $\Delta_{\textrm{tri}}$ with at least constant probability. However, since $\Delta_{\textrm{step}} \ll \Delta_{\textrm{tri}}$, this implies that we cannot attain the two-point testing bound for the step distribution.

The bulk of our effort will be in proving that $\dtv(\rstep(x_1,\dots,x_n),\tri(x)^{\otimes n})$ is small. To do so, we will compute an upper bound on their $\chi^2$ divergence. In an effort to simplify calculations, we will now introduce two modified distributions $\mtri,\rmstep$, that we design to have smaller distance than $\tri,\rstep$ by a data-processing inequality argument: as we show there is a deterministic function $h$ where $h(\mtri)=\tri$ and each $h(\mstep_v)=\step_v$, so 
\begin{align*}
    & \dtv(\tri^{\otimes n},\rstep^{\otimes n}) \\
    & = \dtv(h(\mtri)^{\otimes n},h(\rmstep)^{\otimes n}) \\
    & \le \dtv(\mtri^{\otimes n},\rmstep^{\otimes n}).
\end{align*}
Moreover, we design $\mstep$ so that it is easier to work with because each step interval $[i\eps,(i+1)\eps)$ will be identical (as opposed to the original distributions that have different heights). We design

\begin{definition}[Modified Triangle Distribution]

       \begin{align*}
     \mtri(x) = \begin{cases} 
          \frac{1}{2} + (i+1)\eps - |x| & |x| \in [i\eps,(i+1)\eps) \textrm{ for } i \in \{0,\dots,\frac{1}{2\eps}-1\}\\
          1-|x| & \frac{1}{2} \le |x| < 1 \\
          \frac{1}{2}-(i+1)\eps & |x| \in [1+i\eps,1+(i+1)\eps) \textrm{ for } i \in \{0,\dots,\frac{1}{2\eps}-1\}\\
          0 & |x|>\frac{3}{2}
       \end{cases}
    \end{align*}
\end{definition}

\begin{definition}[Modified Step distribution] Let $v$ be a vector of length $\frac{1}{2\eps}$, where each $v_i \in [0,\eps/2]$. The parameter $v_i$ informs the length of the $i$-th step:

    \begin{align*}
     \mstep_v(x) = \begin{cases} 
          \frac{1}{2} + s_{v_i}((i+1)\eps - |x|) & |x| \in [i\eps,(i+1)\eps) \textrm{ for } i \in \{0,\dots,\frac{1}{2\eps}-1\}\\
          1-|x| & \frac{1}{2} \le |x| < 1 \\
          \frac{1}{2} - (i+1)\eps & |x| \in [1+i\eps,1+(i+1)\eps) \textrm{ for } i \in \{0,\dots,\frac{1}{2\eps}-1\}\\
          0 & |x|>\frac{3}{2}
       \end{cases}
    \end{align*}
\end{definition}

We now give our function $h$:

\begin{definition}[Deterministic Mapping $h$]
    \begin{align*}
     h(x) = \begin{cases} 
          x & |x|<1 \textrm{ or } |x| \ge \frac{3}{2}\\
          x-1 & 1 \le x < \frac{3}{2}\\
          x + 1 & -\frac{3}{2}< x \le -1
       \end{cases}
    \end{align*}
\end{definition}

\begin{claim}\label{claim:use-dpi}
    $\dtv(\tri^{\otimes n},\rstep^{\otimes n}) \le \dtv(\mtri^{\otimes n},\rmstep^{\otimes n})$
\end{claim}
\begin{proof}
    Note how $\tri=h(\mtri)$ and $\step_v = h(\mstep_v)$. Thus, 
    $\dtv(\tri^{\otimes n},\rstep^{\otimes n}) = \dtv(h(\mtri)^{\otimes n},h(\rmstep)^{\otimes n}) \le \dtv(\mtri^{\otimes n},\rmstep^{\otimes n})$ by data-processing inequality.
\end{proof}

Now, we bound $\dtv(\mtri^{\otimes n},\rmstep^{\otimes n})$ by analyzing $\dchi(\rmstep^{\otimes n}\| \mtri^{\otimes n})$, via a mostly routine calculation.

\begin{lemma}\label{lemma:hard-lb}
    There exists a universal constant $C>0$ such that, for any $\eps \le \frac{1}{2}$, if $n \le \frac{C}{\eps^{2.5}}$ then $\dtv(\rmstep^{\otimes n},  \mtri^{\otimes n}) \le \frac{1}{10}$.
\end{lemma}
\begin{proof}
Let us define $p_0(x) \triangleq \mtri(x)$, let $p_v(x) \triangleq \mstep_v(x)$, and let $k \triangleq \frac{1}{2\eps}$. Then:
\begin{align*}
    & \dchi(\rmstep^{\otimes n}\| \mtri^{\otimes n})\\
    & = \Ex_{v,w \sim \unif(0,\eps/2)^{k}}\left[\int_{x_1,\dots,x_n} \Pi_{i=1}^n \frac{p_v(x_i)p_w(x_i)}{p_0(x_i)}\right] -1 \\
    & = \Ex_{v,w \sim \unif(0,\eps/2)^{k}}\left[\left(\int_{-\infty}^\infty \frac{p_v(x)p_w(x)}{p_0(x)} dx\right)^n\right] -1 \\
    \intertext{For ease of notation, let us denote $f(v,w) \triangleq \int_{-\infty}^\infty \frac{p_v(x) p_w(x)}{p_0(x)} dx$}
    & = \Ex_{v,w \sim \unif(0,\eps/2)^{k}}\left[f(v,w)^n\right] -1 \numberthis \label{step:chi2-f}
\end{align*}
We will hence aim to bound $f(v,w)$:
\begin{align*}
    & f(v,w) \triangleq \int_{-\infty}^\infty \frac{p_v(x)p_w(x)}{p_0(x)} dx \\
    \intertext{Now, we use the actual values of $p_0$ and $p_v$ to start calculating the integral. Note how $p_0$ and $p_v$ are symmetric around $0$ for all $v$, all $v$ satisfy $p_0(x)=p_v(x)$ for $|x|>\frac{1}{2}$, and $p_0(x)=0$ for $|x|>\frac{3}{2}$.}
    & = 2 \int_{0}^{1/2} \frac{p_v(x) p_w(x)}{p_0(x)} dx + 2 \int_{1/2}^1 p_0(x) dx + 2 \int_1^{3/2} p_0(x) dx\\
    & = 2 \int_{0}^{1/2} \frac{p_v(x) p_w(x)}{p_0(x)} dx + 2 \int_{0}^{1/2} x \, dx + 2 \sum_{i=0}^{k-1} i \eps^2 \\
    & = 2 \int_{0}^{1/2} \frac{p_v(x) p_w(x)}{p_0(x)} dx + \frac{1}{4} +\left(\frac{1}{4}-\eps/2\right) \\
    & = \frac{1}{2} - \eps/2 +  2 \sum_{i=0}^{k-1} \int_0^\eps \frac{(\frac{1}{2}+s_{v_i}(x))(\frac{1}{2}+s_{w_i}(x))}{\frac{1}{2}+x} dx \\
    \intertext{To evaluate this integral, we will separate into the five intervals where $s_{v_i}(x)$ and $s_{w_i}(x)$ are constant. For ease of notation, let $a_i \triangleq \min(v_i,w_i)$ and $b_i \triangleq \max(v_i,w_i)$.}
    & = \frac{1}{2}-\frac{\eps}{2} + 2 \sum_{i=0}^{k-1} \int_0^{\eps/2-b_i} \frac{1/4}{1/2 + x} dx + \int_{\eps/2-b_i}^{\eps/2-a_i} \frac{1/2 (1/2 + \eps/2)}{1/2 + x} dx + \int_{\eps/2-a_i}^{\eps/2+a_i} \frac{(1/2 + \eps/2)^2}{1/2 + x} dx \\
    &+\int_{\eps/2+a_i}^{\eps/2+b_i} \frac{(1/2+\eps/2)(1/2+\eps)}{1/2+x} dx + \int_{\eps/2+b_i}^{\eps} \frac{(1/2+\eps)^2}{1/2+x} dx\\
    & = \frac{1}{2} - \frac{\eps}{2} + 2 \sum_{i=0}^{k-1} \frac{1}{4} \ln\left(\frac{1/2 + \eps/2 - b_i}{1/2}\right) + \left(\frac{1}{4}+\eps/4\right) \ln\left(\frac{1/2 + \eps/2-a_i}{1/2 + \eps/2-b_i}\right) + \left(\frac{1}{2} + \frac{\eps}{2}\right)^2 \ln\left(\frac{1/2 + \eps/2 + a_i}{1/2 + \eps/2 - a_i}\right) \\
    & + \left(\frac{1}{2} + \frac{\eps}{2}\right)\left(\frac{1}{2}+\eps\right)\ln\left(\frac{1/2 + \eps/2 + b_i}{1/2 + \eps/2 + a_i}\right) + \left(\frac{1}{2}+\eps\right)^2 \ln\left(\frac{1/2 + \eps}{1/2 + \eps/2 + b_i}\right)\\
    & = \frac{1}{2} - \frac{\eps}{2} + 2 \sum_{i=0}^{k-1} \frac{1}{4} \ln(2) + \frac{\eps}{4} \ln\left(\frac{1}{1/2 + \eps/2 - b_i}\right) + \left(\frac{\eps}{4}+\frac{\eps^2}{4}\right)\ln\left(\frac{1}{1/2 + \eps/2 - a_i}\right) \\
    & + \left(\frac{\eps}{4} + \frac{\eps^2}{4}\right)\ln\left(\frac{1}{1/2 + \eps/2 + a_i}\right) + \left(\frac{\eps}{4}+\frac{\eps^2}{2}\right)\ln\left(\frac{1}{1/2 + \eps/2 + b_i}\right) + \left(\frac{1}{2}+\eps\right)^2 \ln\left(\frac{1}{2}+\eps\right)\\
    \intertext{Now, we modify to make a later Taylor expansion cleaner (roughly, changing arguments from $\ln(\frac{1}{2}+x)$ to $\ln(1+2x)-\ln(2)$):}
    & = \frac{1}{2} - \frac{\eps}{2} + 2 \sum_{i=0}^{k-1} \frac{1}{4} \ln(2) + \frac{\eps}{4} \left(\ln\left(\frac{1}{1 + \eps - 2b_i}\right)+\ln(2)\right) + \left(\frac{\eps}{4} + \frac{\eps^2}{4}\right) \left(\ln\left(\frac{1}{1+\eps-2a_i}\right) + \ln(2)\right) \\
    &+ \left(\frac{\eps}{4}+\frac{\eps^2}{4}\right)\left(\ln\left(\frac{1}{1+\eps+2a_i}\right)+\ln(2)\right) + \left(\frac{\eps}{4}+\frac{\eps^2}{2}\right)\left(\ln\left(\frac{1}{1 + \eps + 2b_i}\right)+\ln(2)\right) \\
    & + \left(\frac{1}{2}+\eps\right)^2 \left(\ln(1+2\eps)-\ln(2)\right)\\
    & = \frac{1}{2} - \frac{\eps}{2} + 2 \sum_{i=0}^{k-1} \frac{\eps}{4} \ln\left(1-\frac{\eps-2b_i}{1+\eps-2b_i}\right) + \left(\frac{\eps}{4}+\frac{\eps^2}{4}\right)\ln\left(1-\frac{\eps-2a_i}{1+\eps-2a_i}\right) + \left(\frac{\eps}{4}+\frac{\eps^2}{4}\right)\ln\left(1-\frac{\eps+2a_i}{1+\eps+2a_i}\right) \\
    & + \left(\frac{\eps}{4}+\frac{\eps^2}{2}\right) \ln\left(1 - \frac{\eps+2b_i}{1+\eps+2b_i}\right) + \left(\frac{1}{2}+\eps\right)^2 \ln(1+2 \eps)\\
    & = \sum_{i=0}^{k-1} \eps - \eps^2 +\frac{\eps}{2} \ln\left(1-\frac{\eps-2b_i}{1+\eps-2b_i}\right) + \left(\frac{\eps}{2}+\frac{\eps^2}{2}\right)\ln\left(1-\frac{\eps-2a_i}{1+\eps-2a_i}\right) + \left(\frac{\eps}{2}+\frac{\eps^2}{2}\right)\ln\left(1-\frac{\eps+2a_i}{1+\eps+2a_i}\right) \\
    & + \left(\frac{\eps}{2}+\eps^2\right) \ln\left(1 - \frac{\eps+2b_i}{1+\eps+2b_i}\right) + 2\left(\frac{1}{2}+\eps\right)^2 \ln(1+2 \eps) 
\end{align*}

As will soon be more clear, for all values of $v,w$ it will be that case that $f(v,w)\approx 1$. Accordingly, to study $f(v,w)^n-1$, it may be more insightful to analyze $(1+g(v,w))^n$, where $g(v,w)\triangleq f(v,w)-1$. We define the following $g_i(\cdot)$ function so that $\sum_{i=1}^k g_i(v_i,w_i) = g(v,w) = f(v,w)-1$:
\begin{align*}
    g_i(v_i,w_i) & \triangleq -\eps  - \eps^2 +\frac{\eps}{2} \ln\left(1-\frac{\eps-2b_i}{1+\eps-2b_i}\right) + \left(\frac{\eps}{2}+\frac{\eps^2}{2}\right)\ln\left(1-\frac{\eps-2a_i}{1+\eps-2a_i}\right) + \\
    & \left(\frac{\eps}{2} +\frac{\eps^2}{2}\right)\ln\left(1-\frac{\eps+2a_i}{1+\eps+2a_i}\right) + \left(\frac{\eps}{2}+\eps^2\right) \ln\left(1 - \frac{\eps+2b_i}{1+\eps+2b_i}\right) + 2\left(\frac{1}{2}+\eps\right)^2 \ln\left(1+2 \eps\right) \\
    & \intertext{We will now bound $g_i(v_i,w_i)$. Starting with a Taylor expansion that uses $\ln(1+x) \le x - \frac{x^2}{2} + \frac{x^3}{3}$ and $\ln(1-x) \le -x$, then this is valid for $0 < \eps \le \frac{1}{2}$:}
    & \le -\eps - \eps^2 + \frac{\eps}{2} \cdot \frac{2b_i - \eps}{1 + \eps - 2b_i} + \left(\frac{\eps}{2}+\frac{\eps^2}{2}\right) \cdot \frac{2a_i - \eps}{1 + \eps - 2 a_i} + \left(\frac{\eps}{2}+\frac{\eps^2}{2}\right) \cdot \frac{-2a_i - \eps}{1 + \eps + 2a_i} \\
    &+ \left(\frac{\eps}{2}+\eps^2\right)\cdot \frac{- \eps - 2b_i}{1 + \eps + 2b_i} + 2\left(\frac{1}{2}+\eps\right)^2 \cdot \left(2\eps - 2\eps^2 + \frac{8 \eps^3}{3}\right)\\
    \intertext{Note that all terms other than the last are non-positive, as $0 < a_i,b_i \le \frac{\eps}{2}$. Now, we use $\frac{1}{1+z} = 1 - \frac{z}{1+z} \ge (1-z)$ for $z\ge 0$. }
    & \le -\eps - \eps^2 + \frac{\eps}{2} \cdot (2b_i - \eps)(1-\eps) + \left(\frac{\eps}{2}+\frac{\eps^2}{2}\right) \cdot (2a_i - \eps) \cdot (1-\eps) + \left(\frac{\eps}{2}+\frac{\eps^2}{2}\right) \cdot (-2a_i - \eps) \cdot (1-2\eps) \\
    &+ \left(\frac{\eps}{2}+\eps^2\right)\cdot (- \eps - 2b_i)(1-2\eps) + 2\left(\frac{1}{2}+\eps\right)^2 \cdot \left(2\eps - 2\eps^2 + \frac{8 \eps^3}{3}\right)\\
    & = a_i \eps^2 - b_i \eps^2 + \frac{7 \eps^3}{3} + a_i \eps^3 + 4 b_i \eps^3 + \frac{29 \eps^4}{6} + \frac{16 \eps^5}{3} \\
    & \le O(1) \cdot \eps^3 \numberthis \label{step:gi-bound}
\end{align*}

Finally, we show how this enables us to directly bound $\Ex[f(v,w)^n]-1$, picking up from \cref{step:chi2-f}:

\begin{align*}
    & \dchi(\rmstep^n\| \mtri^n) \\
    & = \Ex_{v,w \sim \unif(0,\eps/2)^k}\left[ f(v,w)^n \right] - 1 \\
    & = \Ex_{v,w \sim \unif(0,\eps/2)^k}\left[ (1+g(v,w))^n \right] - 1 \\
    & = \Ex_{v,w \sim \unif(0,\eps/2)^k}\left[ \sum_{j=0}^n \binom{n}{j} g(v,w)^j \right] -1 \\
    & = \Ex_{v,w \sim \unif(0,\eps/2)^k}\left[ \sum_{j=1}^n \binom{n}{j} g(v,w)^j \right] \\
    & = \Ex_{v,w \sim \unif(0,\eps/2)^k}\left[ \sum_{j=2}^n \binom{n}{j} g(v,w)^j \right] + n \Ex_{v,w \sim \unif(0,\eps/2)^k}\left[ g(v,w) \right]\\
    \intertext{Note how $\Ex_{v,w}\left[ g(v,w) \right]=\Ex_{v,w}\left[ f(v,w) \right]-1$ and that we designed our distributions so that $\Ex_v[p_v(x)]=p_0(x)$ for all $x$, and thus $\Ex_{v,w}\left[ g(v,w) \right]=\Ex_{v,w}\left[ f(v,w) \right]-1 = 1-1 = 0$:}
    & = \Ex_{v,w \sim \unif(0,\eps/2)^k}\left[ \sum_{j=2}^n \binom{n}{j} g(v,w)^j \right]\\
    & = \Ex_{v,w \sim \unif(0,\eps/2)^k}\left[ \sum_{j=2}^n \binom{n}{j} \cdot \left( \sum_{i=0}^{k-1} g_i(v_i,w_i) \right)^j \right]\\
    \intertext{Recall how we just used $\Ex_{v,w}[g(v,w)]=0$. As each $g_i(v_i,w_i)$ is i.i.d., then we also know $\Ex_{v,w}[g_i(v_i,w_i)]=0$ for all $i$, and when we expand the previous step, any term with an $i$ appearing exactly once will evaluate to $0$. Let us use that the number of ordered sequences of length $j$ from $k$ elements, that have no element occuring exactly once, is at most $k^{\lfloor j/2 \rfloor} \lfloor j/2 \rfloor^j \le \frac{k^{j/2} j^j}{2^j}$:}
    & \le \Ex_{v,w \sim \unif(0,\eps/2)^k}\left[ \sum_{j=2}^n \binom{n}{j} \cdot \frac{k^{j/2} j^j}{2^j} \cdot \left(\max_{i} |g_i(v_i,w_i)| \right)^j \right]\\
    & \le \Ex_{v,w \sim \unif(0,\eps/2)^k}\left[ \sum_{j=2}^n \left( \frac{en}{j} \right)^j \cdot \frac{k^{j/2} j^j}{2^j} \cdot \left(\max_{i} |g_i(v_i,w_i)| \right)^j \right]\\
    \intertext{Recall our upper bound on $g_i(v_i,w_i)$ from \cref{step:gi-bound}. Also observe that $g_i(v_i,w_i) \ge 0$, as otherwise the distribution corresponding to the $v,w$ that have each entry identical to $v_i,w_i$ would have negative $\chi^2$ divergence with $p_0$, which is impossible. Thus, our upper bound on $g_i(v_i,w_i)$ is also an upper bound on $|g_i(v_i,w_i)|$:}
    & \le \Ex_{v,w \sim \unif(0,\eps/2)^k}\left[ \sum_{j=2}^n \left( \frac{en}{j} \right)^j \cdot \frac{k^{j/2} j^j}{2^j} \cdot \left(O(1) \cdot \eps^3 \right)^j \right]\\
    & = \Ex_{v,w \sim \unif(0,\eps/2)^k}\left[ \sum_{j=2}^n \frac{e^j n^j k^{j/2} O(1)^j \eps^{3j}}{2^j}  \right]\\
    \\
    & = \Ex_{v,w \sim \unif(0,\eps/2)^k}\left[ \sum_{j=2}^n \frac{e^j n^j O(1)^j \eps^{2.5j}}{2^j}  \right] \rtag{recall $k = \frac{1}{2\eps}$}\\
    \intertext{This sum will be upper bounded by at most a constant factor more than its first term, as long as the ratio of consecutive terms is bounded above by, say, $\frac{1}{2}$. The ratio is at most $O(1) n \eps^{2.5}$, so there exists a universal constant $0<C_0<1$ such that if $n \le \frac{C_0}{\eps^{2.5}}$, then this expectation is bounded by:}
    & \le O(1) \cdot n^2 \eps^5
\end{align*}

Thus, we may conclude there is a constant $0<C<1$ such that for any $0< \eps \le \frac{1}{2}$, if $n \le \frac{C}{\eps^{2.5}}$, then $\dchi(\rmstep^{\otimes n},\mtri^{\otimes n}) \le \frac{1}{50}$. Using $\dtv(P,Q) \le \sqrt{\frac{1}{2}\cdot \dchi(P,Q)}$ (e.g. see Section 13.2.1 of \cite{duchinotes}, which also outlines the general technique of this point-mixture lower bound style used in this proof), then:

\begin{equation*}
    \dtv(\rmstep^{\otimes n},\mtri^{\otimes n}) \le \sqrt{\frac{1}{2} \cdot \dchi(\rmstep^{\otimes n},\mtri^{\otimes n})} \le \frac{1}{10}
\end{equation*}
\end{proof}

We now show that it is hard to distinguish the triangle distribution from a translated version with an appropriately chosen translation:

\begin{claim}\label{claim:tri-hard}
    There exists a constant $0<C < 1$ where, if $n \ge 2$ and we let $\Delta_{\textrm{tri}} \triangleq \frac{C}{\sqrt{\log(n) \cdot n}}$, then:
    \begin{equation*}
        \dtv(\tri(x)^{\otimes n}, \tri(x-\Delta_{\textrm{tri}})^{\otimes n}) \le \frac{1}{10}
    \end{equation*}
\end{claim}
\begin{proof}
    \begin{align*}
        & \dtv(\tri(x)^{\otimes n}, \tri(x-\Delta_{\textrm{tri}})^{\otimes n})\\
        & \le \sqrt{2} \cdot \sqrt{\dhsq(\tri(x)^{\otimes n}, \tri(x-\Delta_{\textrm{tri}})^{\otimes n})} \\
        & =\sqrt{2} \cdot \sqrt{1 - (1 - \dhsq(\tri(x), \tri(x-\Delta_{\textrm{tri}})))^n} \intertext{Observe that at least a quarter of the Hellinger distance comes from the domain $[-1,0]$:}
        & \le \sqrt{2} \cdot \sqrt{1 - \left(1 - 2 \cdot \left( \int_{0}^{1-\Delta_{\textrm{tri}}} (\sqrt{\tri(x)}-\sqrt{\tri(x+\Delta_{\textrm{tri}})})^2 \diff x + \int_{1-\Delta_{\textrm{tri}}}^1 \tri(x) \diff x \right)\right)^n}\\
        & \le \sqrt{2} \cdot \sqrt{1 - \left(1 - 2 \cdot \left( \int_{2 \cdot \Delta_{\textrm{tri}}}^{1} (\sqrt{x}-\sqrt{x-\Delta_{\textrm{tri}}})^2 \diff x + \int_0^{2\Delta_{\textrm{tri}}} x \diff x \right)\right)^n}\\
        & \le \sqrt{2} \cdot \sqrt{1 - \left(1 - 2 \cdot \left( \int_{2 \cdot \Delta_{\textrm{tri}}}^{1} (\Delta_{\textrm{tri}}/\sqrt{x-\Delta_{\textrm{tri}}})^2 \diff x + \int_0^{2\Delta_{\textrm{tri}}} x \diff x \right)\right)^n}\\
        & \le \sqrt{2} \cdot \sqrt{1 - \left(1 - 2 \cdot \left( \Delta_{\textrm{tri}}^2 \cdot \ln(1/\Delta_{\textrm{tri}}) + 2 \cdot \Delta_{\textrm{tri}}^2 \right)\right)^n} \intertext{We will choose a sufficiently small $C$ where $\ln(1/\Delta_{\textrm{tri}}) \ge 1$, as it enforced by $n \ge 2$ and $C \le \frac{\sqrt{2}}{e}$:}
        & \le \sqrt{2} \cdot \sqrt{1 - \left(1 - 6 \cdot \frac{C^2}{\log(n) \cdot n} \cdot \ln(\sqrt{\log(n) \cdot n}/C)\right)^n} \\
        & \le \sqrt{2} \cdot \sqrt{1 - \left(1 - 6 \cdot \frac{\ln(n/C) C^2}{\log(n) \cdot n}\right)^n} \rtag{using $\log(n) \le n$}\\
        & \le \sqrt{2} \cdot \sqrt{1 - \left(1 - 6 \cdot \frac{C^2 \cdot (1 + \ln(1/C))}{n}\right)^n}\\
        & = \sqrt{2} \cdot \sqrt{1 - \left(1 - 6 \cdot \frac{C^2 \cdot (1 + \ln(1/C))}{n}\right)^{\left(\frac{n}{6 \cdot C^2 \cdot (1 + \ln(1/C))}\right) \cdot \left({6 \cdot C^2 \cdot (1 + \ln(1/C))}\right)}}\\
        & \le \sqrt{2} \cdot \sqrt{1 - 0.3^{{6 \cdot C^2 \cdot (1 + \ln(1/C))}}} \le \frac{1}{10} \rtag{for $C$ where $\frac{6 \cdot C^2 \cdot (1 + \ln(1/C))}{n} \le \frac{1}{4}$}
    \end{align*}
    for sufficiently small $C$.
\end{proof}

Together, \cref{claim:use-dpi,lemma:hard-lb,claim:tri-hard} enable us to show a lower bound for the performance of adaptive mean estimation (which we will not yet relate to the two-point testing rate):
\begin{corollary}\label{cor:est-bad}
    There exists some constant $C>0$ such that if $n \le \frac{C}{\eps^{2.5}}$, then any estimator $\est$ must likely incur $\frac{C}{\sqrt{n \log(n)}}$ error for some translation of some step distribution. More formally:
    \begin{equation*}
        \min_{\est} \max_{v \in [0,\eps/2)^{\frac{1}{2\eps}}, \mu \in \R} \Pr_{X \sim \step_v(x-\mu)^{\otimes n},\est}\left[|\est(X)-\mu| \ge \frac{C}{\sqrt{n \log(n)}}\right] \ge \frac{7}{20}
    \end{equation*}
\end{corollary}
\begin{proof}
    Let $C'$ be the constant in \cref{claim:tri-hard}, and consider a testing problem between $\rstep(x)^{\otimes n}$ and $\rstep(x - \Delta_{\textrm{tri}})^{\otimes n}$ where $\Delta_{\textrm{tri}} \triangleq \frac{C'}{n \log(n)}$. Then, if $C < C'/2$, we remark that an estimator which has error at most $\frac{C}{n \log(n)}$ is able to distinguish the testing problem. Hence:
    \begin{align*}
        & \min_{\est} \max_{v \in [0,\eps/2)^{\frac{1}{2\eps}}, \mu \in \R} \Pr_{X \sim \step_v(x-\mu)^{\otimes n},\est}\left[|\est(X)-\mu| \ge \frac{C}{\sqrt{n \log(n)}}\right]\\
        & \ge \frac{1-\dtv(\rstep(x)^{\otimes n},\rstep(x - \Delta_{\textrm{tri}})^{\otimes n})}{2}\\
        & \ge \frac{1}{2} - \frac{1}{2} \cdot \dtv(\rstep(x)^{\otimes n},\tri(x)^{\otimes n}) - \frac{1}{2} \cdot \dtv(\tri(x)^{\otimes n},\tri(x-\Delta_{\textrm{tri}})^{\otimes n}) \\
        & - \frac{1}{2} \cdot \dtv(\rstep(x- \Delta_{\textrm{tri}})^{\otimes n},\tri(x - \Delta_{\textrm{tri}})^{\otimes n}) \intertext{Using \cref{claim:tri-hard}:}
        & \ge \frac{1}{2} - \frac{1}{20} - \frac{1}{2} \cdot \dtv(\rstep(x)^{\otimes n},\tri(x)^{\otimes n}) - \frac{1}{2} \cdot \dtv(\rstep(x- \Delta_{\textrm{tri}})^{\otimes n},\tri(x - \Delta_{\textrm{tri}})^{\otimes n}) 
        & \ge \frac{1}{2} - \frac{1}{20} -  \dtv(\mtri^{\otimes n},\rmstep^{\otimes n}) \rtag{using \cref{claim:use-dpi}} \\
        & \ge \frac{1}{2} - \frac{1}{20} -  \frac{1}{10} = \frac{7}{20}. \rtag{using \cref{lemma:hard-lb}}\quad\qedhere
    \end{align*}
\end{proof}
All that remains is to analyze the two-point testing rate for step distributions and determine for which $n_0$ is the two-point testing rate for $n_0$ samples still unattainable from $n \gg n_0$ samples given our lower bound from \cref{cor:est-bad}.

\begin{claim}
    For any $\Delta \le \eps/2$, it holds that for all $v \in [0,\eps/2)^{\frac{1}{2\eps}}$:
    \begin{equation*}
        \dhsq(\step_v(x),\step_v(x+\Delta)) \ge \frac{\eps\Delta}{16}
    \end{equation*}
\end{claim}
\begin{proof}
    \begin{align*}
        & \dhsq(\step_v(x),\step_v(x+\Delta)) \\
        & \ge \int_0^{\frac{1}{2}} \left(\sqrt{\step_v(x)} - \sqrt{\step_v(x+\Delta)}  \right)^2 \diff x \\
        & = \sum_{i=0}^{\frac{1}{2\eps}-1} \int_{i\eps}^{(i+1)\eps} \left(\sqrt{\step_v(x)} - \sqrt{\step_v(x+\Delta)}  \right)^2 \diff x \intertext{Using the structure of step functions and that $\Delta \le \eps/2$:}
        & \ge \sum_{i=0}^{\frac{1}{2\eps}-1} \int_{i\eps + \eps/2 + v_i - \Delta}^{i\eps + \eps/2 + v_i} \left(\sqrt{\step_v(x)} - \sqrt{\step_v(x+\Delta)}  \right)^2 \diff x \\
        & \ge \sum_{i=0}^{\frac{1}{2\eps}-1} \int_{i\eps + \eps/2 + v_i - \Delta}^{i\eps + \eps/2 + v_i} \left(\sqrt{\step_v(x)} - \sqrt{\step_v(x) - \eps/2}  \right)^2 \diff x\\
        & \ge \sum_{i=0}^{\frac{1}{2\eps}-1} \int_{i\eps + \eps/2 + v_i - \Delta}^{i\eps + \eps/2 + v_i} \left(\frac{\eps/4}{\sqrt{\step_v(x)}} \right)^2 \diff x \\
        & \ge \sum_{i=0}^{\frac{1}{2\eps}-1} \int_{i\eps + \eps/2 + v_i - \Delta}^{i\eps + \eps/2 + v_i} \left(\frac{\eps/4}{\sqrt{1/2}} \right)^2 \diff x \\
        & = \frac{\eps \Delta}{16}
    \end{align*}
\end{proof}
We remark that the same proof immediately implies the guarantee in terms of $\min(\Delta,\eps/2)$ with no required upper bound on $\Delta$. This enables a lower bound of the Hellinger distance for all translations:
\begin{corollary}\label{cor:modulus-bound}
    For all $v \in [0,\eps/2)^{\frac{1}{2\eps}}$:
    \begin{equation*}
    \dhsq(\step_v(x),\step_v(x+\Delta)) \ge \frac{\eps \cdot \min(\Delta,\eps/2)}{16}
\end{equation*}
This immediately implies that if $\frac{\eps^2}{32} \ge \frac{1}{n_0}$ then:
\begin{equation*}
    \omega_{\step_v}\left(\frac{1}{n_0}\right) \le \frac{16}{\eps \cdot n_0}
\end{equation*}
\end{corollary}

We are finally ready to conclude for which value of $n_0$ must any estimator incur error at least $\nu \cdot \omega_{\step_v}\left(\frac{1}{n_0}\right)$ with constant probability:

\begin{lemma}\label{lemma:related-errs}
    There exists a universal constant $C>0$ such that for any sufficiently large $n$, any value $\nu \ge 1$, and any  estimator $\est$, then there exists a setting of $\eps$ such that $\est$ must incur large error with constant probability for some translation of a step distribution:
    \begin{equation*}
        \min_{\est} \max_{v \in [0,\eps/2)^{\frac{1}{2\eps}}, \mu \in \R} \Pr_{X \sim \step_v^{\otimes n},\est}\left[|\est(X)-\mu| \ge \nu \cdot \omega_{\step_v}\left( \frac{C^{7/5}}{128\nu n^{9/10} \sqrt{\log(n)}} \right)\right] \ge \frac{7}{20}
    \end{equation*}
\end{lemma}
\begin{proof}
    First, we will set $\eps$. It is our intention to use \cref{cor:est-bad}, so we must satisfy $n \le \frac{C}{\eps^{2.5}} \Leftrightarrow \frac{1}{\eps} \ge \left(\frac{n}{C}\right)^{2/5}$. Additionally, we have the constraint that $\frac{1}{2\eps}$ is an integer. For sufficiently large $n$, there will be a satisfying value of $\eps$ where $\frac{1}{\eps} \in [\left(\frac{n}{C}\right)^{2/5}, 2 \cdot \left(\frac{n}{C}\right)^{2/5}]$. 
    
    Given \cref{cor:est-bad}, then it is sufficient to show:
    \begin{align*}
        & \frac{C}{\sqrt{n \log (n)}} \ge \nu \cdot \omega_{\step_v}\left( \frac{1}{n_0} \right) \intertext{It is our goal to see how large $\frac{1}{n_0}$ can be while satisfying this inequality. If we later set parameters such that $\frac{1}{n_0} \le \frac{\eps^2}{32}$, then we may invoke \cref{cor:modulus-bound}. By our choice of $\eps$, this is satisfied as long as $\frac{1}{n_0} \le \frac{1}{128 \cdot \left(\frac{n}{C}\right)^{4/5}}=\frac{C^{4/5}}{128 \cdot n^{4/5}}$:}
        & \impliedby \frac{C}{\sqrt{n \log (n)}} \ge  \frac{16 \nu }{\eps \cdot n_0}\\
        & \Leftrightarrow \frac{C \cdot \eps}{16 \nu \cdot \sqrt{n \log (n)}} \ge  \frac{1}{n_0} \\
        & \impliedby \frac{C \cdot \frac{C^{2/5}}{2 \cdot n^{2/5}}}{16 \nu \cdot \sqrt{n \log (n)}} \ge  \frac{1}{n_0}\\
        & \Leftrightarrow  \frac{C^{7/5}}{32 \nu \cdot n^{9/10} \sqrt{\log (n)}} \ge  \frac{1}{n_0}
    \end{align*}
    Hence, the lemma holds if:
    \begin{equation*}
        \frac{1}{n_0} \le \min\left(\frac{C^{7/5}}{32 \nu \cdot n^{9/10} \sqrt{\log (n)}}, \frac{C^{4/5}}{128 \cdot n^{4/5}} \right) \impliedby \frac{1}{n_0} \le \frac{C^{7/5}}{128 \nu \cdot n^{9/10} \sqrt{\log(n)}}
    \end{equation*}
\end{proof}

The statement of our theorem follows from \cref{lemma:related-errs}, except for the condition $\omega_D \left( \frac{C_1}{\nu \cdot n^{9/10} \sqrt{\log(n)}} \right) > 0$. This is implied by:

\begin{claim}
    For any $\Delta \ge 0$, it holds that for all $v \in [0,\eps/2)^{\frac{1}{2\eps}}$:
    \begin{equation*}
        \dhsq(\step_v(x),\step_v(x+\Delta)) \le 2 \Delta
    \end{equation*}
\end{claim}
\begin{proof}
We start by observing how since $\step_v$ is symmetric and unimodal, the Hellinger distance contribution from $[0,\infty]$ will be at least the value of each of $[-\infty,-\Delta]$, $[-\Delta,-\Delta/2]$, and $[-\Delta/2,0]$:
    \begin{align*}
        & \dhsq(\step_v(x),\step_v(x+\Delta)) \\
        & \le 2\int_0^{\infty} \left(\sqrt{\step_v(x)} - \sqrt{\step_v(x+\Delta)}  \right)^2 \diff x \\
        & \le 2\int_0^{\infty} \left| \step_v(x) - \step_v(x+\Delta)  \right| \diff x \\
        & = 2\int_0^{\Delta} \step_v(x) \diff x \\
        & \le 2 \Delta \quad\qedhere
    \end{align*}
\end{proof}

Hence, this concludes the proof of our theorem.
\end{proof}

\section{Location Estimation for Unimodal Distributions}
We now study location estimation, where the distribution is known up to translation. We will discuss an approach that nearly attains the two-point testing rate for location estimation of unimodal distributions. Suppose the density $p$ is known up to translation ($p(0)$ is the mode of our known density before translation) and let $P_\theta$ denote the distribution with density $p(x-\theta)$. Given that the density is known up to translation, a natural approach would be to compute the MLE among all translations. Indeed, the work of \cite{gupta2024minimax} shows that a variant of the MLE attains a form of minimax optimality for this task. However, it is still not obvious how to directly analyze whether the MLE attains the two-point testing rate for this task. 

Instead, we will analyze a modified version of the MLE. As a warmup, consider the easier task of estimating the mean from a list $L$ candidate means $\theta_1,\dots,\theta_{|L|}$, \textit{where it is promised the true mean $\mu \in L$}. Now, consider a procedure where for each pair $i \ne j$ we compute whether the empirical likelihood of $n$ samples is larger for $\theta_i$ or $\theta_j$. Using the Neyman-Pearson lemma and classical properties of Hellinger distance, we could conclude that with probability $\ge 1 - \delta$, the true mean will only lose in comparisons against $\theta_i$ where $\dhsq(P_\mu, P_{\theta_i}) = O(\frac{\log(|L|/\delta)}{n})$. This is sufficient to find an estimate of the mean within $\omega_P(O(\frac{\log(|L|/\delta)}{n}))$. We simply choose the $\theta_i$ that is undefeated (if one exists), or otherwise we choose the $\theta_i$ whose farthest loss is closest to $\theta_i$. This works because if the chosen $\theta_i$ were a poor enough estimate such that $|\mu - \theta_i| \gg \omega_P(O(\frac{\log(|L|/\delta)}{n}))$, then $\theta_i$ would lose to $\mu$ and have a farther loss from it than $\mu$ has.

This warmup shows promise, but does not actually resolve the task where we are not given such a list. An initial idea is to use the first $n/2$ samples as our list, and then estimate from the latter $n/2$ samples. This is close to working, but does not satisfy the property that the list contains exactly the true mean. Luckily, the Le Cam-Birg\'e's pairwise comparison estimator (exposited in Section 32.2.2 \cite{polyanskiy2025information}; see also \cite{le2012asymptotic,van2002statistical,birge1983approximation}) is designed to handle such a setting: it is the aforementioned comparison procedure, but with a more robust pairwise test subroutine than the naive likelihood ratio. 

Hence, we will first conclude that with high probability, one of the first $n/2$ samples $X_i$ satisfies $\dhsq(P_\mu,P_{X_i}) \le O(\frac{\log(1/\delta)}{n})$. Then, we leverage a procedure that is essentially the same as the Le Cam-Birg\'e's pairwise comparison estimator. We give a self-contained treatment, and the only main difference is that we choose to use a different subroutine for pairwise comparisons. The pairwise test subroutine of Birg\'e \cite{birge2013robust} (see Theorem 32.8 and Remark 32.2 in \cite{polyanskiy2025information} for discussion) would suffice for our theorem statement, although our different pairwise test enables a running time of $\tilde{O}(n^{3/2})$ instead of $\tilde{O}(n^{2})$ (see remarks at end of the section), and is also rather elementary. For our new pairwise comparison test, we realize that $P_\mu$ and $P_{X_i}$ cannot be well-distinguished from only $\ll \frac{n}{\log(1/\delta)}$ samples. This means that likelihood tests between $P_\mu$ and some $P_{\theta_j}$ from $\ll \frac{n}{\log(1/\delta)}$ samples must perform similarly to likelihood tests between $P_{X_i}$ and $P_{\theta_j}$ by data processing inequality. Meaning, if we use \textit{purposefully underpowered} tests (i.e. do not use all of our samples) for the pairwise comparisons, then they will gain our preferred guarantees. Accordingly, we employ an approach where we use the first half of samples to get a candidate list, and then use the Le Cam-Birg\'e's pairwise comparison estimator with our purposefully underpowered tests (followed by a boosting step). We prove it nearly attains the two-point testing rate:

\localg*
\begin{proof}
    We remark that the condition of $\sqrt{n} \ge 6 \log(2/\delta)$ is semi-arbitrary, but our proof does require at least some bound on $\delta$ in relation to $n$. We also note that it is valid to argue with statements like ``for $n$ larger than a sufficiently large constant'', because this can be enforced by setting $\cdist$ large to enforce $\cdist \cdot \frac{\log(n/\delta)}{n} > 1$ for small $n$, for which the theorem is vacuous.
    
    The algorithm will begin by using the $n/2$ samples as candidates. Our hope is that at least one of these candidates $X_i$ is sufficiently close to $\mu$ such that $\dhsq(P_\mu,P_{X_i}) = O(\frac{\log(1/\delta)}{n})$. We show a result that $\dhsq(P,P_\Delta)$ lower bounds the probability of samples within $[-\Delta,+\Delta]$:

\begin{lemma}\label{lemma:hel-to-mass}
    Let $P$ be a unimodal distribution with location $0$, and let $P_{\Delta}$ be the distribution shifted by $\Delta$. Then, $\dhsq(P,P_{\Delta}) \le \int_{-\Delta}^{\Delta} p(x) \, dx$
\end{lemma}
\begin{proof}
    \begin{align*}
        & \dhsq(P,P_\Delta) \triangleq \frac{1}{2} \int (\sqrt{p(x)} - \sqrt{p(x-\Delta)})^2 \\
        & = \frac{1}{2} \left( \int_{-\infty}^0 (\sqrt{p(x)} - \sqrt{p(x-\Delta)})^2 + \int_{\Delta}^\infty (\sqrt{p(x)} - \sqrt{p(x-\Delta)})^2  + \int_{0}^\Delta (\sqrt{p(x)} - \sqrt{p(x-\Delta)})^2 \right) \\
        & \le \frac{1}{2} \int_{-\infty}^0 (\sqrt{p(x)} - \sqrt{p(x-\Delta)})^2 + \frac{1}{2} \int_{\Delta}^\infty (\sqrt{p(x)} - \sqrt{p(x-\Delta)})^2 + \frac{1}{2} \int_{-\Delta}^\Delta p(x)\\
        & = \frac{1}{2}\left( \int_{-\infty}^0 p(x) + \int_{-\infty}^{-\Delta} p(x) - 2\int_{-\infty}^0 \sqrt{p(x)}\sqrt{p(x-\Delta)} \right) \nonumber \\
        & + \frac{1}{2} \left( \int_0^\infty p(x) + \int_\Delta^\infty p(x) - 2\int_0^\infty \sqrt{p(x)}\sqrt{p(x+\Delta)}\right) + \frac{1}{2} \int_{-\Delta}^\Delta p(x)\\
        & \le \frac{1}{2}\left( \int_{-\infty}^0 p(x) + \int_{-\infty}^{-\Delta} p(x) - 2\int_{-\infty}^0 p(x-\Delta) \right) \nonumber \rtag{using that $P$ is unimodal}\\
        & + \frac{1}{2} \left( \int_0^\infty p(x) + \int_\Delta^\infty p(x) - 2\int_0^\infty p(x+\Delta)\right) + \frac{1}{2} \int_{-\Delta}^\Delta p(x) \label{step:use-unimodal}\\
        & = \int_{-\Delta}^{\Delta} p(x) dx \quad\qedhere
    \end{align*}
\end{proof}

This lets us conclude that with high probability, one of the first $n/2$ samples will be close to $\mu$:

\begin{corollary}\label{cor:cand-good}
    Let $\Delta_1 \ge 0$ be the smallest value such that:
    \begin{equation*}
        \int_{\mu - \Delta_1}^{\mu + \Delta_1} p(x - \mu) \diff x \ge \frac{2}{\log(e)} \cdot \frac{\log(2/\delta)}{n}
    \end{equation*}
    Then, with probability at least $1-\delta/2$, one of the first $n/2$ samples will have value $X_i \in [\mu-\Delta_1,\mu+\Delta_1]$. Moreover, for such an $X_i$ it holds that:
    \begin{equation*}
        \dhsq(P_\mu,P_{X_i}) \le \frac{2}{\log(e)} \cdot \frac{\log(2/\delta)}{n}
    \end{equation*}
\end{corollary}
\begin{proof}
    The probability of none of the first $n/2$ samples being in this range is at most:
    \begin{align*}
        & \left(1- \frac{2}{\log(e)} \cdot \frac{\log(2/\delta)}{n}\right)^{n/2} \\
        & = \left(1 - \frac{(2/\log(e)) \cdot \log(2/\delta)}{n}\right)^{\frac{n}{(2/\log(e))\cdot \log(2/\delta)} \cdot \frac{n/2}{n/((2/\log(e)) \cdot \log(2/\delta))}} \\
        & \le \left(\frac{1}{e}\right)^{\ln(2/\delta)} = \delta/2
    \end{align*}
    Additionally, \cref{lemma:hel-to-mass} immediately implies that for any $X_i \in [\mu - \Delta_1,\mu + \Delta_1]$ it holds that $\dhsq(P_\mu,P_{X_i}) \le \frac{2}{\log(e)} \cdot \frac{\log(2/\delta)}{n}$.
\end{proof}

Assuming this event holds, let $X_{i^*}$ be an arbitrary one of the desired samples. With the remaining $n/2$ samples we hope to use likelihood tests of size $n_{\textrm{test}} \triangleq \lfloor C_{\textrm{test}} \cdot \frac{n}{\log(n/\delta)} \rfloor$ for a later-chosen $0<C_{\textrm{test}}<1$. We will show that $P_\mu$ and $P_{X_{i^*}}$ have small total variation distance over $n_{\textrm{test}}$ samples, and then show how this implies likelihood tests with $P_{X_{i^*}}$ will perform well.

\begin{lemma}\label{lemma:underpowered-small-tv}
    There exists a constant $0<C_{\textrm{test}}<1$ such that if $n_{\textrm{test}} \triangleq \lfloor C_{\textrm{test}} \cdot \frac{n}{\log(n/\delta)} \rfloor$ then $\dtv(P_{\mu}^{\otimes n_{\textrm{test}}},P_{X_{i^*}}^{\otimes n_{\textrm{test}}}) \le 0.01$.
\end{lemma}
\begin{proof}
    \begin{align*}
        & \dtv(P_{\mu}^{\otimes n_{\textrm{test}}},P_{X_{i^*}}^{\otimes n_{\textrm{test}}}) \\
        & \le \sqrt{2} \cdot \sqrt{\dhsq(P_{\mu}^{\otimes n_{\textrm{test}}},P_{X_{i^*}}^{\otimes n_{\textrm{test}}})} \\
        & = \sqrt{2} \cdot \sqrt{1 - (1-\dhsq(P_\mu,P_{X_{i^*}}))^{n_{\textrm{test}}}} \\
        & = \sqrt{2} \cdot \sqrt{1 - (1-\dhsq(P_\mu,P_{X_{i^*}}))^{(1/\dhsq(P_\mu,P_{X_{i^*}})) \cdot (n_{\textrm{test}} \cdot \dhsq(P_\mu,P_{X_{i^*}}))}}\intertext{We will assume $\dhsq(P_\mu,P_{X_{i^*}}) \le \frac{1}{4}$ to imply $(1-\dhsq(P_\mu,P_{X_{i^*}}))^{\dhsq(P_\mu,P_{X_{i^*}})} \ge 0.3$. This assumption holds if $\frac{2}{\log(e)} \cdot \frac{\log(2/\delta)}{n} \le \frac{1}{4}$, which is implied by $n \ge 6 \log(2/\delta)$: }
        & \le \sqrt{2} \cdot \sqrt{1 - 0.3^{n_{\textrm{test}} \cdot \dhsq(P_\mu,P_{X_{i^*}})}}\\
        & \le \sqrt{2} \cdot \sqrt{1 - 0.3^{C_{\textrm{test}} \cdot \frac{2}{\log(e)}}} \le 0.01 
    \end{align*}
    for sufficiently small $0<C_{\textrm{test}}<1$.
\end{proof}

We will use the Neyman-Pearson lemma:
\begin{fact}\label{fact:mu-test-good}
    Consider the task of testing between two distributions $P_1,P_2$. Let $\est_{\textrm{likelihood}}^{P_1,P_2}(X)$ to be the estimator that outputs $1$ if $P_1(X) > P_2(X)$ and $2$ otherwise. Then:
    \begin{equation*}
        \min_{i \in \{1,2\}} \Pr_{X \sim P_i}[\est_{\textrm{likelihood}}^{P_1,P_2}(X) = i] \ge \dtv(P_1,P_2) 
    \end{equation*}
\end{fact}

Now, we show that any sufficiently bad $X_j$ will most likely fail a likelihood test against $X_{i^*}$:
\begin{lemma}
    There exists a constant $\cdist \ge \frac{2}{\log(e)}$ (that is only a function of $C_{\textrm{test}}$) such that if
    \begin{equation*}
        \Delta^* \triangleq \omega_P \left( \cdist \cdot \frac{\log(n/\delta)}{n}\right)
    \end{equation*}

    then, for any $\theta \notin [\mu-2\Delta^*,\mu+2\Delta^*]$ it holds that:
    \begin{equation*}
        \Pr_{X \sim P_\mu^{\otimes n_{\textrm{test}}}}[\est_{\textrm{likelihood}}^{P_{X_{i^*}}^{\otimes n_{\textrm{test}}},P_{\theta}^{\otimes n_{\textrm{test}}}}(X) = 1] \ge 0.98
    \end{equation*}
\end{lemma}
\begin{proof}
    We remark that the constraint $\cdist \ge \frac{2}{\log(e)}$ was chosen to imply that $\Delta^* \ge \Delta_1$ (as long as $n \ge 2$) for convenience. 

    \begin{align*}
        & \Pr_{X \sim P_\mu^{\otimes n_{\textrm{test}}}}[\est_{\textrm{likelihood}}^{P_{X_{i^*}}^{\otimes n_{\textrm{test}}},P_{\theta}^{\otimes n_{\textrm{test}}}}(X) = 1] \\
        & \ge \Pr_{X \sim P_{X_{i^*}}^{\otimes n_{\textrm{test}}}}[\est_{\textrm{likelihood}}^{P_{X_{i^*}}^{\otimes n_{\textrm{test}}},P_{\theta}^{\otimes n_{\textrm{test}}}}(X) = 1] - \dtv\left(\est_{\textrm{likelihood}}^{P_{X_{i^*}}^{\otimes n_{\textrm{test}}},P_{\theta}^{\otimes n_{\textrm{test}}}}(X \sim P_\mu^{\otimes n_{\textrm{test}}}), \est_{\textrm{likelihood}}^{P_{X_{i^*}}^{\otimes n_{\textrm{test}}},P_{\theta}^{\otimes n_{\textrm{test}}}}(X \sim P_{X_{i^*}}^{\otimes n_{\textrm{test}}})\right) \\
        & \ge \Pr_{X \sim P_{X_{i^*}}^{\otimes n_{\textrm{test}}}}[\est_{\textrm{likelihood}}^{P_{X_{i^*}}^{\otimes n_{\textrm{test}}},P_{\theta}^{\otimes n_{\textrm{test}}}}(X) = 1] - \dtv(P_\mu^{\otimes n_{\textrm{test}}},P_{X_{i^*}}^{\otimes n_{\textrm{test}}}) \rtag{using data processing inequality}\\
        & \ge \Pr_{X \sim P_{X_{i^*}}^{\otimes n_{\textrm{test}}}}[\est_{\textrm{likelihood}}^{P_{X_{i^*}}^{\otimes n_{\textrm{test}}},P_{\theta}^{\otimes n_{\textrm{test}}}}(X) = 1] - 0.01 \rtag{using \cref{lemma:underpowered-small-tv}} \\
        & \ge \dtv(P_{X_{i^*}}^{\otimes n_{\textrm{test}}},P_{\theta}^{\otimes n_{\textrm{test}}}) - 0.01 \rtag{using \cref{fact:mu-test-good}}\\
        & \ge \dhsq(P_{X_{i^*}}^{\otimes n_{\textrm{test}}},P_{\theta}^{\otimes n_{\textrm{test}}}) - 0.01 \\
        & = 1 - (1 - \dhsq(P_{X_{i^*}},P_{\theta}))^{n_{\textrm{test}}} - 0.01 \\
        & \ge 1 - e^{-n_{\textrm{test}} \cdot \dhsq(P_{X_{i^*}},P_{\theta})} - 0.01 \intertext{Using that $|X_{i^*} - \theta| > 2 \Delta^* - \Delta_1 \ge \Delta^*$ implies $\dhsq(P_{X_{i^*}}, P_\theta) > \cdist \cdot \frac{\log(n/\delta)}{n}$:}
        & \ge 0.99 - e^{-\lfloor C_{\textrm{test}} \cdot \frac{n}{\log(n/\delta)} \rfloor \cdot \cdist \cdot \frac{\log(n/\delta)}{n}}\\
        & \ge 0.99 - e^{-\cdist \cdot C_{\textrm{test}}/2} \ge 0.98 \rtag{with sufficiently large $n$}
    \end{align*}
    for sufficiently large $\cdist$.
\end{proof}

We are now ready to argue that with probability $1-\delta$, $X_{i^*}$ passes all likelihood tests against $\theta \notin [\mu-2\Delta^*,\mu+2\Delta^*]$ when we take the majority answer of $k_{\textrm{num-tests}} \triangleq \lfloor \frac{n/2}{\lfloor C_{\textrm{test}} \cdot \frac{n}{\log(n/\delta)} \rfloor} \rfloor$ tests:

\begin{claim}\label{claim:majority-good}
    Consider for each pair of the first $n/2$ samples we take the majority outcome of $k_{\textrm{num-tests}}$ likelihood tests. Then, with probability at least $1-\delta/2$, $X_{i^*}$ has a strict majority against all tested $\theta$ where $\theta \notin [\mu - 2 \Delta^*, \mu + 2 \Delta^*]$. 
\end{claim}
\begin{proof}
    Let $S$ be the set of the first $n/2$ samples that are not in $[\mu - 2 \Delta^*, \mu + 2\Delta^*]$. Then:
    \begin{align*}
        & \Pr_{\textrm{$k_{\textrm{num-tests}}$ groups of $n_{\textrm{test}}$-sized tests}}[X_{i^*} \textrm{ does not have strict majority over all } S] \\
        & \le \frac{n}{2} \cdot \max_{\theta \notin [\mu-2\Delta^*, \mu + 2 \Delta^*]}\Pr_{\textrm{$k_{\textrm{num-tests}}$ groups of $n_{\textrm{test}}$-sized tests}}[X_{i^*} \textrm{ does not have strict majority over } \theta] \\
        & \le \frac{n}{2} \cdot \Pr\left[\frac{1}{k_{\textrm{num-tests}}} \cdot \sum_{j=1}^{k_{\textrm{num-tests}}} \bern(0.98) \le 0.5\right]\\
        & \le \frac{n}{2} \cdot 2 \cdot \exp\left( - \frac{2 \cdot (0.4 \cdot k_{\textrm{num-tests}})^2}{k_{\textrm{num-tests}}} \right) \\
        & = n \cdot \exp\left( - 0.32 \cdot k_{\textrm{num-tests}} \right)\\
        & \le n \cdot \exp\left( - 0.32 \cdot \left\lfloor \frac{\log(n/\delta)}{2 C_{\textrm{test}}} \right\rfloor \right) \le \delta/2
    \end{align*}
    for sufficiently small $C_{\textrm{test}}$.
\end{proof}

Wrapping up, from our initial $n/2$ samples, our algorithm will choose one sample $X_{j'}$ as our estimate. If there is an undefeated $X_{j'}$ then it will choose this one. Otherwise, it will choose the $j'$ that minimizes the furthest loss:

\begin{equation}
    j' \triangleq \argmin_{j' \in \{1,\dots, n/2\}} \max_{\ell \in \{1,\dots, n/2\} \textrm{ where $X_\ell$ beats $X_{j'}$ }} |X_{j'} - X_\ell| \label{eq:jprime}
\end{equation}

\begin{claim}\label{claim:jprime-good}
    Under the event in \cref{claim:majority-good}, we conclude $X_{j'} \in [\mu - 4 \Delta^*, \mu + 4\Delta^*]$
\end{claim}
\begin{proof}

If there is an undefeated $X_{j'}$ then either $j' = i^*$ or all the first $n/2$ samples are in $[\mu-2\Delta^*,\mu+2\Delta^*]$; in either case, our desired result immediately follows. Otherwise, if no sample is undefeated, let a sample's ``radius'' be the distance from its farthest loss. By \cref{claim:majority-good}, $X_{i^*}$ will have radius at most $\Delta_1 + 2 \Delta^* \le 3 \Delta^*$. For sake of contradiction, suppose $X_{j'} \notin [\mu - 4\Delta^*, \mu + 4 \Delta^*]$. Then, $X_{i^*}$ must beat it, yet their distance is $> 3 \Delta^*$, so this is impossible. Thus, our algorithm incurs error at most $4 \Delta^*$. 
\end{proof}

In summary, our algorithm is as follows:
\begin{itemize}
    \item We use the first $n/2$ samples as a list of candidate estimates. By \cref{cor:cand-good}, we conclude that there is at least one sample $X_{i^*} \in [\mu - \Delta_1, \mu + \Delta_1]$.
    \item For sufficiently large $\cdist$ and sufficiently small $0<C_{\textrm{test}}<1$, we group the remaining $n/2$ samples into $k_{\textrm{num-tests}} \triangleq \lfloor \frac{n/2}{\lfloor C_{\textrm{test}} \cdot \frac{n}{\log(n/\delta)} \rfloor} \rfloor$ tests of size $n_{\textrm{test}} \triangleq \lfloor C_{\textrm{test}} \cdot \frac{n}{\log(n/\delta)} \rfloor$. We also define $\Delta^*$ in terms of $\cdist$.
    \item For each pair of candidate estimates, we perform the  $k_{\textrm{num-tests}}$ likelihood tests, and we say that one of the pair ``wins'' if it has strictly larger likelihood for a strict majority of the tests. By \cref{claim:majority-good}, with probability $1-\delta/2$ (conditioned on the existence of an $X_{i^*}$), $X_{i^*}$ will have a strict majority against any $X_j \notin [\mu - 2\Delta^*, \mu + 2\Delta^*]$. 
    \item We choose our estimate to be $X_{j'}$: the candidate whose furthest loss in the closest as indicated in \cref{eq:jprime} (or the undefeated candidate, if there is one). By \cref{claim:jprime-good}, this estimate will be within $[\mu - 4\Delta^*,\mu+ 4\Delta^*]$.
\end{itemize}
\end{proof}

\textbf{Remarks. } First, we informally remark that this procedure can be optimized to run in $\tilde{O}(n^{3/2})$ time (we will not focus on polylogarithmic dependence of $\log(1/\delta)$ for this remark). Since we know the density, we know the quantile of the mode. By standard concentration arguments, the index of sample $X_{i^*}$ will be within $\tilde{O}(\sqrt{n})$ of the quantile. So, we can choose to only consider the nearby $\tilde{O}(\sqrt{n})$ samples for our list. We can now precompute the likelihood over all batches for all list entries in $\tilde{O}(n^{3/2})$ time. Then, given this precomputed data, each of the $\tilde{O}((\sqrt{n})^2)=\tilde{O}(n)$ required pairwise tests can be computed in $\tilde{O}(1)$ time. This is a crucial difference from the pairwise test of Birg\'e \cite{birge2013robust} (see Theorem 32.8 and Remark 32.2 in \cite{polyanskiy2025information} for discussion), which is not obviously able to leverage precomputed data, so each pairwise test uses $\tilde{O}(n)$ time. Meaning, the new pairwise test enables a speedup from $\tilde{O}(n^2)$ to $\tilde{O}(n^{3/2})$ time. There are other ways to get this speedup; leveraging the technique of \cite{daskalakis2014faster} should yield a similar runtime.

We also remark that, if desired, we expect this same proof method should naturally extend towards an analogous positive result for mixtures of unimodal distributions (not necessarily with the same center). The itemized summary previously stated should still essentially hold. Modifying the first item of the summary, instead show that one of the first $n/2$ samples will be sufficiently close to the mode of one of the mixture components, such that using the translation that overlays the component's mode over the sample will have small Hellinger distance with the correct translation. We avoid this additional complication,  as our motivation is primarily to contrast with our negative result for symmetric, unimodal distributions in the adaptive setting.

\section{Lower Bound for Location Estimation of Symmetric Distributions}

In the asymptotic setting, symmetry is a strong enough condition to attain the Fisher information rate \cite{stone1975adaptive}. In contrast, we will show that for any number of samples $n$, there is a symmetric distribution where any estimator $\est(X)$ will incur error arbitrarily larger than the two-point testing rate (even incurring error worse than $\omega_D(C)$ for a constant $C>0$):

\loclb*
\begin{proof}
We first remark that the constants in our theorem statement are semi-arbitrary. 
Additionally, the $\omega_{D}(\frac{1}{10})$ yielded by our construction will be strictly positive, otherwise the theorem could be vacuously true. Let us begin with some intuition for constructing the distribution. Consider the uniform distribution $\unif(\mu-1,\mu+1)$: it is well-known that the optimal error for estimating $\mu$ from $n$ samples is $\Theta(\frac{1}{n})$ (by taking the midpoint of the minimum sample and the maximum sample). Now, consider modifying the uniform distribution by discretizing the domain $[\mu-1,\mu+1]$ it into $T \gg n$ equally-sized buckets, and then for a random half of the buckets we set the density to $0$, while we double the density for the other half of the buckets. Even if we are told the new modified distribution, it does not seem significantly easier to estimate its mean compared to the original uniform distribution. However, the two-point testing rate dramatically changes. For such a randomly modified distribution, consider the distance between this distribution and some translation larger than $\frac{2}{T}$. At any value $x$ in the domain, its bucket was either set to $0$ or doubled, and the translated bucket was modified independently, so there is a $\frac{1}{2}$ chance they were modified in opposite ways. Roughly, this means about half of the domain will correspond to $x$ where one translation has density $0$, and the other translation has density $1$. Since this intuits that the Hellinger distance is large for any translation larger than $\frac{2}{T}$, then we expect the two-point testing rate will be $\le \frac{2}{T} \ll \frac{1}{n}$ for most randomly modified distributions.

Our proof will aim to capture a similar intuition, where we discretize the domain into $4T$ buckets, and define a randomly modified version of the uniform distribution over these buckets that is always symmetric. $\mathcal{F}_D$ will be our family of modified distributions, and our goal will be to show that there exists a $D \in \mathcal{F}_D$ where:

\begin{enumerate}
    \item $\omega_D(\frac{1}{10}) \le \frac{1}{T}$
    \item $\min_{\est} \max_\mu \Pr_{X \sim D(x-\mu)^{\otimes n}}[|\est(X)-\mu| \ge \nu \cdot \frac{1}{T}] \ge \frac{1}{4}$
\end{enumerate}

Together, these two properties would imply our entire theorem. It appears simple to hand-design distributions where property (2) holds, but property (1) is more inconvenient (e.g. because any distribution that nearly has some small periodicity will not satisfy this property). Hence, our proof will show the existence of such a $D$ by the probabilistic method: a uniformly random $D$ sampled from $\mathcal{F}_D$ will have both properties with positive probability. For random $D$ from our family, property (1) will be simpler to prove, but property (2) will become slightly more involved.

Let us begin by defining $D_v$, a distribution parameterized by a vector $v \in \{0,1\}^T$. We can think of $D_v$ as having discretized the domain $[-1,+1]$ into $4T$ buckets, and we consider buckets in batches of $4$. Each batch $i \in \{0,\dots,T-1\}$ will consist of two adjacent buckets corresponding to $[i \cdot \frac{1}{T},i \cdot \frac{1}{T} + \frac{1}{2T})$ and $[i \cdot \frac{1}{T} + \frac{1}{2T},(i+1) \cdot \frac{1}{T})$, and the two buckets when mirrored over $0$. Depending on $v_i$, one of the two adjacent buckets will have density $1$ and the other will have density $0$, while the mirrored buckets will have the mirrored values; this will enforce that each batch contains constant probability mass, and that the distribution is symmetric. We formally define the distribution:

\begin{definition}[Modified symmetric uniform distribution]
    Let $v$ be a vector in $\{0,1\}^T$, then $D_v$ is the distribution:
     \begin{align*}
     \D_v(x) = \begin{cases} 
          v_i & |x| \in [i \cdot \frac{1}{T},i \cdot \frac{1}{T} + \frac{1}{2T}) \textrm{ for } i \in \{0,\dots,T-1\}\\
          1-v_i & |x| \in [i \cdot \frac{1}{T} + \frac{1}{2T},(i+1) \cdot \frac{1}{T}) \textrm{ for } i \in \{0,\dots,T-1\}\\
          0 & |x| \ge 1
       \end{cases}
    \end{align*}
\end{definition}

Our family $\mathcal{F}_D$ will be the collection of all $2^T$ possible $D_v$. We will first show that for a uniformly random $D_v$ from $\mathcal{F}_D$, that $\omega_{D_v}(\frac{1}{10})$ is probably small:

\begin{lemma} \label{lemma:twopoint-lookgood} There exists a universal constant $T_0 > 0$ such that for any integer $T \ge T_0$:
    \begin{equation*}
        \Pr_{D_v \sim \mathcal{F}_D}\left[\omega_{D_v}\left(\frac{1}{10}\right) \le \frac{1}{T}\right] \ge \frac{3}{4}
    \end{equation*}
\end{lemma}
\begin{proof}
    Recall how $w_{D_v}(\eps) \triangleq \sup \{|\theta| \, | \, \dhsq(D_v(x),D_v(x-\theta)) \le \eps \}$. Hence, to show our lemma it will be sufficient to show that $\dhsq(D_v(x),D_v(x-\theta)) > \frac{1}{10}$ for $|\theta| \ge \frac{1}{T}$.

    We first remark that if $|\theta| > \frac{1}{5}+\frac{1}{T}$ then $\dhsq(D_v(x),D_v(x-\theta)) > \frac{1}{10}$ for any value of $v$ and sufficiently large $T$. We may then just focus on obtaining a lower bound for:
    \begin{align*}
        & \min_{|\theta| \in [\frac{1}{T},\frac{1}{5}+\frac{1}{T}]} \dhsq(D_v(x),D_v(x-\theta)) \intertext{Note that since our distributions always have values $D_v(x)$ of $0$ or $1$, then the squared Hellinger distance is equal to the total variation distance:}
        & = \min_{|\theta| \in [\frac{1}{T},\frac{1}{5}+\frac{1}{T}]} \dtv(D_v(x),D_v(x-\theta)) \intertext{Additionally, the total variation distance interpolates linearly between the $\theta$ for the previous multiple of $\frac{1}{T}$ to the next multiple. Accordingly, the distance is at least the distance of the centering from the adjacent multiples:}
        & \ge \min_{|\theta| \in \{\frac{1}{T},\frac{2}{T},\dots, \lceil \frac{T}{5} + 1 \rceil \cdot \frac{1}{T}\}} \dtv(D_v(x),D_v(x-\theta)) 
    \end{align*}
    This shows us that it is sufficient to consider the total variation distance for a finite number of translations, which will be easier to work with:
    \begin{align*}
        & \Pr_{D_v \sim \mathcal{F}_D}\left[\omega_{D_v}\left(\frac{1}{10}\right) > \frac{1}{T}\right]\\
        & \le \Pr_{D_v \sim \mathcal{F}_D}\left[\min_{|\theta| \in \{\frac{1}{T},\frac{2}{T},\dots, \lceil \frac{T}{5} \rceil \cdot \frac{1}{T} + \frac{1}{T} \}} \dtv(D_v(x),D_v(x-\theta)) \le \frac{1}{10}\right] \intertext{By union bound:}
        & \le \left(2 \cdot \left\lceil \frac{T}{5} + 1 \right\rceil\right) \cdot \max_{|\theta| \in \{\frac{1}{T},\frac{2}{T},\dots, \lceil \frac{T}{5} \rceil \cdot \frac{1}{T} + \frac{1}{T} \}} \cdot \Pr_{D_v \sim \mathcal{F}_D}\left[\dtv(D_v(x),D_v(x-\theta)) \le \frac{1}{10}\right] \numberthis \label{step:translate-prob}
        \end{align*}
        We now bound this probability. For every point $x \in (-1,+1)$, we call $\operatorname{Batch}(x) \triangleq \lfloor |x| \cdot T \rfloor$ the batch of four buckets related to this domain point. Accordingly, the density of $D_{v}(x)$ is only affected by $v_{\operatorname{Batch}(x)}$, and the density of $D_{v}(x-\theta)$ is only affected by $v_{\operatorname{Batch}(x-\theta)}$. Our goal is to examine subsets of the domain where the density of $D_v(x)$ is determined by a disjoint set of coordinates from those that determine the density of $D_v(x-\theta)$. Then, we may hope to lower bound their total variation distance by the sum of i.i.d. random variables.

        Without loss of generality, suppose $\theta > 0$. We will choose subsets of the domain that are to the right of $\theta$, starting with $[\theta,\theta+\frac{1}{T}), [\theta + \frac{1}{T}, \theta + \frac{2}{T}),\dots, [2\theta -\frac{1}{T},2\theta)$. However, starting at $[2\theta,2\theta+\frac{1}{T})$ we observe that this batch for $D_v(x-\theta)$ is the same as the batch of $[\theta,\theta + \frac{1}{T})$ for $D_v(x)$. To avoid this issue, we will choose alternating sets of batches: including the first $\theta T$ segments of length $\frac{1}{T}$, then skipping the next $\theta T$, then including the next $\theta T$, and so on, stopping at $\theta + 1$. Throughout this process, we will include at least $\lceil T/2 \rceil$ segments of length $\frac{1}{T}$, where each segment will either deterministically contribute total variation distance of $\frac{1}{2T}$ (if the segment is not within $(-1,+1)$), or will i.i.d. contribute total variation distance of $0$ or $\frac{1}{2T}$ with equal probability. We may now conveniently bound the probability:
        \begin{align*}
            \cref{step:translate-prob} & \le \left(2 \cdot \left\lceil \frac{T}{5} + 1 \right\rceil\right) \cdot \max_{|\theta| \in \{\frac{1}{T},\frac{2}{T},\dots, \lceil \frac{T}{5} \rceil \cdot \frac{1}{T} + \frac{1}{T} \}} \cdot \Pr_{D_v \sim \mathcal{F}_D}\left[\left(\sum_{k=1}^{\lceil T/2 \rceil} \bern\left(\frac{1}{2}\right) \cdot \frac{1}{2T} \right) \le \frac{1}{10}\right] \\
            & \le \left(\frac{2T}{5} + 4\right) \cdot 2 \cdot \exp\left( \frac{-2 \cdot (1/80)^2}{\lceil \frac{T}{2} \rceil (1/2T)^2}\right)\\
            & \le  \left(\frac{2T}{5} + 4\right) \cdot \exp\left( \frac{-T}{800}\right)
        \end{align*}
        For sufficiently large $T$, this quantity is upper bounded by $\frac{1}{4}$ (or any chosen constant).
\end{proof}

\cref{lemma:twopoint-lookgood} indicated that a random $D_v \sim \mathcal{F}_D$ often has a very optimistic two-point testing lower bound. Next remains to show a minimax-style lower bound for most $D_v$ that implies this is unattainable. Let us define a packing as is typically used in techniques like Le Cam's method:

\begin{definition}
    A family of $m$ distributions $P_1,\dots, P_m$ with corresponding parameters $\Theta_1,\dots,\Theta_m$ is an $\eps$-packing of size $m$ if for all $i \ne j$ it holds that $|\theta_i-\theta_j| \ge 2\eps$.
\end{definition}

We now show a general lower bound that applies when random samples from some $P_i$ will often have some $P_j$ where the sample has at least as large of a likelihood. In other words, if samples from some distribution often look at least as likely to be from some other distribution in the packing, it will be difficult to determine which distribution samples come from. We expect this style of lower bound has appeared in many works before:

\begin{lemma}\label{lemma:other-likely}
    Consider an $\eps$-packing of size $m$: $P_1,\dots,P_m$, where each $P_i$ is a distribution supported over $\R^d$. Suppose for all $i \in [m]$ it holds that:
    \begin{equation*}
        \Pr_{X \sim P_i}\left[\left(\max_{j \ne i} P_j(x) \right) \ge P_i(x)\right] \ge \alpha.
    \end{equation*}

    Then, $\min_{\est} \max_{i \in [m]} \Pr_{X \sim P_i, \est}[|\est(X)-\theta_i| \ge \eps] \ge \alpha/2$.
\end{lemma}
\begin{proof}
    \begin{align*}
        & \min_{\est} \max_{i \in [m]} \Pr_{X \sim P_i, \est}[|\est(X)-\theta_i| \ge \eps] \\
        & \ge \frac{1}{m} \cdot \min_{\est} \sum_{i \in [m]} \Pr_{X \sim P_i, \est}[|\est(X)-\theta_i| \ge \eps] \\
        & = \frac{1}{m}   \cdot\min_{\est} \int_{\R^d} \left( \sum_{i \in [m]} \Pr_{\est}[|\est(x)-\theta_i| \ge \eps ] \cdot p_i(x) \right) \diff x \intertext{This minimum over $\est$ is attained for each value of $x$ by the estimator that estimates $\theta_{i^*}$ for $i^* \triangleq \argmax_{i^*} p_{i^*}(x)$:}
        & = \frac{1}{m}  \cdot\int_{\R^d} \left( \left( \sum_{i \in [m]} p_i(x) \right) - \max_{i^*} p_{i^*}(x) \right) \diff x \intertext{For each value of $x$, we may relate this quantity to the total probability from $p_i(x)$ for each $p_i(x)$ that is not the unique maximum:}
        & \ge \frac{1}{m}  \cdot \int_{\R^d} \left( \sum_{i \in [m]} \frac{p_i(x)}{2} \cdot \mathbbm{1}\left[\left(\max_{j \ne i} p_j(x) \right) \ge p_i(x)\right] \right) \diff x\\
        & = \frac{1}{m} \cdot \sum_{i \in [m]} \frac{1}{2} \cdot \Pr_{X \sim P_i}\left[\left( \max_{j \ne i} p_j(x) \right) \ge p_i(x)\right]\\
        & \ge \frac{\alpha}{2} \rtag{using the main assumption of our lemma}\quad\qedhere
    \end{align*}
    \end{proof}

Now we will show that most $D$ in $\mathcal{F}_D$ have a packing of their translations that satisfies this property. For an integer $m$ (we defer this choice until later), we choose a collection $\theta_1,\dots,\theta_m$ such that:
\begin{enumerate}
    \item For all $i \ne j$ it holds that $|\theta_i - \theta_j| \ge \frac{1}{100nm^2}$
    \item All $\theta_i$ are multiples of $\frac{1}{T}$
    \item All $|\theta_i| \le \frac{1}{100nm}$
    \item $m \ge 2^n \ln(100m)$ \label{item:a}
    \item $\frac{n^2 m}{T} \le \frac{1}{100m}$ \label{item:b}
    \item $\frac{nm^2}{T} \le \frac{1}{50m}$
\end{enumerate}
Later, (1) will dictate how good of a lower bound we can get from this packing, while the remaining properties will enable that it is not possible to estimate which is the true $\theta_i$. We now set parameters to satisfy these properties. Setting $m = 9 \cdot 2^n \cdot n$ will satisfy (4), using $n \ge 1$. Setting $T \ge \lceil 100nm^3 \rceil = \lceil 100 \cdot 9^3 \cdot 2^{3n} \cdot n^4 \rceil$ will satisfy (5) and (6). Finally, if we seek to pack $\theta_i$ such that all $|\theta_i| \le \frac{1}{100nm}$ and all $\theta_i$ are multiples of $\frac{1}{T}$, then there exists such a packing where all $i\ne j$ satisfy $|\theta_i - \theta_j| \ge \frac{2/(100nm)}{m} - \frac{1}{T} \ge \frac{1}{50nm^2} - \frac{1}{100n^2m^2} \ge \frac{1}{100nm^2}$, satisfying (1).

With this packing of $\theta_i$ in hand, for a given $D_v$ we define translations $D_{v,1},\dots,D_{v,m}$, where $D_{v,i}(x) \triangleq D_v(x - \theta_i)$. We prove the crucial property required to use the probabilistic method to invoke \cref{lemma:other-likely}:

\begin{lemma}\label{lemma:rand-has-overlap}
    $\Pr_{D_v \sim \mathcal{F}_D}\left[\min_i \Pr_{X \sim D_{v,i}^{\otimes n}}\left[(\max_{j \ne i} D_{v,j}^{\otimes n}(x) ) \ge D_{v,i}^{\otimes n}(x)\right] \ge \frac{1}{2}\right] \ge \frac{9}{10}$
\end{lemma}
\begin{proof}
    Note that the constant $\frac{1}{2}$ in the lemma statement could be an arbitrary constant in $(0,1)$. Let us focus first on showing this claim for a particular $i$, instead of the minimum $i$:

    \begin{claim}\label{claim:each-i}
        For any $i \in [m]$:
        \begin{equation*}
            \Pr_{D_v \sim \mathcal{F}_D}\left[\Pr_{X \sim D_{v,i}^{\otimes n}}\left[(\max_{j \ne i} D_{v,j}^{\otimes n}(x) ) \ge D_{v,i}^{\otimes n}(x)\right] \ge \frac{1}{2}\right] \ge 1 - \frac{1}{10m}
        \end{equation*}
    \end{claim}
    \begin{proof}
        We will first try relate the desired quantity (a probability over distributions $\mathcal{F} \sim D_v$ that samples from a translation will have some property) to a more natural quantity (the probability of an event jointly over $D_v$ and its samples from a translation $X \sim D_{v,i}^{\otimes n}$):
        \begin{align*}
            & \Pr_{D_v \sim \mathcal{F}_D}\left[\Pr_{X \sim D_{v,i}^{\otimes n}}\left[\left(\max_{j \ne i} D_{v,j}^{\otimes n}(x) \right) \ge D_{v,i}^{\otimes n}(x)\right] \ge \frac{1}{2}\right]\\
            & = 1 - \Ex_{D_v \sim \mathcal{F}_D}\left[\mathbbm{1}\left[ \Pr_{X \sim D_{v,i}^{\otimes n}}\left[\left(\max_{j \ne i} D_{v,j}^{\otimes n}(x) \right) \ge D_{v,i}^{\otimes n}(x)\right] < \frac{1}{2} \right]\right]\\
            & \ge 1 - \Ex_{D_v \sim \mathcal{F}_D}\left[2 \cdot \left( 1 -  \Pr_{X \sim D_{v,i}^{\otimes n}}\left[\left(\max_{j \ne i} D_{v,j}^{\otimes n}(x) \right) \ge D_{v,i}^{\otimes n}(x)\right] \right) \right]\\
            & = 2 \cdot \Pr_{D_v \sim \mathcal{F}_D, X \sim D_{v,i}^{\otimes n}}\left[\left(\max_{j \ne i} D_{v,j}^{\otimes n}(x) \right) \ge D_{v,i}^{\otimes n}(x)\right] - 1 \numberthis \label{step:pr-now}
        \end{align*}
        Now we are analyzing the probability that if we take a random distribution $D_v \sim \mathcal{F}_D$, and we sample from one translation of this distribution $X \sim D_{v,i}^{\otimes n}$, the probability that our sample is at least equally likely to be from some other translation $D_{v,j}^{\otimes n}$ where $i\ne j$.\\
        Our main intuition will be that for $T \gg n,m$, the realization of $D_{v,i}^{\otimes n}$ for most samples will almost be independent of $D_{v,j}^{\otimes n}$ for every $i \ne j$, in the sense that for a sample $x \sim D_{v,i}$ it is only necessary to realize one coordinate of $v$ which may not be the coordinate relevant to the other translation. Moreover, if they were truly independent, then the probability of some $D_{v,j}^{\otimes n}$ also being supported on all the samples is $2^{-n}$, so if $m \gg 2^n$ we might expect our desired property to hold.\\
        Let us try formalize this event of independence: $\mathcal{E}$. For every point $x \in (-1,+1)$, we call $\operatorname{Batch}(x) \triangleq \lfloor |x| \cdot T \rfloor$ the batch of four buckets related to this domain point. Equivalently, the density of $D_{v,i}(x)$ is only affected by $v_{\operatorname{Batch}(x-\theta_i)}$. We refer to $\mathcal{E}$ as the event that for all $a \in [n]$ and $b \in [m]$: (i) all $x_a \in (\theta_b-1,\theta_b+1)$, and (ii) all values of $\operatorname{Batch}(x_a-\theta_b)$ are $nm$ distinct values. For large $T$, we expect $\mathcal{E}$ to almost always occur, so it should not be too lossy to focus on the occurrences of our event that also have $\mathcal{E}$:
        \begin{align*}
            \cref{step:pr-now}& \ge 2 \cdot \Pr_{D_v \sim \mathcal{F}_d, X \sim D_{v,i}^{\otimes n}}\left[ \mathcal{E} \right] \cdot \Pr_{D_v \sim \mathcal{F}_d, X \sim D_{v,i}^{\otimes n}}\left[ \left(\max_{j \ne i} D_{v,j}^{\otimes n}(x) \right) \ge D_{v,i}^{\otimes n}(x)|  \mathcal{E} \right] - 1\intertext{To lower bound the probability of $\mathcal{E}$, we consider a collection of causes why the event may fail. First, some $x_a \notin (\theta_b-1,\theta_b+1)$, has probability at most $\frac{1}{2} \cdot \max_{i,j} |\theta_i - \theta_j| \le \max_j |\theta_j|\le \frac{1}{100nm}$ for a single sample and we union bound to $\frac{1}{100m}$ over all samples. Second, some $x_a$ satisfies $\operatorname{Batch}(x_a-\theta_{b_1})=\operatorname{Batch}(x_a-\theta_{b_2})$ for $b_1\ne b_2$ and $b_1,b_2\in [m]$, has probability at most $\frac{\binom{m}{2}}{2T}$ for a single sample and we union bound to $\frac{m^2n}{2T} \le \frac{1}{100m}$ over all samples. Third, some sample $x_{a_2}$ satisfies $\operatorname{Batch}(x_{a_1}-\theta_{b_1})=\operatorname{Batch}(x_{a_2}-\theta_{b_2})$ for $a_1<a_2$, with $a_1,a_2 \in [n]$ and $b_1,b_2 \in [m]$, has probability at most $\frac{nm}{T}$ for a single sample $x_{a_2}$ and we union bound to $\frac{n^2m}{T} \le \frac{1}{100m}$ over all samples. Combining these:}
            & \ge 2 \cdot \left(1 - \frac{3}{100m}\right) \cdot \Pr_{D_v \sim \mathcal{F}_d, X \sim D_{v,i}^{\otimes n}}\left[ \left(\max_{j \ne i} D_{v,j}^{\otimes n}(x) \right) \ge D_{v,i}^{\otimes n}(x)|  \mathcal{E} \right] - 1\intertext{Observe how the event $\mathcal{E}$ was not actually affected by the realization of $D_v$, it was only affected by which segments of length $\frac{1}{T}$ had samples realized within them, and these have the same joint probabilities for all $D_v \in \mathcal{F}_D$. Moreover, by definition of $\mathcal{E}$, each sample is within the potential support of each $D_{v,j}$, and all values of $\operatorname{Batch}(x_{a}-\theta_{b})$ are distinct, so the events of whether a $D_{v,j}(x) > 0$ for $i \ne j$ are exactly i.i.d. Bernoulli random variables with probability $2^{-n}$:}
            & = 2 \cdot \left(1 - \frac{3}{100m}\right) \cdot \left(1 - \left(1 - 2^{-n} \right)^m \right) - 1\\
            & \ge 2 \cdot \left(1 - \frac{3}{100m}\right) \cdot \left(1 - e^{\frac{-m}{2^n}} \right) - 1 \\
            & \ge 2 \cdot \left(1 - \frac{3}{100m}\right) \cdot \left(1 - \frac{1}{100m} \right) - 1 \rtag{using $m \ge 2^n \ln(100m)$}\\
            & \ge 1 - \frac{8}{100m} \quad\qedhere
        \end{align*}
    \end{proof}
    We may conclude our entire lemma with the $\max_{i}$ quantifier by invoking \cref{claim:each-i} over each of the $m$ translations and using a union bound.
\end{proof}
Combining \cref{lemma:other-likely,lemma:rand-has-overlap}, we obtain:
\begin{corollary}\label{cor:minimax-err}
\begin{equation*}
    \Pr_{D_v \sim \mathcal{F}_D}\left[\min_{\est} \max_{\mu} \Pr_{X \sim D_v(x-\mu)^{\otimes n}, \est}\left[|\est(X)-\theta_i| \ge \frac{1}{200nm^2}\right] \ge \frac{1}{4} \right] \ge \frac{9}{10}
\end{equation*}
\end{corollary}

Finally, by combining \cref{lemma:twopoint-lookgood,cor:minimax-err}, the probabilistic method implies existence of a $D_v$ with the following properties:
\begin{enumerate}
    \item $\omega_{D_v}(\frac{1}{10}) \le \frac{1}{T}$
    \item $\min_{\est} \max_{\mu} \Pr_{X \sim D_v(x-\mu)^{\otimes n}, \est}[|\est(X)-\mu| \ge \frac{1}{200nm^2}] \ge \frac{1}{4}$
\end{enumerate}
This immediately yields our desired result by setting $T$ to be sufficiently large, except we must also conclude $\nu \cdot \omega_{D_{n,\nu}} \left(  \frac{1}{10} \right)>0$, which follows from:

\begin{claim}
    For any $0 \le \Delta \le \frac{1}{2T}$, it holds that for all $v \in [0,\eps/2)^{\frac{1}{2\eps}}$:
    \begin{equation*}
        \dhsq(D_v(x),D_v(x+\Delta)) \le  (4T + 2) \cdot \Delta
    \end{equation*}
\end{claim}
\begin{proof}
We start by observing how since $\step_v$ is symmetric and unimodal, the Hellinger distance contribution from $[0,\infty]$ will be at least the value of each of $[-\infty,-\Delta]$, $[-\Delta,-\Delta/2]$, and $[-\Delta/2,0]$:
    \begin{align*}
        & \dhsq(D_v(x),D_v(x+\Delta)) \\
        & = \frac{1}{2} \cdot \int_{-1-\frac{1}{T}}^{1} \left(\sqrt{D_v(x)} - \sqrt{D_v(x+\Delta)}  \right)^2 \diff x \\
        & \le \frac{1}{2} \cdot \int_{-1-\frac{1}{T}}^{1} \left|D_v(x) - D_v(x+\Delta)  \right| \diff x \\
        & = \frac{1}{2} \cdot \sum_{i \in \{-1 - \frac{1}{T}, -1 - \frac{1}{2T}, -1 ,\dots, 1-\frac{1}{2T}\} } \int_{0}^{\frac{1}{2T}} \left|D_v(i+x) - D_v(i+x+\Delta)  \right| \diff x \\
        & = \frac{1}{2} \cdot \sum_{i \in \{-1 - \frac{1}{T}, -1 - \frac{1}{2T}, -1 ,\dots, 1-\frac{1}{2T}\} } \int_{\frac{1}{2T}-\Delta}^{\frac{1}{2T}} \left|D_v(i+x) - D_v(i+x+\Delta)  \right| \diff x \\
        & \le \frac{1}{2} \cdot \sum_{i \in \{-1 - \frac{1}{T}, -1 - \frac{1}{2T}, -1 ,\dots, 1-\frac{1}{2T}\} } \Delta \\
        & = (4T + 2) \cdot \Delta \quad\qedhere
    \end{align*}
\end{proof}

Hence, this concludes the proof of our theorem.
\end{proof}

\section{Discussion}

In this work, we studied the conditions under which the two-point testing rate is attainable for the tasks of location estimation and adaptive location estimation. Together, \cref{thm:fastthm,theorem:adaptive-lb,thm:loc-alg} provide an interesting perspective on the differences between adaptive and non-adaptive estimation. We now know that given knowledge of a unimodal distribution up to translation, the two-point testing rate is nearly attainable (\cref{thm:loc-alg}). Yet, this is not adaptively attainable when the distribution is symmetric and unimodal (\cref{theorem:adaptive-lb}), meaning that adaptive estimators cannot simultaneously match the performance of distribution-specific estimators for all symmetric, unimodal distributions. When the distribution is a mixture of $k$ centered/symmetric log-concave distributions (for small $k$), then \cref{thm:fastthm} surmises that adaptive estimators can again nearly match distribution-specific estimators.

For the task of entangled mean estimation, we remark that the main result of \cite{compton2024near} is not implied by the statement of our \cref{thm:fastthm}, yet by looking inside the proof we may conclude the result is recovered by \cref{algo:fast}. In particular, recall how \cref{algo:fast} succeeds depending on a condition where there is an interval witnessing large Hellinger distance between a distribution and its translation. The ``balance tests'' studied in Section 3.3 of \cite{compton2024near} would directly prove the existence of such a witnessing interval, after following the proof.

We discuss two interesting avenues for further work:

\textbf{Estimation in higher dimensions. } Our results focus entirely on the 1-dimensional setting. A similar study in higher dimensions could be very interesting. On one hand, even for the multivariate Gaussian $N(\mu, I_d)$, the two-point testing rate of $1/\sqrt{n}$ is not attainable (the optimal error is $\sqrt{d/n}$). This seems concerning for multivariate extensions. On the other hand, it seems potentially interesting to study the attainability of two-point testing rates for adaptive location estimation of certain unimodal, radially symmetric densities in $2 \le d \le O(1)$ dimensions (the Gaussian counter-example does not rule this out). For further inspiration, the earlier-discussed related task of entangled mean estimation demonstrates interesting behavior in higher dimensions. The works of \cite{chierichetti2014learning,pensia2022estimating,compton2024near} studied how the task becomes easier in higher dimensions given  radial symmetry (demonstrating how this is a very strong condition, and how techniques similar to ours do have applications in higher dimensions). The recent work of \cite{diakonikolas2025entangled} studied high-dimensional entangled mean estimation without stringent radial symmetry assumptions (instead studying bounded covariance matrices), encountering different rates and techniques. More generally, it seems interesting to understand the performance of our algorithmic techniques even when the two-point testing rate is unattainable.

\textbf{Adaptive location estimation for more general distributions. } Our main result \cref{thm:fastthm} shows a positive result for adaptive location estimation of log-concave mixtures that are symmetric around a common point.  While our negative result \cref{theorem:adaptive-lb} shows that the two-point testing rate is unattainable for symmetric, unimodal distributions, it still seems quite possible that the rate is attainable for more general distributions than the assumptions of \cref{thm:fastthm}. For example, consider a symmetric mixture of log-concave distributions,
\begin{equation*}
    p(x) = \sum_{i=1}^k \frac{w_i}{2} \cdot (p_i(x - \Delta_i) + p_i(x + \Delta_i)),
\end{equation*}
where each $p_i$ is a log-concave distribution that is symmetric around $0$. One such distribution is the Gaussian mixture $\frac{1}{2} N(\mu - \Delta, \sigma^2) + \frac{1}{2} N(\mu + \Delta,\sigma^2)$ (learning parameters of such a mixture is studied in e.g. \cite{wu2021randomly}). We remark that \cref{algo:ident} would immediately handle this generalization if the technical result of \cref{lemma:intervals} could be appropriately strengthened (the proof contains remarks about where the current method fails to generalize). As a starting point, if one made more stringent assumptions on the log-concave components, such as assuming they are Gaussian, then it seems that the result of \cref{lemma:intervals} would more easily generalize.

\section{Acknowledgements}
We would like to thank John Duchi for conversations that introduced the perspective of the Hellinger modulus of continuity. We would like to thank Tselil Schramm for helpful technical discussions and feedback.  Thank you to all the anonymous reviewers for their helpful feedback in improving the presentation of this paper. This work was supported by the National Defense Science \& Engineering Graduate (NDSEG) Fellowship Program, Tselil Schramm's NSF CAREER Grant no. 2143246, and Gregory Valiant's Simons Foundation Investigator Award and NSF award AF-2341890.

\bibliography{ref}
\end{document}